\documentclass[a4paper,11pt]{article}

\usepackage{amsmath,amsfonts,amssymb,amsthm}
\usepackage{mathtools}
\usepackage{stmaryrd}
\usepackage{latexsym}
\usepackage{cite}
\usepackage[dvips]{graphicx}
\usepackage[small,bf]{caption}
\usepackage{multirow}
\usepackage{float}
\usepackage{subfig}

\usepackage{booktabs}
\usepackage{listings}
\usepackage{xcolor}
\lstdefinelanguage{pseudo}{
frame=BT,
mathescape,
morekeywords={1.,2.,3.,4.,5.,6.,7.,8.,9.,10.,11.,12.,(a),(b),(c),(d)},
basicstyle=\normalsize \ttfamily \color{black},
showspaces=false, 
showstringspaces=false,        
showtabs=false,
keywordstyle=\bfseries,
}

\newcommand{\llbrace}{ \{\!\{ }
\newcommand{\rrbrace}{ \}\!\} }
\def \Id {\operatorname{Id}}
\def \Hdiv {H(\operatorname{div})}
\def \etal {\emph{et al.}}
\numberwithin{equation}{section}
\theoremstyle{plain}
\newtheorem{theorem}{Theorem}[section]
\theoremstyle{remark}
\newtheorem{rmrk}[theorem]{Remark}
\theoremstyle{section}
\newtheorem{defin}[theorem]{Definition}
\theoremstyle{section}
\newtheorem{lemma}[theorem]{Lemma}
\usepackage[hang,flushmargin]{footmisc}
\providecommand{\keywords}[1]{{\small{\textbf{Keywords:}} #1}} 

\title{ An equilibrated fluxes approach to the certified descent algorithm for shape optimization using conforming finite element and discontinuous Galerkin discretizations }
\author{ M. Giacomini \footnotemark[1]\textsuperscript{ \ ,}\footnotemark[2]\textsuperscript{ \ ,}\footnotemark[3] }
\date{}

\begin{document}

\maketitle

\renewcommand{\thefootnote}{\fnsymbol{footnote}}

\footnotetext[1]{
CMAP, Inria, Ecole polytechnique, CNRS, Universit\'e Paris-Saclay, 91128 Palaiseau, France.
\texttt{matteo.giacomini@polytechnique.edu} }
\footnotetext[2]{
DRI Institut Polytechnique des Sciences Avanc\'ees, 63 Boulevard de Brandebourg, 94200 Ivry-sur-Seine, France. }
\footnotetext[3]{
\textit{Current affiliation:} Laboratori de C\`alcul Num\`eric, E.T.S. de Ingenieros de Caminos, Universitat Polit\`ecnica de Catalunya -- BarcelonaTech, Jordi Girona 1, E-08034 Barcelona, Spain.
\vspace{5pt} }

\footnotetext{
\textit{
M. Giacomini is member of the DeFI team at Inria Saclay \^Ile-de-France. 
}
}

\renewcommand{\thefootnote}{\arabic{footnote}}

\begin{abstract}
The certified descent algorithm (CDA) is a gradient-based method for shape optimization which certifies that the direction computed using the shape gradient is a genuine descent direction for the objective functional under analysis.
It relies on the computation of an upper bound of the error introduced by the finite element approximation of the shape gradient.
In this paper, we present a goal-oriented error estimator which depends solely on local quantities and is fully-computable.
By means of the equilibrated fluxes approach, we construct a unified strategy valid for both conforming finite element approximations and discontinuous Galerkin discretizations.
The new variant of the CDA is tested on the inverse identification problem of electrical impedance tomography: both its ability to identify a genuine descent direction at each iteration and its reliable stopping criterion are confirmed.
\end{abstract}
%


\section{Introduction}
\label{ref:intro}

Shape optimization problems - that is optimization problems featuring shape-dependent functionals - have been successfully tackled in the literature by means of gradient-based methods.
The major problem of the existing strategies to solve shape optimization problems is represented by the choice of the stopping criterion when moving from the continuous framework to its discrete counterpart, e.g. by means of a finite element approximation.
As a matter of fact, stopping criteria based on the norm of the shape gradient may never be fulfilled if the tolerance is chosen too small with respect to the discretization. 
In order to circumvent this issue, in \cite{giacomini:hal-01201914} we proposed a strategy to solve shape optimization problems based on a certification procedure. Basic idea relies on the derivation of an upper bound of the error due to the approximation of the shape gradient to verify at each iteration that the direction computed using the discretized shape gradient is a genuine descent direction for the objective functional. 
The resulting algorithm obtained by coupling a descent method based on the shape gradient with an \emph{a posteriori} error estimator proved to automatically stop after generating a sequence of shapes that improved the value of the objective functional at each iteration.

Several works on the adaptive finite element method (AFEM) for shape optimization may be found in the literature \cite{AFEM-Maute, Alauzet, AFEM-Kikuchi}.
As a matter of fact, the idea of coupling shape optimization and \emph{a posteriori} error estimators is not new. It can be traced back to the work of Banichuk \etal \cite{MeshRef-Banichuk} and has been later extended by Morin \etal \cite{COV:8787109}: in these works, the authors split the error into a component due to the approximation of the geometry and another one related to the discretization of the boundary value problem. 
Concerning the latter, a key aspect of a \emph{good} estimator is the low computational cost associated to its derivation whence the great interest in estimators constructed using solely local quantities.
The construction of \emph{a posteriori} error estimators in the context of finite element approximations is another extensively investigated subject and we refer the reader to \cite{ainsworth2000posteriori, nochetto2009, verfurth2013} for a complete introduction to the field. 
In this work we consider a strategy - known as equilibrated fluxes approach - to derive fully-computable guaranteed \emph{a posteriori} error estimators for both conforming and discontinuous Galerkin discretizations.
Within the framework of conforming finite element, local $\Hdiv$-reconstructions leading to fully-computable upper bounds have been studied by several authors \cite{DBLP:journals/moc/DestuynderM99, doi:10.1137/S0036142903433790, Nicaise01042008, MR2373174, Vohralik2011}.
Moreover, we refer to \cite{doi:10.1137/060665993, MR2261011, MR2427189, MR2601287} for the main results obtained in recent years on equilibrated fluxes for discontinuous Galerkin formulations.

The present paper starts from the framework introduced in the aforementioned work \cite{giacomini:hal-01201914}, where we neglect the error due to the approximation of the geometry in order to focus on the component arising from the discretization of the governing equation. 
The novelty of our approach resides in the certification procedure for the descent direction, for which fully-computable \emph{a posteriori} error estimators are required. 
As a matter of fact, to the best of our knowledge all the works in the literature on AFEM for shape optimization focus on the qualitative information provided by the error estimators to drive mesh adaptation and do not exploit the quantitative information they carry to improve and automatize the overall optimization strategy. \\
To construct the required fully-computable \emph{a posteriori} error estimators, we follow the approach proposed by Ern and Vohral\'ik in \cite{doi:10.1137/130950100}: after presenting a thorough review of 
the recent developments in the field, the authors depict a unified framework for the construction and the analysis of \emph{a posteriori} error estimators based on equilibrated fluxes.
Thus, on the one hand, for the conforming finite element approximation we construct the equilibrated fluxes by introducing the mixed finite element formulation of a local boundary value problem with Neumann boundary conditions over patches of elements.
On the other hand, for the discontinuous Galerkin discretization we reconstruct the fluxes element-wise in the Raviart-Thomas finite element space  by specifying the values of the degrees of freedom using an average of the gradient of the solution at the interfaces.
We especially highlight the interest of this latter approach for the study of the problem of electrical impedance tomography (EIT). 
As a matter of fact, there has been a growing interest in recent years for a particular class of discontinuous Galerkin methods - known as symmetric weighted interior penalty discontinuous Galerkin - 
for problems featuring an inhomogeneous diffusion tensor \cite{doi:10.1137/050634736, MR2491426} as the one appearing in the EIT.

The rest of the paper is organized as follows. 
In section \ref{ref:ShapePB}, we recall the general formulation of a shape optimization problem, the so-called boundary variation algorithm and its improved version known as certified descent algorithm.
Then, we discuss the strategy to construct the \emph{a posteriori} estimator of the error in the shape gradient (Section \ref{ref:errorQoI}) and we introduce the problem of electrical impedance tomography 
(Section \ref{ref:InversePB}).
In section \ref{ref:ConformingFE} and \ref{ref:DiscontinuousGalerkin} we provide the details of the discretized formulations and the equilibrated fluxes estimators respectively for the conforming finite element and the discontinuous Galerkin approximations.
Eventually, in section \ref{ref:numerics} we present some numerical tests of the application of the CDA featuring the equilibrated fluxes estimators to the EIT problem and section \ref{ref:conclusion} summarizes our results.

\section{Gradient-based methods for shape optimization}
\label{ref:ShapePB}

In this section, we recall the abstract formulation of a shape optimization problem and a gradient-based method to solve it.
Let $\Omega \subset \mathbb{R}^d$ ($d \geq 2$) be an open domain with Lipschitz boundary $\partial\Omega$.
We introduce a separable Hilbert space $V_\Omega$ depending on $\Omega$, a continuous bilinear form $a_\Omega(\cdot,\cdot):V_\Omega \times V_\Omega \rightarrow \mathbb{R}$ and a continuous linear form $F_\Omega(\cdot)$ on $V_\Omega$.
We define the following state problem in $\Omega$: we seek $u_\Omega \in V_\Omega$ such that
\begin{equation}
a_\Omega(u_\Omega,\delta u) = F_\Omega(\delta u) 
\quad \forall \delta u \in V_\Omega .
\label{eq:state}
\end{equation}
Under the assumption that the bilinear form satisfies the inf-sup condition
$$
\adjustlimits\inf_{w \in V_\Omega} \sup_{v \in V_\Omega} 
\frac{a_\Omega(v,w)}{\|v\| \|w\|} 
= \adjustlimits\inf_{v \in V_\Omega} \sup_{w \in V_\Omega} 
\frac{a_\Omega(v,w)}{\|v\| \|w\|} > 0 
$$
problem \eqref{eq:state} has a unique solution $u_\Omega$.

Let us consider a cost functional $J(\Omega)=j(\Omega,u_\Omega)$ which depends on the domain $\Omega$ itself and on the solution $u_\Omega$ of the state equation. 
We denote the set of admissible domains in $\mathbb{R}^d$ with $\mathcal{U}_{\text{ad}}$  and we introduce the following problem for the minimization of the functional $J(\Omega)$:
\begin{equation}
\min_{\Omega \in \mathcal{U}_{\text{ad}}} J(\Omega) .
\label{eq:shapeOpt} 
\end{equation}
Hence, we seek a domain $\Omega$ that minimizes the functional $j(\Omega,u)$ under the constraint that $u$ is solution of the state equation \eqref{eq:state}. 
In the literature, \eqref{eq:shapeOpt} is called a shape optimization problem, that is a PDE-constrained optimization problem of a shape-dependent functional.

\subsection{Optimize-then-discretize: the boundary variation algorithm}
\label{ref:BVA}

This work exploits a gradient-based method for the numerical approximation of problem \eqref{eq:shapeOpt}. 
In particular, two main approaches have been proposed in the literature: the discretize-then-optimize strategy and the optimize-then-discretize one. 
The former relies on the idea of computing a discretized version of the objective functional and subsequently constructing its gradient to run the optimization procedure. 
The latter works the other way around, by first computing the gradient of the cost functional and then discretizing it for the optimization loop. 
The discretize-then-optimize strategy has two main drawbacks: on the one hand, the discretized functional may not be differentiable, thus limiting the possibility of using a gradient method; on the other hand, this approach may suffer from severe mesh dependency. 
Hence, we consider an optimize-then-discretize approach for problem \eqref{eq:shapeOpt} by studying a variant of the boundary variation algorithm (BVA) discussed in \cite{smo-AP}: this method relies on the computation of the so-called shape gradient which arises from the differentiation of the functional with respect to the shape (cf. section \ref{ref:shapeDifferentiation}).

The key aspect of the BVA is the computation of a descent direction for $J(\Omega)$, that is we seek a direction $\boldsymbol\theta$ along which the objective functional decreases.
Once a descent direction has been identified, the domain is deformed by means of a perturbation of the identity map $(\Id + \boldsymbol\theta)\Omega$.

\subsubsection{Differentiation with respect to the shape}
\label{ref:shapeDifferentiation}

Let $X \subset W^{1,\infty}(\Omega;\mathbb{R}^d)$ be a Banach space and $\boldsymbol\theta \in X$ be an admissible smooth deformation of $\Omega$. 
The cost functional $J(\Omega)$ is said to be $X$-differentiable at $\Omega \in \mathcal{U}_{\text{ad}}$ if there exists a continuous linear form $dJ(\Omega)$ on $X$ such that $\forall \boldsymbol\theta \in X$
$$
J((\Id+\boldsymbol\theta)\Omega) = J(\Omega) + \langle dJ(\Omega) , \boldsymbol\theta \rangle + 
o(\boldsymbol\theta).
$$
Several approaches are feasible to compute the shape gradient and we refer to \cite{giacomini:hal-01201914} for a brief review of the existing techniques.
In this work, we consider the material derivative approach \cite{sokolowski1992introduction}. 
Let us define a diffeomorphism $\varphi: \mathbb{R}^d \rightarrow \mathbb{R}^d$ such that every admissible set in $\mathcal{U}_{\text{ad}}$ may be written as $\Omega_{\varphi} \coloneqq \varphi(\Omega)$. \\
We introduce the Lagrangian functional, defined for every admissible open set $\Omega$ and every $u, \ p \in V_\Omega$ by
\begin{equation}
\mathcal{L}(\Omega,u,p) = j(\Omega,u) + a_\Omega(u,p) - F_\Omega(p).
\label{eq:Lagrangian}
\end{equation}
We define the following adjoint problem, in which we seek $p_\Omega \in V_\Omega$ such that
\begin{equation}
a_\Omega(\delta p,p_\Omega) + \left\langle \frac{\partial j}{\partial u}
(\Omega,u_\Omega), \delta p \right\rangle = 0 \quad \forall \delta p \in 
V_\Omega .
\label{eq:adjoint}
\end{equation}
Moreover, all functions $u_\varphi,\ p_\varphi \in V_{\Omega_\varphi}$ defined on the deformed domain $\Omega_{\varphi}$ may be mapped to the fixed reference domain $\Omega$ as follows:
\begin{align*}
& u_\varphi \coloneqq u\circ\varphi^{-1} \quad , \quad u \in V_\Omega , \\
& p_\varphi\coloneqq p\circ\varphi^{-1} \quad , \quad p \in V_\Omega .
\end{align*}
We admit that $u \mapsto u_\varphi$ is a one-to-one map between $V_\Omega$ and $V_{\Omega_\varphi}$. 
The Lagrangian \eqref{eq:Lagrangian} is said to admit a material derivative if there exists a linear form $\partial \mathcal{L}/\partial \varphi$ such that
$$
\mathcal{L}(\Omega_{\varphi},u_{\varphi},p_{\varphi}) = \mathcal{L}(\Omega,u,p) 
+ \left\langle \frac{\partial \mathcal{L}}{\partial \varphi}(\Omega,u,p) , 
\boldsymbol\theta \right\rangle + o(\boldsymbol\theta)
$$
where $\varphi=\Id+\boldsymbol\theta$.
Provided that $u_\varphi$ is differentiable with respect to $\varphi$ at $\varphi=\Id$ in $V_{\Omega_\varphi}$, from the fast derivation method of C\'ea \cite{Cea1986} we obtain the following expression for the shape derivative:
\begin{equation}
\langle dJ(\Omega),\boldsymbol\theta \rangle = \left\langle 
\frac{\partial \mathcal{L}}{\partial \varphi}(\Omega,u_\Omega,p_\Omega) , 
\boldsymbol\theta \right\rangle.
\label{eq:shapeDerDiff} 
\end{equation}
\begin{rmrk}
The application of the fast derivation method of C\'ea requires some caution. 
In \cite{PANTZ-sauts}, Pantz prove that if the solution of the state problem lacks of regularity (e.g. in the case of interface problems), the aforementioned method may lead to inaccurate expressions of the shape derivative. 
As a matter of fact, within the context of interface problems - as the electrical impedance tomography problem discussed in section \ref{ref:InversePB} - the discontinuity of the conductivity parameter is responsible for the solution of the state problem to be non-differentiable in the classical sense. 
Hence, a modified Lagrangian functional has to be introduced to correctly compute the expression of the shape derivative by exploiting the regularity of the restrictions of the solution of the state problem to the subdomains excluding the interface. 
For a detailed description of this procedure and possible alternative approaches, we refer to \cite{PANTZ-sauts, LaurainSturm2015}.
\end{rmrk}
We remark that the previously introduced shape derivative belongs to the dual space $X^*$. 
The corresponding shape gradient is obtained as its Riesz representative by introducing an appropriate scalar product on the space $X$. 
For the sake of simplicity, henceforth we consider $X$ to be a Hilbert space. 
Under this assumption, we may define the following variational problem to compute a descent direction for $J(\Omega)$: we seek $\boldsymbol\theta \in X$ such that
\begin{equation}
( \boldsymbol\theta, \boldsymbol\delta\boldsymbol\theta )_X + \langle dJ(\Omega),\boldsymbol\delta\boldsymbol\theta \rangle = 0 \quad 
\forall \boldsymbol\delta\boldsymbol\theta \in X .
\label{eq:variationalP}
\end{equation}
A key aspect in \eqref{eq:variationalP} is the choice of an appropriate scalar product. 
Several approaches have been proposed in the literature, spanning from the traction method inspired by the linear elasticity equation \cite{tractionMethod} to the classical $L^2$ and $H^1$ scalar products \cite{Dogan20073898}. 
A thorough discussion on this subject is presented in \cite{MR3582826} where the authors employ a Steklov-Poincar\'e-type metric to devise novel efficient algorithms based on shape derivatives.
For the rest of this paper, the classical $H^1$ scalar product is employed to construct the descent direction $\boldsymbol\theta$. 
This choice is motivated on the one hand by its simplicity from the point of view of the implementation and on the other hand by the additional regularity it introduces with respect to the $L^2$ scalar product, thus leading to smoother deformations of the shape as we will observe in the numerical simulations in subsections \ref{ref:1incl} and \ref{ref:2incl}.

\subsection{The certified descent algorithm}
\label{ref:CDA}

In this section, we introduce the finite element approximations of the state problem \eqref{eq:state} and adjoint problem \eqref{eq:adjoint} and we present the conditions that the discretized direction $\boldsymbol\theta^h$ has to fulfill in order to be a genuine descent direction for the functional $J(\Omega)$. \\
First, we define the bilinear form $a_\Omega^h(\cdot,\cdot)$ and the linear form $F_\Omega^h(\cdot)$ associated with the following discrete state problem: we seek $u_\Omega^h \in V_\Omega^{h,\ell}$ such that
\begin{equation}
a_\Omega^h(u_\Omega^h,\delta u^h) = F_\Omega^h(\delta u^h) 
\quad \forall \delta u^h \in V_\Omega^{h,\ell} 
\label{eq:stateDiscrete}
\end{equation}
where $V_\Omega^{h,\ell}$ is an appropriate finite element or discontinuous Galerkin approximation space featuring basis functions of degree $\ell$.
Following the same procedure, we introduce the discrete adjoint problem which consists in seeking $p_\Omega^h \in V_\Omega^{h,\ell}$ such that 
\begin{equation*}
a_\Omega^h(\delta p^h, p_\Omega^h) + \left\langle \frac{\partial j}{\partial u}
(\Omega,u_\Omega^h), \delta p^h \right\rangle = 0 \quad \forall \delta p^h \in 
V_\Omega^{h,\ell} .
\label{eq:adjointDiscreteDJ}
\end{equation*}
The details on the approximation space $V_\Omega^{h,\ell}$ will be discussed in section \ref{ref:ConformingFE} for the case of conforming finite element and in section \ref{ref:DiscontinuousGalerkin} for the 
discontinuous Galerkin approximation. \\
We may now introduce the discretized shape derivative $\langle d_h J(\Omega),\boldsymbol\delta\boldsymbol\theta \rangle$ defined as
\begin{equation}
\langle d_h J(\Omega),\boldsymbol\delta\boldsymbol\theta \rangle \coloneqq \left\langle 
\frac{\partial 
\mathcal{L}}{\partial \varphi}(\Omega,u_\Omega^h,p_\Omega^h), \boldsymbol\delta\boldsymbol\theta 
\right\rangle.
\label{eq:discreteShapeGrad}
\end{equation}
The discretized direction $\boldsymbol\theta^h \in X$ is computed as the solution of problem \eqref{eq:variationalP} substituting $d J(\Omega)$ by $d_h J(\Omega)$.
We recall the following definition:
\begin{defin}
A direction $\boldsymbol\theta$ is said to be a genuine descent direction for the functional $J(\Omega)$ if 
\begin{equation}
\langle d J(\Omega),\boldsymbol\theta \rangle < 0 .
\label{eq:descentDirection}
\end{equation}
\end{defin}
\noindent It is straightforward to observe that a direction $\boldsymbol\theta$ fulfilling \eqref{eq:descentDirection} is such that $J(\Omega)$ decreases along $\boldsymbol\theta$, that is $J((\Id + \boldsymbol\theta)\Omega) < J(\Omega)$. \\
Nevertheless, due to the numerical error introduced by the finite element discretization, even though $\langle d_h J(\Omega),\boldsymbol\theta^h \rangle < 0$, \ $\boldsymbol\theta^h$ is not necessarily a genuine descent direction for the functional $J(\Omega)$. 
In the following subsection, we introduce the notion of certified descent direction and we describe a procedure that allows to identify a genuine descent direction for $J(\Omega)$ by solving problem \eqref{eq:variationalP} with the discretized shape derivative $d_h J(\Omega)$ and subsequently accounting for the error in the shape gradient.

\subsubsection{Certification of a genuine descent direction}
\label{ref:certification}

Let us define the error $E^h$ due to the approximation of the shape gradient as follows:
\begin{equation}
E^h \coloneqq \langle dJ(\Omega) - d_h J(\Omega),\boldsymbol\theta^h \rangle.
\label{eq:errorH}
\end{equation}
From \eqref{eq:errorH}, it follows that 
\begin{equation}
\langle d J(\Omega),\boldsymbol\theta^h \rangle  = \langle d_h J(\Omega),\boldsymbol\theta^h \rangle + E^h .
\label{eq:gradH+err}
\end{equation}
As stated before, $\boldsymbol\theta^h$ is constructed starting from \eqref{eq:variationalP} and substituting the expression of the shape derivative with its discrete counterpart \eqref{eq:discreteShapeGrad}.
This results in a discretized direction $\boldsymbol\theta^h$ such that $\langle d_h J(\Omega),\boldsymbol\theta^h \rangle < 0$.
Nevertheless, in order for $\boldsymbol\theta^h$ to be a descent direction for the objective functional, condition \eqref{eq:descentDirection} has to be fulfilled thus the quantity $E^h$ in \eqref{eq:gradH+err} has to be accounted for.
Within this framework, we obtain the following condition on $\boldsymbol\theta^h$:
\begin{equation}
\langle d_h J(\Omega),\boldsymbol\theta^h \rangle + E^h < 0 .
\label{eq:condition-with-err}
\end{equation}
This condition does not imply that $\boldsymbol\theta^h$ is a genuine descent direction for $J(\Omega)$ since the quantity $E^h$ may be either positive or negative. 
In order to derive a relationship that stands independently of the sign of $E^h$ and since no a priori information on the aforementioned sign is available, we modify \eqref{eq:condition-with-err} by introducing the absolute value of the error in the shape gradient:
\begin{equation*}
\langle d_h J(\Omega),\boldsymbol\theta^h \rangle + E^h \leq \langle d_h J(\Omega),\boldsymbol\theta^h \rangle + | E^h | < 0.
\label{eq:Grad+AbsErr}
\end{equation*}
We may now introduce the following definition:
\begin{defin}
Let $\overline{E}$ be the upper bound of the error $| E^h |$ in the shape gradient. A direction $\boldsymbol\theta^h$ is said to be a certified descent direction for the functional $J(\Omega)$ if 
\begin{equation}
\langle d_h J(\Omega),\boldsymbol\theta^h \rangle + \overline{E} < 0.
\label{eq:certified}
\end{equation}
\end{defin}
\noindent The expression \emph{certified} is due to the fact that a direction constructed within this framework is verified to be a genuine descent direction for the functional $J(\Omega)$. 
As a matter of fact, it is straightforward to observe that if $\boldsymbol\theta^h$ fulfills \eqref{eq:certified}, then it verifies \eqref{eq:descentDirection} as well.
\begin{rmrk}
It is important to observe that a direction fulfilling \eqref{eq:certified} is a genuine descent direction for $J(\Omega)$, whether it is the solution of equation \eqref{eq:variationalP} or not. 
This is extremely important since the computation of the descent direction is done through the discretization of the aforementioned variational problem, that is $\boldsymbol\theta^h$ is only an approximation of the direction $\boldsymbol\theta$ solution of \eqref{eq:variationalP}.
\end{rmrk}

In \cite{1742-6596-657-1-012004,giacomini:hal-01201914}, we coupled the certification procedure to the boundary variation algorithm and we derived a new gradient-based method for shape optimization named certified descent algorithm (CDA - script \ref{scpt:shape-opt-adaptive}).
On the one hand, the computation of the upper bound of the numerical error in the shape gradient provides useful information to identify a certified descent direction thus improving the objective functional at each iteration of the optimization strategy. 
On the other hand, owing to the quantitative information encapsulated in $\overline{E}$ the CDA features a guaranteed stopping criterion for the overall optimization procedure. \\
The key aspect of this algorithm is the derivation of a fully-computable guaranteed estimator of the error in the shape gradient.
In particular, in section \ref{ref:errorQoI} we introduce an approach based on the equilibrated fluxes method which relies solely on the computation of local quantities.

\subsubsection{The CDA workflow}

At each iteration, the algorithm solves the state and adjoint problems and computes a descent direction $\boldsymbol\theta^h$.
Then, an upper bound of the numerical error in the shape derivative along the direction $\boldsymbol\theta^h$ is derived. 
If condition \eqref{eq:certified} is not fulfilled, the mesh is adapted in order to improve the error estimate.
This procedure is iterated  until the direction $\boldsymbol\theta^h$ is a certified descent direction for $J(\Omega)$. 
Once a certified descent direction has been identified, we compute a step $\mu$ via an Armijo rule and the shape of the domain is updated according to the computed perturbation of the identity $\Id + \mu \boldsymbol\theta^h$. 
Eventually, a novel stopping criterion is proposed in order to use the information embedded in the error bound $\overline{E}$ to derive a reliable condition to end the evolution of the algorithm.

A key aspect of the sketched procedure is the mesh adaptation routine that is performed if condition \eqref{eq:certified} is not fulfilled. 
In order to reduce the quantity $\overline{E}$ at each iteration, we construct an indicator based on the error in the shape gradient and we drive the mesh adaptation using the information carried by the indicator itself. 
This strategy is known as goal-oriented mesh adaptation (cf. \cite{Oden2001735}) and aims to identify the areas of the domain that are mainly responsible for the error in the target quantity and reduce it by means of a local refinement in order to limit the insertion of new degrees of freedom.
\begin{lstlisting}[language=pseudo, escapeinside={/*@}{@*/}, 
caption={The certified descent algorithm (CDA)}, 
label=scpt:shape-opt-adaptive]
Given the domain $\Omega_0$, set $\texttt{tol}>0$, $j=0$ and iterate:
1. Solve the state problem in $\Omega_j$;
2. Solve the adjoint problem in $\Omega_j$;
3. Compute a descent direction $\boldsymbol\theta^h_j$;
4. Compute an upper bound $\overline{E}$ of the numerical error $E^h$;
5. If $\langle d_hJ(\Omega_j),\boldsymbol\theta^h_j \rangle + \overline{E} \geq  
0$, refine the mesh and go to 1;
6. Identify an admissible step $\mu_j$;
7. Update the shape $\Omega_{j+1} = (\Id + \mu_j \boldsymbol\theta^h_j) 
\Omega_j$;
8. While $| \langle d_hJ(\Omega_j),\boldsymbol\theta^h_j \rangle | + 
\overline{E} 
>  \texttt{tol}$, $j=j+1$ and repeat.
\end{lstlisting}

\section{Approximation error for the shape gradient}
\label{ref:errorQoI}

In this section, we present a strategy to construct a fully-computable guaranteed upper bound $\overline{E}$ of the error $E^h$ in the approximation of the shape gradient in order to practically implement the certification procedure described in section \ref{ref:certification}.

\subsection{Bound for the approximation error of a linear functional}

Let us recall the framework described in \cite{Oden2001735} for the derivation of an estimate of the error in a bounded linear functional. 
We consider a quantity of interest $Q: V_\Omega \rightarrow \mathbb{R}$ which we aim to evaluate for the function $u_\Omega$, solution of the state problem \eqref{eq:state}. 
We introduce the solution $u_\Omega^h$ of the corresponding discretized problem \eqref{eq:stateDiscrete} and we seek an estimate of the error in the target functional $Q$:
\begin{equation*}
E_Q \coloneqq Q(u_\Omega) - Q(u_\Omega^h) = Q(u_\Omega-u_\Omega^h)
\label{eq:linFunc}
\end{equation*}
where the first equality follows from the linearity of $Q$.
We introduce an adjoint problem featuring the quantity $Q$ as right-hand side, that is we seek an influence function $r_\Omega \in V_\Omega$ such that
\begin{equation}
a_\Omega(\delta r,r_\Omega) = Q(\delta r)
\quad \forall \delta r \in V_\Omega .
\label{eq:adjointGoalOriented}
\end{equation}
We remark that problem \eqref{eq:adjointGoalOriented} is well-posed owing to the Lax-Milgram theorem.
The approximation of \eqref{eq:adjointGoalOriented} is obtained following the framework introduced in section \ref{ref:CDA} for the state problem.
It is well-known in the literature (cf. e.g. \cite{MR2899560}) that in order to retrieve a sharp upper bound of the error in a quantity of interest, a higher-order approximation has to be employed for the solution of the adjoint problem. 
Let $a_\Omega^h(\cdot,\cdot)$ be the discrete bilinear form and $V_\Omega^{h,m}$ the space of finite element (respectively discontinuous Galerkin) functions of degree $m, \ m > \ell$. 
We seek $r_\Omega^h \in V_\Omega^{h,m}$ such that
\begin{equation}
a_\Omega^h(\delta r^h, r_\Omega^h) = Q(\delta r^h) 
\quad \forall \delta r^h \in V_\Omega^{h,m} .
\label{eq:adjointDiscrete}
\end{equation}
From \eqref{eq:adjointGoalOriented} and \eqref{eq:state} it is straightforward to observe that
\begin{equation*}
F_\Omega(r_\Omega) = a_\Omega(u_\Omega,r_\Omega) = Q(u_\Omega) .
\label{eq:primalDual}
\end{equation*}
Thus, the following relationship between the error in the approximation of the target functional and the solutions of the state and adjoint problems is derived:
\begin{equation}
E_Q = Q(u_\Omega) - Q(u_\Omega^h) = F_\Omega(r_\Omega) - a_\Omega^h(u_\Omega^h,r_\Omega^h) = F_\Omega(r_\Omega) - a_\Omega^h(u_\Omega^h,r_\Omega)  
\label{eq:errGOAL}
\end{equation}
where the first equality follows from the approximation \eqref{eq:adjointDiscrete} of the adjoint problem \eqref{eq:adjointGoalOriented} whereas the justification of the last one exploits different properties when dealing with conforming finite element or discontinuous Galerkin approximations. \\
For the case of conforming finite element approximations, we have that $a_\Omega^h(\cdot,\cdot) = a_\Omega(\cdot,\cdot)$ and $a_\Omega(\delta r^h,r_\Omega) = Q(\delta r^h)$ stands for all $\delta r^h \in V_\Omega^{h,m}$ owing to \eqref{eq:adjointGoalOriented} and the fact that the discretization space $V_\Omega^{h,m}$ is a subspace of $V_\Omega$.
Thus the last equality in \eqref{eq:errGOAL} reduces to the classical expression of the residue of the state equation applied to the function $r_\Omega$:
$$
R_\Omega^u(r_\Omega) \coloneqq F_\Omega(r_\Omega) - a_\Omega(u_\Omega^h,r_\Omega) .
$$
On the contrary, when dealing with discontinuous Galerkin formulations, the expression of the discrete bilinear form also features the terms associated with the jumps of the discontinuous functions and it cannot be identified with its continuous version. 
Within this framework, the last equality in \eqref{eq:errGOAL} stands if the numerical method used to discretize the adjoint problem is consistent, that is if $a_\Omega^h(\delta r^h, r_\Omega) = Q(\delta r^h) \ \forall \delta r^h \in V_\Omega^{h,m}$.
The adjoint consistency is equivalent to the usual Galerkin orthogonality property in finite element and we refer to \cite{MR2361907} for more details on its role in the construction of discretizations of optimal order in terms of target functionals.
In section \ref{ref:DiscontinuousGalerkin}, we provide some details on the discontinuous Galerkin strategy chosen for the test case of electrical impedance tomography and we refer to \cite{MR2882148} for a complete analysis of the numerical scheme.

\subsection{Goal-oriented estimate of the error in the shape gradient}
\label{ref:GOAL}

In order to apply the technique described in the previous section to the estimate of the error in the shape gradient, we need to extend the previously introduced framework to the case of non-linear quantities of interest. 
We rely on the approach in \cite{2006-nonlin}, by performing a linearization of the target functional that leads to the introduction of the following linearized error $\widetilde{E}^h$:
\begin{equation}
\begin{aligned}
E^h = & \left\langle \frac{\partial \mathcal{L}}{\partial \varphi}
(\Omega,u_\Omega,p_\Omega) - \frac{\partial\mathcal{L}}{\partial\varphi}
(\Omega,u_\Omega^h,p_\Omega^h), \boldsymbol\theta^h \right\rangle \\
 & \simeq \frac{\partial^2 \mathcal{L}}{\partial \varphi \partial u}
 (\Omega,u_\Omega^h,p_\Omega^h)[\boldsymbol\theta^h,u_\Omega-u_\Omega^h] 
+ \frac{\partial^2 \mathcal{L}}{\partial \varphi \partial p}
(\Omega,u_\Omega^h,p_\Omega^h)[\boldsymbol\theta^h,p_\Omega-p_\Omega^h] \eqqcolon 
\widetilde{E}^h.
\end{aligned}
\label{eq:errDJ}
\end{equation}
For the rest of this paper, we neglect the linearization error introduced in \eqref{eq:errDJ} and we construct an estimator to evaluate the error in the shape gradient by considering solely the information in $\widetilde{E}^h$. 
More precisely, the quantities on the second line of \eqref{eq:errDJ} will serve as linear functionals to develop the goal-oriented estimate sketched in the previous subsection.
In particular, we introduce two adjoint problems associated with the terms on the right-hand side of \eqref{eq:errDJ} and we seek $r_\Omega,s_\Omega \in V_\Omega$ such that
\begin{equation}
\begin{aligned}  
& a_\Omega(\delta r,r_\Omega) = \frac{\partial^2 \mathcal{L}}
{\partial\varphi \partial u}(\Omega,u_\Omega^h,p_\Omega^h)[\boldsymbol\theta^h,\delta 
r] \quad \forall \delta r \in V_\Omega , \\
& a_\Omega(\delta s,s_\Omega) = \frac{\partial^2 \mathcal{L}}
{\partial\varphi \partial p}(\Omega,u_\Omega^h,p_\Omega^h)[\boldsymbol\theta^h,\delta 
s] \quad \forall \delta s \in V_\Omega .
\end{aligned}
\label{eq:adjointDJ}
\end{equation}
The problems \eqref{eq:adjointDJ} are well-posed since their right-hand sides are  linear and continuous forms on $V_\Omega$.
For the corresponding conforming finite element (respectively discontinuous Galerkin) discretizations of problems \eqref{eq:adjointDJ}, we seek $r_\Omega^h,s_\Omega^h \in V_\Omega^{h,m}$ such that
\begin{equation*}
\begin{aligned}  
& a_\Omega^h(\delta r^h,r_\Omega^h) = \frac{\partial^2 \mathcal{L}}
{\partial\varphi \partial u}(\Omega,u_\Omega^h,p_\Omega^h)[\boldsymbol\theta^h,\delta 
r^h] \quad \forall \delta r^h \in V_\Omega^{h,m} , \\
& a_\Omega^h(\delta s^h,s_\Omega^h) = \frac{\partial^2 \mathcal{L}}
{\partial\varphi \partial p}(\Omega,u_\Omega^h,p_\Omega^h)[\boldsymbol\theta^h,\delta 
s^h] \quad \forall \delta s^h \in V_\Omega^{h,m} .
\end{aligned}
\label{eq:adjointDJdiscretized}
\end{equation*}

Let us now define two quantities associated respectively with the contribution of the state error $u_\Omega-u_\Omega^h$ and the adjoint error $p_\Omega-p_\Omega^h$ to the linearized error in the shape gradient $\widetilde{E}^h$:
\begin{align}
& \widetilde{E}^h_u= \frac{\partial^2 \mathcal{L}}{\partial \varphi \partial u}
 (\Omega,u_\Omega^h,p_\Omega^h)[\boldsymbol\theta^h,u_\Omega-u_\Omega^h] = F_\Omega(r_\Omega) - a_\Omega^h(u_\Omega^h,r_\Omega) ,
 \label{eq:Eu} \\
& \widetilde{E}^h_p = \frac{\partial^2 \mathcal{L}}{\partial \varphi \partial p}
(\Omega,u_\Omega^h,p_\Omega^h)[\boldsymbol\theta^h,p_\Omega-p_\Omega^h] = - \left\langle \frac{\partial j}{\partial u}
(\Omega,u_\Omega), s_\Omega \right\rangle - a_\Omega^h(p_\Omega^h,s_\Omega) .
 \label{eq:Ep}
\end{align}
It is straightforward to observe that $\widetilde{E}^h =\widetilde{E}^h_u + \widetilde{E}^h_p$. In order to derive a practical strategy to perform the certification procedure in section \ref{ref:certification} by verifying condition \eqref{eq:certified}, we have to compute an upper bound $\overline{E}$ of the error $| E^h |$, that is $| \widetilde{E}^h |$ under the aforementioned assumption of a negligible linearization error. \\
In this paper, we propose an implicit error estimator based on the equilibrated fluxes method \cite{ainsworth2000posteriori}. 
This technique provides a fully-computable guaranteed upper bound of the error and relies solely on the computation of local quantities, namely the equilibrated fluxes.
A detailed description of the construction of the equilibrated fluxes for the state and adjoint problems and the derivation of the estimator of the error in the shape gradient starting from the quantities $\widetilde{E}^h_u$ and $\widetilde{E}^h_p$ is presented in the following sections. 
Since the expressions of both the equilibrated fluxes and the error estimator depend on the nature of the problem under analysis, in the next section we introduce the formulation of the inverse identification problem of electrical impedance tomography that will act as proof of concept for the discussed method.

\section{Electrical impedance tomography}
\label{ref:InversePB}

We consider the inverse identification problem of electrical impedance tomography in the most classical Point Electrode Model. 
We consider a shape optimization formulation for this inverse problem and we solve it by means of a version of the certified descent algorithm (cf. algorithm \ref{scpt:shape-opt-adaptive}) featuring an equilibrated fluxes strategy for the certification procedure. \\
It is well-known in the literature that the problem of electrical impedance tomography is severely ill-posed.
Moreover, classical shape optimization methods proved to be highly unsatisfactory by remaining trapped in local minima and consequently providing fairly poor reconstructions of the inclusions.
The certified descent algorithm does not aim to solve the known issues of gradient-based strategies when dealing with ill-posed problems.
Nevertheless, the interest in the EIT problem is twofold. On the one hand, the EIT is a non-trivial scalar problem that may guide towards the establishment of some properties of this new version of the certified descent algorithm. 
On the other hand, the CDA confirms the aforementioned limitations of gradient-based methods applied to inverse problems. 
In particular, the rapidly increasing number of degrees of freedom required for the certification procedure highlights the severe ill-posedness of the EIT problem.
We refer the reader interested in an overview of the methods investigated in the literature for the EIT problem to \cite{LaurainSturm2015, MR2211069, MR2407028} for shape optimization approaches, 
to \cite{Hintermuller2008, MR2886190, MR2966180} for topology optimization strategies and to \cite{MR3019471, MR2132313, holder2004electrical} for regularization techniques.

We are now ready to introduce the formulation of the electrical impedance tomography problem.
Let us consider an open domain $\mathcal{D} \subset \mathbb{R}^2$ featuring an open subdomain $\Omega \subset\subset \mathcal{D}$ such that the electrical conductivity is piecewise constant in $\Omega$ and $\mathcal{D} \setminus \Omega$:
\begin{equation*}
k_\Omega \coloneqq k_I \chi_\Omega + k_E (1-\chi_\Omega)
\label{eq:conductivity}
\end{equation*}
where $k_I$ and $k_E$ are two positive parameters and $\chi_\Omega$ is the characteristic function of $\Omega$. 
The goal is to identify the location and the shape of the inclusion $\Omega$ fitting non-invasive measurements $g$ and $U_D$ respectively of the flux and the potential taken on the external boundary $\partial\mathcal{D}$. 
In order to solve this problem, we introduce two boundary value problems (Section \ref{ref:state}) and owing to \cite{CPA:CPA3160400605} we define a minimization problem for an objective functional inspired by the Kohn-Vogelius one (Section \ref{ref:shapeOpt}).

The inverse problem of electrical impedance tomography - also known as Calder\'on's problem - has been extensively studied over the years. 
We suggest that the interested reader reviews \cite{Calderon1980, Cheney99electricalimpedance, 0266-5611-18-6-201} for a detailed exposition of the physical problem, its mathematical formulation 
and its numerical approximation.

\subsection{State problems}
\label{ref:state}

We consider $g \in L^2(\partial\mathcal{D})$ and $U_D \in H^{\frac{1}{2}}(\partial\mathcal{D})$ as boundary data for the flux and the potential.
Let $i=N,D$ be the subscripts associated respectively with Neumann and Dirichlet boundary conditions. We introduce the boundary value problems
\begin{align}
\left\{
\begin{aligned}
& -k_\Omega \Delta u_{\Omega,i} + u_{\Omega,i} = 0 \quad & \text{in} \ 
\mathcal{D} \setminus \partial\Omega \\
& \llbracket u_{\Omega,i} \rrbracket = 0 \quad & \text{on} \ \partial\Omega \\
& \llbracket k_\Omega \nabla u_{\Omega,i} \cdot \mathbf{n} \rrbracket = 0 \quad & 
\text{on} \ \partial\Omega
\end{aligned}
\right.
\label{eq:statePB}
\end{align}
with the following sets of boundary conditions on $\partial\mathcal{D}$:
\begin{gather}
k_E \nabla u_{\Omega,N} \cdot \mathbf{n} = g ,
\label{eq:NeumannBC} \\
u_{\Omega,D} = U_D .
\label{eq:DirichletBC}
\end{gather}

\subsection{Shape derivative of the Kohn-Vogelius functional}
\label{ref:shapeOpt}

Let us consider the following objective functional inspired by the work of Kohn and Vogelius \cite{CPA:CPA3160400605}:
\begin{equation}
J(\Omega) = \frac{1}{2}\int_\mathcal{D}{\Big( k_\Omega \left| 
\nabla(u_{\Omega,N} - u_{\Omega,D})\right|^2 + |u_{\Omega,N} - u_{\Omega,D}|^2  
\Big) d\mathbf{x}}.
\label{eq:kohn-vogelius}
\end{equation}
The problem of retrieving the inclusion $\Omega$ starting from the boundary measurements $g$ and $U_D$ may be viewed as an optimization problem in which we seek the open subset that minimizes \eqref{eq:kohn-vogelius}, $u_{\Omega,N}$ and $u_{\Omega,D}$ being the solutions of the state problems \eqref{eq:statePB} with boundary conditions given respectively by the Neumann \eqref{eq:NeumannBC} and the Dirichlet \eqref{eq:DirichletBC} data. 

As stated in section \ref{ref:shapeDifferentiation}, in order to differentiate a functional with respect to the shape, we introduce an adjoint problem for each state variable. 
Owing to the fact that the Kohn-Vogelius problem is self-adjoint, we get that $p_{\Omega,N}=u_{\Omega,N} - u_{\Omega,D}$ and $p_{\Omega,D}=0$. 
In this work, we consider a volumetric formulation of the shape derivative. 
As a matter of fact, it has been recently proved by Hiptmair and co-workers (cf. \cite{HPS_bit}) that the volumetric expression provides better numerical accuracy than its corresponding surface representation\footnote{
We refer to \cite{HPS_bit} for a detailed comparison of the volumetric and surface expressions of the shape gradient for elliptic state problems. 
In particular, in this work the authors prove that within the framework of finite element discretizations, a better numerical accuracy is achieved when using the volumetric formulation of the shape gradient.
Similar results for the case of interface problems are available in \cite{MR3467382}.}. \\
Let $\boldsymbol\theta \in W^{1,\infty}(\mathcal{D};\mathbb{R}^2)$ be an admissible deformation of the domain such that $\boldsymbol\theta=\mathbf{0} \ \text{on} \ \partial\mathcal{D}$. 
We define $\mathbf{M}(\boldsymbol\theta) = \nabla \boldsymbol\theta + \nabla \boldsymbol\theta^T - (\nabla \cdot \boldsymbol\theta) \Id$. 
By introducing the following operator
\begin{equation*}
\langle G(\Omega,u) , \boldsymbol\theta \rangle = \frac{1}{2}\int_\mathcal{D}{  
\Big(k_\Omega \mathbf{M}(\boldsymbol\theta) \nabla u \cdot \nabla u - \nabla \cdot \boldsymbol\theta \ u^2 
\Big) d\mathbf{x}} ,
\label{eq:operatorG}
\end{equation*}
the volumetric expression of the shape derivative of \eqref{eq:kohn-vogelius} reads as
\begin{equation}
\langle dJ(\Omega),\boldsymbol\theta \rangle = \langle G(\Omega,u_{\Omega,N}) -  
G(\Omega,u_{\Omega,D}) , \boldsymbol\theta \rangle.
\label{eq:volumetricKV}
\end{equation}
The interested reader may refer to \cite{PANTZ-sauts} for more details on the differentiation of the Kohn-Vogelius functional with respect to the shape and its application to the identification of discontinuities of the conductivity parameter.

\section{Conforming finite element approximation}
\label{ref:ConformingFE}

In this section, we introduce a discretization of the EIT problem based on conforming finite element functions. 
Let $\{\mathcal{T}_h\}_{h>0}$ be a family of triangulations of the domain $\mathcal{D}$ with no hanging nodes. 
Having in mind that $d=2$, we consider a mesh such that each element $T \in \mathcal{T}_h$ is a triangle and for each couple $T, T' \in \mathcal{T}_h$ such that $T \neq T'$, the intersection of the two elements is either an empty set or a vertex or an edge.
An edge $e$ is said to be an interior edge of the triangulation $\mathcal{T}_h$ if there exist two elements $T^-(e), \ T^+(e) \in \mathcal{T}_h$ such that $e = T^-(e) \cap T^+(e)$, whereas is a boundary edge if there exists $T(e) \in \mathcal{T}_h$ such that $e = T(e) \cap \partial\mathcal{D}$. 
In the former case, the unit normal vector to $e$ is denoted by $\mathbf{n}_e$ and goes from $T^-(e)$ towards $T^+(e)$. In the latter one, $\mathbf{n}$ is the classical outward normal to $\partial\mathcal{D}$.
The set of the internal edges is noted as $\mathcal{E}_h^\mathcal{I}$, the boundary edges are collected into $\mathcal{E}_h^\mathcal{B}$ and we set $\mathcal{E}_h \coloneqq \mathcal{E}_h^\mathcal{I} \cup \mathcal{E}_h^\mathcal{B}$. \\
The state and adjoint problems are solved using the following Lagrangian finite element space
$$
V_\Omega^{h,\kappa} \coloneqq \{ u^h \in \mathcal{C}^0(\overline{\mathcal{D}}) \ : \ u^h |_T \in \mathbb{P}^\kappa(T) \ \forall T \in \mathcal{T}_h \}
$$
where $\mathbb{P}^\kappa(T)$ is the set of polynomials of degree at most $\kappa$ on an element $T$, being $\kappa=\ell$ and $\kappa=m$ respectively for the state and adjoint equations.
The procedure to construct the equilibrated fluxes is performed via the solution of local subproblems defined on patches of elements using mixed finite element formulations.
A key aspect of this approach - which will be precisely detailed - is the choice of the degree of the approximating functions for both the solution of the problems and the equilibrated fluxes.

\subsection{The state problems}

Let $a_\Omega(\cdot,\cdot)$ be the bilinear form associated with problems \eqref{eq:statePB} and $F_{\Omega,i}(\cdot), \ i=N,D$ the linear forms respectively for the Neumann and the Dirichlet problem:
\begin{gather}
a_\Omega(u_{\Omega,i},\delta u) = \int_\mathcal{D}{\Big(k_\Omega \nabla u_{\Omega,i} \cdot 
\nabla \delta u + u_{\Omega,i} \delta u \Big) d\mathbf{x}} ,
\label{eq:a} \\
F_{\Omega,N}(\delta u) = \int_{\partial\mathcal{D}}{g \delta u \ ds} \quad 
\text{and} \quad F_{\Omega,D}(\delta u) = 0.
\label{eq:F}
\end{gather}
We consider $u_{\Omega,N},\ u_{\Omega,D} \in H^1(\mathcal{D})$ such that $u_{\Omega,D}=U_D \ \text{on} \ \partial \mathcal{D}$, solutions of the following Neumann and Dirichlet variational problems $\forall \delta u_N \in H^1(\mathcal{D})$ and $\forall \delta u_D \in H^1_0(\mathcal{D})$:
\begin{equation}
a_\Omega(u_{\Omega,i},\delta u_i) = F_{\Omega,i}(\delta u_i) 
\quad , \quad i=N,D.
\label{eq:stateEIT}
\end{equation}

We remark that within the framework of conforming finite element discretizations, the continuous and discrete bilinear (respectively linear) forms have the same expressions.
Hence, the corresponding discretized formulations of the state problems \eqref{eq:stateEIT} may be derived by replacing the analytical solutions $u_{\Omega,N}$ and $u_{\Omega,D}$ with their approximations $u_{\Omega,N}^h$ and $u_{\Omega,D}^h$ which belong to the space $V_\Omega^{h,\ell}$ of Lagrangian finite element functions of degree $\ell$. 
In a similar fashion, $\boldsymbol\theta^h$ is the solution of equation \eqref{eq:variationalP} computed using a vector-valued Lagrangian finite element space and substituting the expression of the 
discrete shape derivative \eqref{eq:discreteShapeGrad} to its analytical counterpart \eqref{eq:shapeDerDiff}. 
For both the state problems and the computation of the descent direction, we consider a low-order approximation respectively using $\mathbb{P}^1$ and $\mathbb{P}^1 \times \mathbb{P}^1$ Lagrangian finite element functions.

\subsection{The adjoint problems}

Let $r_{\Omega,N}$ and $r_{\Omega,D}$ be the solutions of the adjoint problems \eqref{eq:adjointDJ} introduced to evaluate the contributions of the Neumann and Dirichlet state problems to the error in the quantity of interest: we seek $r_{\Omega,N} \in H^1(\mathcal{D})$ and $r_{\Omega,D} \in H^1_0(\mathcal{D})$ such that respectively $\forall \delta r_N \in H^1(\mathcal{D})$ and $\forall \delta r_D \in H^1_0(\mathcal{D})$
\begin{equation}
a_\Omega(\delta r_i,r_{\Omega,i}) = H_{\Omega,i}(\delta r_i) 
\quad , \quad i=N,D
\label{eq:adjointEIT}
\end{equation}
where for $i=N,D$ the linear forms $H_{\Omega,i}(\delta r_i)$ read as
\begin{equation*}
\begin{aligned}
H_{\Omega,i}(\delta r) & \coloneqq \frac{\partial G}{\partial u}(\Omega,u_{\Omega,i}^h)[ \boldsymbol\theta^h,\delta r] \\
& = \int_\mathcal{D}{ \Big(k_\Omega \mathbf{M}(\boldsymbol\theta^h) \nabla u_{\Omega,i}^h \cdot \nabla \delta r - \nabla \cdot \boldsymbol\theta^h \ u_{\Omega,i}^h \delta r \Big) d\mathbf{x}} .
\end{aligned}
\label{eq:Fadj}
\end{equation*}
As for the state problems, the discretized solutions $r_{\Omega,N}^h$ and $r_{\Omega,D}^h$ are obtained solving the adjoint equations \eqref{eq:adjointEIT} within an appropriate space of Lagrangian finite element functions, that is the space $V_\Omega^{h,m}$ of degree $m$.
According to the requirement of higher-order methods to solve the adjoint problems, we consider a $\mathbb{P}^2$ Lagrangian finite element space for the discretization of \eqref{eq:adjointEIT}.

\subsection{Estimate of the error in the shape gradient via the equilibrated fluxes}

Starting from the framework described in section \ref{ref:GOAL}, we construct a goal-oriented estimator of the error in the shape gradient by evaluating the quantities $\widetilde{E}^h_u$ and $\widetilde{E}^h_p$ in \eqref{eq:Eu}-\eqref{eq:Ep}.
First of all, we observe that owing to the Kohn-Vogelius problem being self-adjoint, this reduces to estimating the quantity $\widetilde{E}^h_u$ for the Neumann and the Dirichlet cases. 
By recalling the expression \eqref{eq:volumetricKV} of the shape derivative for the Kohn-Vogelius functional, we may rewrite the error in the shape gradient as follows:
\begin{equation}
\begin{aligned}
E^h & = \langle dJ(\Omega) - d_hJ(\Omega),\boldsymbol\theta^h \rangle \\
& = \langle G(\Omega,u_{\Omega,N}) - G(\Omega,u_{\Omega,N}^h) , \boldsymbol\theta^h \rangle 
- \langle G(\Omega,u_{\Omega,D}) - G(\Omega,u_{\Omega,D}^h), \boldsymbol\theta^h \rangle \\
& \simeq H_{\Omega,N}(u_{\Omega,N} - u_{\Omega,N}^h) - H_{\Omega,D}(u_{\Omega,D} - u_{\Omega,D}^h) 
\eqqcolon \widetilde{E}^h_{u,_N} - \widetilde{E}^h_{u,_D} .
\end{aligned}
\label{eq:ERRvolumetricKV}
\end{equation}
Before constructing the components of the estimator of the error in the shape gradient in \eqref{eq:ERRvolumetricKV}, we recall the notion of equilibrated fluxes.
In order to do so, we introduce the space of vector-valued functions $\Hdiv = \{ \boldsymbol\tau \in L^2(\mathcal{D}; \mathbb{R}^d) \ : \ \nabla \cdot \boldsymbol\tau \in L^2(\mathcal{D}) \}$ and the discrete space $W_\Omega^{h,\kappa}$ of the functions that restricted to a single element of the triangulation are Raviart-Thomas finite element functions of degree $\kappa$:
$$
W_\Omega^{h,\kappa} \coloneqq \{ \boldsymbol\tau^h \in \Hdiv \ : \ \boldsymbol\tau^h |_T \in [\mathbb{P}^\kappa(T)]^d + \mathbf{x} \mathbb{P}^\kappa(T) \ \forall T \in \mathcal{T}_h \} .
$$
\begin{rmrk}
A function $\boldsymbol\tau^h \in W_\Omega^{h,\kappa}$ is such that $\nabla \cdot \boldsymbol\tau^h \in \mathbb{P}^\kappa(T) \ \forall T \in \mathcal{T}_h$ , $\boldsymbol\tau^h \cdot \mathbf{n}_e \in \mathbb{P}^\kappa(e) \ \forall e \subset \partial T \ , \ \forall T \in \mathcal{T}_h$ and its normal trace is continuous across all edges $e \subset \partial T$ (cf. \cite{BoffiBrezziFortin}).
\end{rmrk}

\subsubsection{Equilibrated fluxes for the state equations}

The discretized solutions $u_{\Omega,i}^h$'s of the state problems are usually such that $- k_\Omega \nabla u_{\Omega,i}^h$ $\notin \Hdiv$ or $\nabla \cdot (- k_\Omega \nabla u_{\Omega,i}^h) + u_{\Omega,i}^h \neq 0$. 
On the contrary, the weak solutions $u_{\Omega,i}$'s - and their fluxes $\boldsymbol\sigma_{\Omega,i} \coloneqq - k_\Omega \nabla u_{\Omega,i}$ - fulfill $\boldsymbol\sigma_{\Omega,i} \in \Hdiv$ and $\nabla \cdot \boldsymbol\sigma_{\Omega,i} + u_{\Omega,i} = 0$. 
In order to retrieve the aforementioned properties, we construct the discrete quantities known as equilibrated fluxes (cf. \cite{MR2373174}):
\begin{defin}
Let $u_{\Omega,i}^h \in V_\Omega^{h,\ell}$ be the solution of a state problem \eqref{eq:stateEIT} computed using Lagrangian finite element functions of degree $\ell$. 
Let $\kappa = \max \{ 0 , \ell-1 \}$, we define $\pi_Z^\kappa : L^2(\mathcal{D}) \rightarrow Z_\Omega^{h,\kappa}$ the $L^2$-orthogonal projection operator onto the space $Z_\Omega^{h,\kappa}$ of the piecewise discontinuous finite element functions of degree $\kappa$.
A function $\boldsymbol\sigma_{\Omega,i}^h \in W_\Omega^{h,\kappa}$ is said to be an equilibrated flux for the problem \eqref{eq:stateEIT} if
\begin{equation}
\nabla \cdot \boldsymbol\sigma_{\Omega,i}^h + \pi_Z^\kappa u_{\Omega,i}^h = 0 .
\label{eq:eqFLuxState} 
\end{equation}
\label{def:fluxState}
\end{defin}
\noindent Under the previously introduced assumptions on the degree of the discretization spaces, we get that $\ell=1$ and $\kappa=0$, that is the equilibrated flux is sought in the lowest-order Raviart-Thomas space $RT_0$ and the projection operator returns $\mathbb{P}^0$ piecewise constant functions.

To practically reconstruct the equilibrated fluxes $\boldsymbol\sigma_{\Omega,i}^h$'s, we follow the approach proposed by Ern and Vohral\'ik in \cite{doi:10.1137/130950100} which is based on the work by Braess and Sch\"{o}berl \cite{MR2373174}. 
In particular, we consider a procedure that starting from the finite element functions $u_{\Omega,i}^h \ , \ i=N,D$ constructs the equilibrated fluxes locally on subpatches of elements. 
Thus, for each vertex $x_v \ , \ v=1,\ldots,N_v$ of the elements in the computational mesh we introduce a linear shape function $\psi_v$ such that $\psi_v(x_w) = \delta_{vw}$, $\delta$ being the 
classical Kronecker delta. 
The support of $\psi_v$ is the subpatch centered in $x_v$ and is denoted by $\omega_v$. We remark that the family of functions $\psi_v$'s fulfills the condition known as partition of the unity, that is 
$$
\sum_{v=1}^{N_v} \psi_v = 1 .
$$
In order to retrieve a precise approximation of the fluxes, we consider a dual mixed finite element formulation of the aforementioned local problems.
First, let us denote by $W_{\omega_v}^{h,\kappa}$ (respectively $Z_{\omega_v}^{h,\kappa}$) the restriction to $\omega_v$ of the previously defined space $W_\Omega^{h,\kappa}$ (respectively $Z_\Omega^{h,\kappa}$).
Moreover, we introduce the following finite element spaces:
\begin{gather*}
W_{v,0}^{h,\kappa} \coloneqq \{ \boldsymbol\tau^h \in W_{\omega_v}^{h,\kappa} \ : \  \boldsymbol\tau^h \cdot \mathbf{n}_e = 0 \ \text{on} \ e \in \partial\omega_v \} ,
\label{eq:Wloc-all} \\
W_{v,1}^{h,\kappa} \coloneqq \{ \boldsymbol\tau^h \in W_{\omega_v}^{h,\kappa} \ : \  \boldsymbol\tau^h \cdot \mathbf{n}_e = 0 \ \text{on} \ e \in \partial\omega_v \setminus \mathcal{E}_h^\mathcal{B} \} .
\label{eq:Wloc-part} 
\end{gather*}
For each vertex $x_v \ , \ v=1,\ldots,N_v$ and for $i=N,D$, we prescribe $(\boldsymbol\sigma_{i,v}^h,t_{i,v}^h) \in W_{i,v}^{h,\kappa} \times Z_{\omega_v}^{h,\kappa}$ such that $\forall (\boldsymbol\delta \boldsymbol\sigma_i^h,\delta t_i^h) \in W_v^{h,\kappa} \times Z_{\omega_v}^{h,\kappa}$
\begin{equation}
\begin{aligned} 
& \int_{\omega_v}{\nabla \cdot \boldsymbol\sigma_{i,v}^h \delta t_i^h \ d\mathbf{x}} + \int_{\omega_v}{t_{i,v}^h \delta t_i^h \ d\mathbf{x}} = - \int_{\omega_v}{\Big(k_\Omega \nabla u_{\Omega,i}^h \cdot \nabla \psi_v + u_{\Omega,i}^h \psi_v \Big) \delta t_i^h \ d\mathbf{x}} , \\
& \int_{\omega_v}{\boldsymbol\sigma_{i,v}^h \cdot \boldsymbol\delta \boldsymbol\sigma_i^h \ d\mathbf{x}} - \int_{\omega_v}{k_\Omega t_{i,v}^h \nabla \cdot \boldsymbol\delta \boldsymbol\sigma_i^h \ d\mathbf{x}} = - \int_{\omega_v}{k_\Omega \psi_v \nabla u_{\Omega,i}^h \cdot \boldsymbol\delta \boldsymbol\sigma_i^h \ d\mathbf{x}} .
\end{aligned}
\label{eq:mixedPBstate} 
\end{equation}
The spaces in which the trial and the test functions are sought are detailed below.
It is important to highlight the different nature of problem \eqref{eq:mixedPBstate} when the patch $\omega_v$ is centered on a vertex belonging to the interior of $\mathcal{D}$ or to its boundary $\partial\mathcal{D}$.
As Braess and Sch\"{o}berl remark in \cite{MR2373174}, some caution has to be used when dealing with the corresponding boundary conditions: in particular, a flux-free condition is imposed on the whole boundary $\partial\omega_v$ of the patch for interior vertices, whereas it is limited to the edges in $\partial\omega_v \setminus \mathcal{E}_h^\mathcal{B}$ for points which belong to the external boundary of the global domain. \\
To construct the equilibrated fluxes for the Neumann state problem on $\omega_v$ centered in a vertex $x_v \in \partial\mathcal{D}$, equation \eqref{eq:mixedPBstate} is solved using the spaces
 \begin{gather}
\begin{aligned}
W_{N,v}^{h,\kappa} \coloneqq \{ \boldsymbol\tau^h \in W_{\omega_v}^{h,\kappa} \ : & \  \boldsymbol\tau^h \cdot \mathbf{n}_e = 0 \ \text{on} \ e \in \partial\omega_v \setminus \mathcal{E}_h^\mathcal{B} \\
& \text{and} \ \boldsymbol\tau^h \cdot \mathbf{n}_e = \pi_{W \cdot \mathbf{n}}^\kappa(\psi_v g) \ \text{on} \ e \in \partial\omega_v \cap \mathcal{E}_h^\mathcal{B} \} ,
\end{aligned}
\label{eq:trialExtNeu} \\
W_v^{h,\kappa} = W_{v,0}^{h,\kappa} .
\label{eq:testExtNeu}
\end{gather}
When considering the Dirichlet state problem on $\omega_v$ centered in $x_v \in \partial\mathcal{D}$, the trial and test spaces read as follows:
\begin{equation}
W_{D,v}^{h,\kappa} = W_v^{h,\kappa} = W_{v,1}^{h,\kappa} .
\label{eq:trial-testExtDir}
\end{equation}
Eventually, for the vertices $x_v \in int(\mathcal{D})$, we solve \eqref{eq:mixedPBstate} using the spaces
\begin{equation}
W_{N,v}^{h,\kappa} = W_{D,v}^{h,\kappa} = W_v^{h,\kappa} = W_{v,0}^{h,\kappa} .
\label{eq:trial-testInt}
\end{equation}
In \eqref{eq:trialExtNeu}, $\pi_{W \cdot \mathbf{n}}^\kappa$ stands for the $L^2$-orthogonal projection operator from $L^2(\partial\mathcal{D})$ to the space $W_\Omega^{h,\kappa} \cdot \mathbf{n}$
of polynomial functions of degree at most $\kappa$ on the external boundary.
For additional details on the procedure to construct the equilibrated fluxes and on the properties of the resulting \emph{a posteriori} error estimators, we refer to \cite{doi:10.1137/130950100}.

We now extend all the $\boldsymbol\sigma_{i,v}^h$'s by zero in $\mathcal{D} \setminus \omega_v$.
By combining the above information arising from all the subpatches, we may retrieve the global equilibrated fluxes for the state problems:
\begin{equation*}
\boldsymbol\sigma_{\Omega,i}^h = \sum_{v=1}^{N_v} \boldsymbol\sigma_{i,v}^h \quad , \quad i=N,D .
\label{eq:sumFluxState}
\end{equation*}
\begin{lemma}
For the case of the Neumann state problem, it holds 
\begin{equation*}
\boldsymbol\sigma_{\Omega,N}^h \cdot \mathbf{n} = \pi_{W \cdot \mathbf{n}}^\kappa(g) \quad \text{on} \quad \partial\mathcal{D} .
\label{eq:fluxProj}
\end{equation*}
\label{theo:lemmaFlux}
\end{lemma}
\begin{proof}
Let $\chi_v^e$ be equal to $1$ if a given edge $e \in \mathcal{E}_h^\mathcal{B}$ belongs to the subpatch $\omega_v$ centered in $x_v$ and $0$ otherwise. Hence, 
$$
\boldsymbol\sigma_{\Omega,N}^h|_e = \sum_{v=1}^{N_v}{\chi_v^e \boldsymbol\sigma_{N,v}^h} .
$$
Let $\delta u^h \in (W_\Omega^{h,\kappa} \cdot \mathbf{n})|_e$ be a polynomial function of degree at most $\kappa$ on the edge $e \in \mathcal{E}_h^\mathcal{B}$.
Owing to the condition on the normal trace $\boldsymbol\sigma_{N,v}^h \cdot \mathbf{n}_e$ in \eqref{eq:trialExtNeu}, we get
$$
\langle \boldsymbol\sigma_{\Omega,N}^h \cdot \mathbf{n}_e, \delta u^h \rangle_e = \sum_{v=1}^{N_v}{\chi_v^e \langle \boldsymbol\sigma_{N,v}^h \cdot \mathbf{n}_e, \delta u^h \rangle_e} 
= \sum_{v=1}^{N_v}{\chi_v^e \langle \psi_v g, \delta u^h \rangle_e} = \langle g, \delta u^h \rangle_e ,
$$
where the last equality follows from the partition of the unity property fulfilled by the functions $\psi_v$'s.
The result is inferred by observing that the previous chain of equality holds $\forall \delta u^h \in (W_\Omega^{h,\kappa} \cdot \mathbf{n})|_e \ , \ \forall e \in \mathcal{E}_h^\mathcal{B}$.
\end{proof}

\subsubsection{Equilibrated fluxes for the adjoint equations}

Following the same approach discussed above for the state problems, we define the equilibrated fluxes for the adjoint problems:
\begin{defin}
Let $r_{\Omega,i}^h \in V_\Omega^{h,m}$ be the solution of an adjoint problem \eqref{eq:adjointEIT} computed using Lagrangian finite element functions of degree $m$. Let $\kappa = \max \{ 0 , m-1 \}$ 
and $\pi_Z^\kappa : L^2(\mathcal{D}) \rightarrow Z_\Omega^{h,\kappa}$ the $L^2$-orthogonal projection operator onto the space $Z_\Omega^{h,\kappa}$ defined in the previous section.
A function $\boldsymbol\xi_{\Omega,i}^h \in W_\Omega^{h,\kappa}$ is said to be an equilibrated flux for the problem \eqref{eq:adjointEIT} if
\begin{equation}
\nabla \cdot \boldsymbol\xi_{\Omega,i}^h + \pi_Z^\kappa r_{\Omega,i}^h = - \pi_Z^\kappa \Big( \nabla \cdot (k_\Omega \mathbf{M}(\boldsymbol\theta^h) \nabla u_{\Omega,i}^h) + \nabla \cdot \boldsymbol\theta^h \ u_{\Omega,i}^h \Big) .
\label{eq:eqFLuxAdjoint} 
\end{equation}
\label{def:fluxAdjoint}
\end{defin}
\noindent Having in mind that the adjoint equations are solved using $\mathbb{P}^2$ Lagrangian finite element functions - that is $m=2$ - it follows that the equilibrated fluxes $\boldsymbol\xi_{\Omega,i}^h$'s are constructed via $RT_1$ Raviart-Thomas functions of degree $1$ and the operator $\pi_Z^k$ projects functions from $L^2(\mathcal{D})$ to the discrete space of piecewise discontinuous finite elements of degree $1$.

The computation of the equilibrated fluxes for the adjoint problems is again performed via the solution of a mixed finite element problem. 
We consider the same discrete spaces introduced in definitions \eqref{eq:testExtNeu} to \eqref{eq:trial-testInt}, whereas the space $W_{N,v}^{h,\kappa}$ associated with the Neumann adjoint problem featuring a patch centered on a boundary node is $W_{v,0}^{h,\kappa}$.
Thus, for each subpatch $\omega_v \ , \ v=1,\ldots,N_v$ and for $i=N,D$, we seek $(\boldsymbol\xi_{i,v}^h,q_{i,v}^h) \in W_{i,v}^{h,\kappa} \times Z_{\omega_v}^{h,\kappa}$ such that $\forall (\boldsymbol\delta \boldsymbol\xi_i^h,\delta q_i^h) \in W_v^{h,\kappa} \times Z_{\omega_v}^{h,\kappa}$
\begin{equation}
\begin{aligned}
& \begin{aligned}
\int_{\omega_v}\nabla \cdot \boldsymbol\xi_{i,v}^h \delta & q_i \ d\mathbf{x} + \int_{\omega_v}{q_{i,v}^h \delta q_i \ d\mathbf{x}} = \int_{\omega_v}{k_\Omega \mathbf{M}(\boldsymbol\theta^h) \nabla u_{\Omega,i}^h \cdot \nabla \psi_v \delta q_i \ d\mathbf{x}} \\
& - \int_{\omega_v}{\nabla \cdot \boldsymbol\theta^h u_{\Omega,i}^h \psi_v \delta q_i \ d\mathbf{x}} - \int_{\omega_v}{\Big(k_\Omega \nabla r_{\Omega,i}^h \cdot \nabla \psi_v + r_{\Omega,i}^h \psi_v \Big) \delta q_i \ d\mathbf{x}} ,
  \end{aligned} \\
& \begin{aligned}
 \int_{\omega_v}{\boldsymbol\xi_{i,v}^h \cdot \boldsymbol\delta \boldsymbol\xi_i \ d\mathbf{x}} - \int_{\omega_v}{k_\Omega q_{i,v}^h \nabla \cdot \boldsymbol\delta \boldsymbol\xi_i \ d\mathbf{x}} = \int_{\omega_v}k_\Omega \psi_v & \mathbf{M}(\boldsymbol\theta^h) \nabla u_{\Omega,i}^h \cdot \boldsymbol\delta \boldsymbol\xi_i \ d\mathbf{x} \\
& - \int_{\omega_v}{k_\Omega \psi_v \nabla r_{\Omega,i}^h \cdot \boldsymbol\delta \boldsymbol\xi_i \ d\mathbf{x}} .
   \end{aligned}
\end{aligned}
\label{eq:mixedPBadjoint} 
\end{equation}
The corresponding equilibrated fluxes $\boldsymbol\xi_{\Omega,i}^h$'s are obtained by extending the functions $\boldsymbol\xi_{i,v}^h$'s by zero in $\mathcal{D} \setminus \omega_v$ and by combining the previously computed local information:
$$
\boldsymbol\xi_{\Omega,i}^h = \sum_{v=1}^{N_v} \boldsymbol\xi_{i,v}^h \quad , \quad i=N,D .
$$
\begin{rmrk}
A key aspect of the discussed procedure is represented by the local nature of the problems to be solved for the construction of the equilibrated fluxes. 
The advantage of this approach is twofold. On the one hand, solving the local problems \eqref{eq:mixedPBstate}-\eqref{eq:mixedPBadjoint} is computationally inexpensive owing to the small size of the subpatches. On the other hand, every problem set on a subpatch is independent from the remaining ones thus it is straightforward to implement a version of the procedure that can efficiently exploit modern parallel architectures.
\end{rmrk}

\subsubsection{Goal-oriented equilibrated fluxes error estimator}
\label{ref:goalFE}

As previously stated, the construction of the error estimator for the shape gradient for the case of the electrical impedance tomography reduces to the evaluation of \eqref{eq:Eu} for the Neumann and Dirichlet problems. 
For this purpose, we introduce respectively the quantities $\widetilde{E}^h_{u,_N}$ and $\widetilde{E}^h_{u,_D}$ and two parameters $\zeta_i$'s such that $\zeta_N \coloneqq 1$ and $\zeta_D \coloneqq 0$.
By exploiting the formulation of the bilinear and linear forms \eqref{eq:a}-\eqref{eq:F} and adding the expression of the equilibrated fluxes \eqref{eq:eqFLuxState}, $\widetilde{E}^h_{u,_i}$ reads as:
\begin{equation*}
\begin{aligned}
\widetilde{E}^h_{u,_i} \coloneqq F_{\Omega,i}(r_{\Omega,i}) - a_\Omega(& u_{\Omega,i}^h,r_{\Omega,i}) = \ \zeta_i \int_{\partial\mathcal{D}}{g r_{\Omega,i} \ ds} 
  - \int_\mathcal{D}{k_\Omega \nabla u_{\Omega,i}^h \cdot \nabla r_{\Omega,i} d\mathbf{x}} \\
& \hspace{12pt} - \int_\mathcal{D}{u_{\Omega,i}^h r_{\Omega,i} d\mathbf{x}} 
  + \int_\mathcal{D}{\Big( \nabla \cdot \boldsymbol\sigma_{\Omega,i}^h + \pi_Z^\kappa u_{\Omega,i}^h \Big) r_{\Omega,i} \ d\mathbf{x}} .
\end{aligned}
\label{eq:GOAL1step}
\end{equation*}
Integrating by parts the last integral and owing to lemma \ref{theo:lemmaFlux} and to $r_{\Omega,D}=0 \ \text{on} \ \partial\mathcal{D}$, we obtain
\begin{equation*}
\begin{aligned}
\widetilde{E}^h_{u,_i} = \ \zeta_i & \int_{\partial\mathcal{D}}{\Big( g - \pi_{W \cdot \mathbf{n}}^\kappa(g) \Big) r_{\Omega,i} \ ds} 
+ \int_\mathcal{D}{\Big( \pi_Z^\kappa u_{\Omega,i}^h - u_{\Omega,i}^h \Big) r_{\Omega,i} \ d\mathbf{x}} \\
& - \int_\mathcal{D}{\Big( \boldsymbol\sigma_{\Omega,i}^h + k_\Omega \nabla u_{\Omega,i}^h \Big) \cdot \nabla r_{\Omega,i} \ d\mathbf{x}} .
\end{aligned}
\label{eq:GOAL2step}
\end{equation*}
By adding and subtracting the corresponding terms featuring the finite element counterparts $r_{\Omega,i}^h$'s of the adjoint solutions and owing to definition \ref{def:fluxAdjoint} of the equilibrated fluxes 
$\boldsymbol\xi_{\Omega,i}^h$'s, we are finally able to derive the expression of the errors $\widetilde{E}^h_{u,_i}$'s:
\begin{equation}
\begin{aligned} 
\widetilde{E}^h_{u,_i} = & \ \zeta_i \int_{\partial\mathcal{D}}{\Big( g - \pi_{W \cdot \mathbf{n}}^\kappa(g) \Big) r_{\Omega,i}^h \ ds} 
+ \zeta_i \int_{\partial\mathcal{D}}{\Big( g - \pi_{W \cdot \mathbf{n}}^\kappa(g) \Big)(r_{\Omega,i} - r_{\Omega,i}^h) ds} \\
& + \int_\mathcal{D}{\Big( \pi_Z^\kappa u_{\Omega,i}^h - u_{\Omega,i}^h \Big) r_{\Omega,i}^h \ d\mathbf{x}} 
+ \int_\mathcal{D}{\Big( \pi_Z^\kappa u_{\Omega,i}^h - u_{\Omega,i}^h \Big)(r_{\Omega,i} - r_{\Omega,i}^h) d\mathbf{x}} \\
& + \int_\mathcal{D}{\Big( \boldsymbol\sigma_{\Omega,i}^h + k_\Omega \nabla u_{\Omega,i}^h \Big) \cdot k_\Omega^{-1} \boldsymbol\xi_{\Omega,i}^h \ d\mathbf{x}} \\
& - \int_\mathcal{D}{\Big( \boldsymbol\sigma_{\Omega,i}^h + k_\Omega \nabla u_{\Omega,i}^h \Big) \cdot \Big(\nabla r_{\Omega,i} + k_\Omega^{-1} \boldsymbol\xi_{\Omega,i}^h\Big) d\mathbf{x}} .
\end{aligned}
\label{eq:GOALconformingFE}
\end{equation}
We remark that in \eqref{eq:GOALconformingFE} both the exact and the discretized solutions of the adjoint problems appear. 
From a practical point of view, in order to fully compute the quantity \eqref{eq:GOALconformingFE} we substitute the exact solutions with their finite element counterparts $r_{\Omega,i}^h \in V_\Omega^{h,m}$ obtained by the high-order approximation of \eqref{eq:adjointEIT}. 
The corresponding approximated solutions are then replaced by the projection of the high-order approximations onto the space $V_\Omega^{h,\ell}$ of the low-order finite element functions used for the discretization of the state problems. Let $I_m^\ell: V_\Omega^{h,m} \rightarrow V_\Omega^{h,\ell}$ be the projection operator from the space of high-order approximations to the low-order one. 
The fully-computable version of the estimator of the quantity $\widetilde{E}^h_{u,_i}$ is obtained by substituting $r_{\Omega,i} - r_{\Omega,i}^h$ with $r_{\Omega,i}^h - I_m^\ell r_{\Omega,i}^h$ and $\nabla r_{\Omega,i}$ with $\nabla r_{\Omega,i}^h$ in \eqref{eq:GOALconformingFE}. 
By plugging the expressions of $\widetilde{E}^h_{u,_N}$ and $\widetilde{E}^h_{u,_D}$ arising from \eqref{eq:GOALconformingFE} into \eqref{eq:ERRvolumetricKV}, we obtain a computable expression of the error in the shape gradient and the bound $\overline{E}$ follows by considering its absolute value.
\begin{rmrk}
The goal-oriented error estimators constructed using the equilibrated fluxes approach are known to be asymptotically exact (cf. \cite{MR2899560}). 
Owing to the aforementioned asymptotic exactness, the term $\overline{E}$ tends to zero as the mesh size tends to zero.
This property plays a crucial role since it guarantees that the mesh adaptation routine performed to certify the descent direction (cf. algorithm \ref{scpt:shape-opt-adaptive} - step 5) eventually leads to the fulfillment of condition \eqref{eq:certified}.
\end{rmrk}

\section{Discontinuous Galerkin approximation}
\label{ref:DiscontinuousGalerkin}

In this section, we present an alternative strategy for the approximation of the EIT problem based on the symmetric weighted interior penalty discontinuous Galerkin (SWIP-dG) formulation. \\
Let us consider the notations introduced in section \ref{ref:ConformingFE} for the triangulation $\mathcal{T}_h$. The discontinuous Galerkin (dG) problems are solved within the space
$$
V_\Omega^{h,\kappa} \coloneqq \{ u^h \in L^2(\mathcal{D}) \ : \ u^h|_T \in \mathbb{P}^\kappa(T) \ \forall T \in \mathcal{T}_h  \} 
$$ 
of the discontinuous functions whose restrictions to a single element are polynomials of degree at most $\kappa$.
When dealing with dG formulations, discontinuous functions - as the ones of the aforementioned space $V_\Omega^{h,\kappa}$ - which are double-valued on $\mathcal{E}_h^\mathcal{I}$ and single-valued on $\mathcal{E}_h^\mathcal{B}$ have to be properly handled.
We define the jump of $u^h$ across the edge $e$ shared by the elements $T^\pm(e)$ as
\begin{equation*}
\llbracket u^h \rrbracket_e \coloneqq u^h|_{T^-(e)} - u^h|_{T^+(e)} .
\label{eq:jumpDG} 
\end{equation*}
In a similar fashion, the weighted average of $u^h$ on $e \in \mathcal{E}_h^\mathcal{I}$ reads as follows
\begin{equation}
\llbrace u^h \rrbrace_\alpha \coloneqq \alpha_{T^-(e),e} u^h|_{T^-(e)} + \alpha_{T^+(e),e} u^h|_{T^+(e)} .
\label{eq:weightedAvDG} 
\end{equation}
where the weights are non-negative quantities such that $\alpha_{T^-(e),e} + \alpha_{T^+(e),e} = 1$.
On boundary edges, we set $\llbracket u^h \rrbracket_e = u^h |_e$, $\alpha_{T^-(e),e} = 1$ and $\llbrace u^h \rrbrace_\alpha = u^h$.

Classical discontinuous Galerkin methods use arithmetic averages in \eqref{eq:weightedAvDG}, that is for all edges the weights are constant and equal $\alpha_{T^-(e),e} = \alpha_{T^+(e),e} = 1/2$.
As stated in the introduction, in recent years there has been a growing interest towards the so-called symmetric weighted interior penalty dG methods, especially when dealing with problems featuring inhomogeneous coefficients for the diffusion term (cf. \cite{doi:10.1137/050634736, MR2491426}). 
In particular, these methods rely on the definition of weights based on the information carried by the diffusion tensor. For the case of the electrical impedance tomography under analysis, this results in the following weights based on the different values of the electrical conductivity:
$$
\alpha_{T^-(e),e} \coloneqq \frac{k_\Omega|_{T^+(e)}}{k_\Omega|_{T^+(e)} + k_\Omega|_{T^-(e)}} 
\quad , \quad 
\alpha_{T^+(e),e} \coloneqq \frac{k_\Omega|_{T^-(e)}}{k_\Omega|_{T^+(e)} + k_\Omega|_{T^-(e)}} .
$$
It is well-known in the literature \cite{MR2882148} that the bilinear form associated with discontinuous Galerkin methods may suffer from lack of coercivity thus preventing the discrete problem from having a unique solution. 
A widely-spread workaround (cf. \cite{Shahbazi2005401}) is represented by the interior penalty approach that introduces a \emph{sufficiently large} penalization in order to retrieve the coercivity of the discrete bilinear form. 
Owing to the idea of exploiting the information carried by the diffusion tensor to construct the weights for the jump term, we define the stabilization parameter in a similar way \cite{MR2427189}:
$$
\gamma_e \coloneqq \beta_e \frac{k_\Omega|_{T^+(e)} k_\Omega|_{T^-(e)}}{k_\Omega|_{T^+(e)} + k_\Omega|_{T^-(e)}} 
$$
where $\beta_e > 0$ is a user-dependent parameter. 

As for the conforming finite element approximation described in the previous section, first we introduce the discrete state and adjoint problems and then we construct the equilibrated fluxes via a procedure relying solely on local quantities. 
As previously stated, a key aspect of this approach is represented by the choice of the degree of the approximating functions for both the solution of the problems and the equilibrated fluxes. 
The details of this choice will be discussed in the following subsections.
For the sake of readability, from now on we will omit the subscript $e$ associated with jumps, weights and averages if there is no risk of ambiguity.

\subsection{The state problems}

In order to appropriately handle the terms involving the effect of the boundary data in the estimator of the error in the shape gradient, the boundary conditions have to imposed using the same strategy in both the weak and the discrete formulation. 
Owing to the fact that the essential boundary conditions are classically verified in a weak sense in discontinuous Galerkin methods, we consider an alternative formulation of \eqref{eq:a}-\eqref{eq:F} to weakly 
impose the Dirichlet boundary condition on $\partial\mathcal{D}$.
Let $\zeta_N \coloneqq 1$ and $\zeta_D \coloneqq 0$. The bilinear forms $a_{\Omega,i}(\cdot,\cdot)$ and the linear ones $F_{\Omega,i}(\cdot)$ associated with problems \eqref{eq:statePB} coupled with the boundary conditions \eqref{eq:NeumannBC} and \eqref{eq:DirichletBC} respectively read as:
\begin{equation}
\begin{aligned}
a_{\Omega,i}(u_{\Omega,i},\delta u) =  &
\int_\mathcal{D}{\Big(k_\Omega \nabla u_{\Omega,i} \cdot \nabla \delta u + u_{\Omega,i} \delta u \Big) d\mathbf{x}} \\
& - (1-\zeta_i) \int_{\partial\mathcal{D}}{\Big( k_\Omega \nabla u_{\Omega,i} \cdot \mathbf{n} \delta u + u_{\Omega,i} k_\Omega \nabla \delta u \cdot \mathbf{n} \Big) ds} \\
& + (1-\zeta_i) \int_{\partial\mathcal{D}}{\gamma u_{\Omega,i} \delta u \ ds} ,
\end{aligned}
\label{eq:aWeak}
\end{equation}
\begin{equation}
F_{\Omega,N}(\delta u) = \int_{\partial\mathcal{D}}{g \delta u \ ds} 
\quad , \quad 
F_{\Omega,D}(\delta u) = \int_{\partial\mathcal{D}} { U_D (\gamma \delta u - k_\Omega \nabla \delta u \cdot \mathbf{n}) ds} .
\label{eq:Fweak} 
\end{equation}
We refer to appendix \ref{ref:essential} for the formal derivation of \eqref{eq:aWeak}-\eqref{eq:Fweak} in the Dirichlet case.
The variational formulation of the state equations \eqref{eq:statePB} reads as follows: for $i=N,D$ we seek $u_{\Omega,i} \in H^1(\mathcal{D})$ such that 
\begin{equation*}
a_{\Omega,i}(u_{\Omega,i},\delta u_i) = F_{\Omega,i}(\delta u_i) \quad \forall \delta u_i \in H^1(\mathcal{D}) .
\label{eq:stateNitsche}
\end{equation*}
The corresponding discrete bilinear and linear forms arising from the interior penalty discontinuous Galerkin method have the following expressions:
\begin{gather}
\begin{aligned}
  a_{\Omega,i}^h(u_{\Omega,i}^h,\delta u^h) =  &
  \sum_{T \in \mathcal{T}_h}{\int_T{\Big( k_{\Omega} \nabla u_{\Omega,i}^h \cdot \nabla \delta u^h + u_{\Omega,i}^h \delta u^h \Big) d\mathbf{x}}} \\
  - & \sum_{e \in \mathcal{E}_h^\mathcal{I}}{\int_{e}{\Big( \mathbf{n}_e \cdot \llbrace k_{\Omega} 
\nabla u_{\Omega,i}^h \rrbrace_\alpha \llbracket \delta u^h \rrbracket + 
\llbracket u_{\Omega,i}^h \rrbracket \mathbf{n}_e \cdot \llbrace k_{\Omega} \nabla 
\delta u^h \rrbrace_\alpha \Big) ds}} \\
  - & (1-\zeta_i) \sum_{e \in \mathcal{E}_h^\mathcal{B}}{  \int_{e}{\mathbf{n}_e \cdot \llbrace k_{\Omega} 
\nabla u_{\Omega,i}^h \rrbrace_\alpha \llbracket \delta u^h \rrbracket ds}} \\
  - & (1-\zeta_i) \sum_{e \in \mathcal{E}_h^\mathcal{B}}{  \int_{e}{\llbracket u_{\Omega,i}^h \rrbracket 
  \mathbf{n}_e \cdot \llbrace k_{\Omega} \nabla \delta u^h \rrbrace_\alpha ds}} \\
&  + \sum_{e \in \mathcal{E}_h^\mathcal{I}}{ \int_{e}{\frac{\gamma_e}{|e|} \llbracket 
u_{\Omega,i}^h \rrbracket \llbracket \delta u^h \rrbracket ds}} 
  + (1-\zeta_i) \sum_{e \in \mathcal{E}_h^\mathcal{B}}{ \int_{e}{\frac{\gamma_e}{|e|} \llbracket 
u_{\Omega,i}^h \rrbracket \llbracket \delta u^h \rrbracket ds}} ,
  \end{aligned}
\label{eq:aDG} \\
\begin{aligned}
F_{\Omega,N}^h(\delta u^h) & = \int_{\partial\mathcal{D}}{g \delta u^h \ ds} , \\
F_{\Omega,D}^h(\delta u^h) = \sum_{e \in \mathcal{E}_h^B} & { \int_{e}{U_D \Big(\frac{\gamma_e}{|e|} \delta u^h - k_\Omega \nabla \delta u^h \cdot \mathbf{n}_e \Big) ds}} .
\end{aligned}
\notag
\end{gather}
Thus, according to the SWIP-dG problem we seek $u_{\Omega,N}^h, u_{\Omega,D}^h \in V_\Omega^{h,\ell}$ such that 
\begin{equation}
a_{\Omega,i}^h(u_{\Omega,i}^h,\delta u_i^h) = F_{\Omega,i}^h(\delta u_i^h) \quad \forall \delta u_i^h \in V_\Omega^{h,\ell} .
\label{eq:stateSWIP}
\end{equation}
Concerning the degree of the discontinuous Galerkin approximating functions, we maintain the same choice previously presented for the conforming finite element discretization, that is a low-order approximation based on piecewise linear polynomials ($\ell=1$). 
In a similar fashion, the computation of the descent direction $\theta^h$ is performed by means of the conforming discretization using the space of $\mathbb{P}^1 \times \mathbb{P}^1$ Lagrangian finite element functions discussed in section \ref{ref:ConformingFE}.

\subsection{The adjoint problems}

The symmetric weighted interior penalty discontinuous Galerkin formulation of the adjoint problems may be derived following the same procedure used for the state problems. 
In particular, the bilinear forms in \eqref{eq:aDG} also stand for the Neumann and Dirichlet adjoint problems. The corresponding linear forms for $i=N,D$ read as
\begin{equation*}
\begin{aligned}
H_{\Omega,i}^h(\delta r^h) = &
\sum_{T \in \mathcal{T}_h}{\int_T{\Big(k_\Omega \mathbf{M}(\boldsymbol\theta^h) \nabla u_{\Omega,i}^h \cdot \nabla \delta r^h - \nabla \cdot \boldsymbol\theta^h \ u_{\Omega,i}^h \delta r^h \Big) d\mathbf{x}}} \\
  & - \sum_{e \in \mathcal{E}_h^\mathcal{I}}{\int_{e}{\mathbf{n}_e \cdot \llbrace k_{\Omega} 
\mathbf{M}(\boldsymbol\theta^h) \nabla u_{\Omega,i}^h \rrbrace_\alpha \llbracket \delta r^h \rrbracket ds}} \\
  & - \sum_{e \in \mathcal{E}_h^\mathcal{I}}{\int_{e}{\llbracket k_{\Omega} \mathbf{M}(\boldsymbol\theta^h) 
  \nabla u_{\Omega,i}^h \rrbracket \mathbf{n}_e \cdot \llbrace \delta r^h \rrbrace_\alpha ds}} \\
& - (1-\zeta_i) \int_{\partial\mathcal{D}}{ k_{\Omega} \mathbf{M}(\boldsymbol\theta^h) \nabla u_{\Omega,i}^h \cdot \mathbf{n} \ \delta r^h \ ds} .
\end{aligned}
\label{eq:FadjDG} 
\end{equation*}
The discretized solutions of the adjoint problems are the functions $r_{\Omega,i}^h \in V_\Omega^{h,m}$ such that $\forall \delta r_i^h \in V_\Omega^{h,m}$
\begin{equation}
a_{\Omega,i}^h(\delta r_i^h,r_{\Omega,i}^h) = H_{\Omega,i}^h(\delta r_i^h) \quad  , \quad i=N,D .
\label{eq:adjointSWIP}
\end{equation}
It is straightforward to verify that the SWIP-dG formulation of the adjoint problems is consistent, that is \eqref{eq:adjointSWIP} stands substituting the analytical solutions $r_{\Omega,i}$'s to their discretized counterparts $r_{\Omega,i}^h$'s (cf. \cite{MR2882148}).
As previously stated, this property plays a crucial role in the construction of discretizations of optimal order in terms of target functionals and we refer to \cite{MR2361907} for a detailed presentation of this subject.
In order to obtain a higher-order approximation of the adjoint problems, we consider $m=2$, as for the case of the conforming finite element approximation in section \ref{ref:ConformingFE}.

\subsection{Estimate of the error in the shape gradient via the equilibrated fluxes}

In this section we construct the equilibrated fluxes associated with the discontinuous Galerkin approximations \eqref{eq:stateSWIP} and \eqref{eq:adjointSWIP} and we derive the corresponding goal-oriented estimator of the error in the shape gradient. 
Following the procedure introduced for the case of conforming finite element discretization, this problem reduces to estimating the quantity \eqref{eq:ERRvolumetricKV}.

\subsubsection{Equilibrated fluxes for the state equations}

We introduced the notion of equilibrated fluxes for the state problems in definition \ref{def:fluxState}.
In particular, for each problem we aim to construct an $\Hdiv$-conforming flux $\boldsymbol\sigma_{\Omega,i}^h \in W_\Omega^{h,\kappa}$ such that \eqref{eq:eqFLuxState} stands.
We recall that the state problems are approximated using discontinuous Galerkin functions of degree $\ell=1$, thus the fluxes are reconstructed using $RT_0$ finite element functions ($\kappa=0$).
Owing to the nature of the degrees of freedom of the lowest-order Raviart-Thomas finite element functions, the construction of the equilibrated fluxes is straightforward via the prescription of the normal fluxes on all the edges:
\begin{align}
& \begin{aligned}
\int_e{\boldsymbol\sigma_{\Omega,i}^h \cdot \mathbf{n}_e \ \delta t^h \ ds} =  \int_e \Big( \frac{\gamma_e}{|e|} \llbracket u_{\Omega,i}^h \rrbracket - \mathbf{n}_e \cdot \llbrace & k_\Omega \nabla u_{\Omega,i}^h \rrbrace_\alpha \Big) \delta t^h \ ds 
\quad , \\
& \ \forall \delta t^h \in \mathbb{P}^\kappa(e) \ \ \forall e \in \mathcal{E}_h^\mathcal{I}    
  \end{aligned}
\label{eq:internalFluxState} \\
& \begin{aligned}
\int_e{\boldsymbol\sigma_{\Omega,i}^h \cdot \mathbf{n}_e \ \delta t^h \ ds} = (1-\zeta_i & ) \int_e{\Big( \frac{\gamma_e}{|e|} (u_{\Omega,i}^h - U_D) - k_\Omega \nabla u_{\Omega,i}^h \cdot \mathbf{n}_e \Big) \delta t^h \ ds} \\
& - \zeta_i \int_e{g \ \delta t^h \ ds} 
\quad , \quad 
\forall \delta t^h \in \mathbb{P}^\kappa(e) \ \ \forall e \in \mathcal{E}_h^\mathcal{B} . 
\end{aligned}
\label{eq:boundaryFluxState}
\end{align}

\subsubsection{Equilibrated fluxes for the adjoint equations}

In an analogous way, we may construct the equilibrated fluxes for the adjoint problems. 
We remark that owing to the higher-order approximation of \eqref{eq:adjointSWIP} with respect to \eqref{eq:stateSWIP} - i.e. $m=2$ -, the equilibrated fluxes $\boldsymbol\xi_{\Omega,i}^h$'s in definition \ref{def:fluxAdjoint} are sought in the space $W_\Omega^{h,\kappa} \ , \ \kappa=1$. 
The $RT_1$ reconstructed fluxes are such that
\begin{align*}
& \begin{aligned}
\int_e{\boldsymbol\xi_{\Omega,i}^h \cdot \mathbf{n}_e \ \delta q_1^h \ ds} =  \int_e \Big( \frac{\gamma_e}{|e|} \llbracket r_{\Omega,i}^h \rrbracket - \mathbf{n}_e \cdot \llbrace & k_\Omega \nabla r_{\Omega,i}^h \rrbrace_\alpha \Big) \delta q_1^h \ ds 
\quad , \\
& \ \forall \delta q_1^h \in \mathbb{P}^\kappa(e) \ \ \forall e \in \mathcal{E}_h^\mathcal{I}
  \end{aligned}
\\
& \begin{aligned}
\int_e{\boldsymbol\xi_{\Omega,i}^h \cdot \mathbf{n}_e \ \delta q_1^h \ ds} = \ (1-\zeta_i) & \int_e{\Big( \frac{\gamma_e}{|e|} r_{\Omega,i}^h - k_\Omega \nabla r_{\Omega,i}^h \cdot \mathbf{n}_e \Big) \delta q_1^h \ ds} \\
- & \zeta_i \int_e{k_\Omega \mathbf{M}(\boldsymbol\theta^h) \nabla u_{\Omega,i}^h \cdot \mathbf{n}_e \ \delta q_1^h \ ds} 
\quad , \\
& \hspace{55pt} \forall \delta q_1^h \in \mathbb{P}^\kappa(e) \ \ \forall e \in \mathcal{E}_h^\mathcal{B} &
\end{aligned}
\\
& \begin{aligned}
\int_T \boldsymbol\xi_{\Omega,i}^h \cdot \boldsymbol\delta \mathbf{q}_2^h \ d\mathbf{x} = & - \int_T{ k_\Omega \nabla r_{\Omega,i}^h \cdot \boldsymbol\delta \mathbf{q}_2^h \ d\mathbf{x}} \\
& + \sum_{e \subset \partial T \setminus \mathcal{E}_h^\mathcal{B}}{\alpha_{T(e),e} \int_e{k_\Omega \llbracket r_{\Omega,i}^h \rrbracket \boldsymbol\delta \mathbf{q}_2^h \cdot \mathbf{n}_e \ ds}} \\
& + (1-\zeta_i) \sum_{e \subset \partial T \cap \mathcal{E}_h^\mathcal{B}}{\int_e{k_\Omega r_{\Omega,i}^h \boldsymbol\delta \mathbf{q}_2^h \cdot \mathbf{n}_e \ ds}} 
\quad , \\
& \hspace{60pt} \forall \boldsymbol\delta \mathbf{q}_2^h \in [\mathbb{P}^{\kappa-1}(T)]^d \ \ \forall T \in \mathcal{T}_h .
\end{aligned}
\end{align*}
\begin{rmrk}
The flux reconstruction procedure presented for both the state and adjoint equations relies solely on the computation of local quantities and is computationally inexpensive.
A great advantage of the discontinuous Galerkin framework is represented by the cheap algorithms to construct the equilibrated fluxes on an element-wise level as discussed by several authors, e.g. in \cite{MR2376644,MR2261011,MR3018142,MR3249368}. 
As previously remarked for the construction of the equilibrated fluxes in the case of conforming finite element discretizations, the local nature of the procedure allows the parallelization of the algorithm and the exploitation of modern parallel architectures.
\end{rmrk}

\subsubsection{Goal-oriented equilibrated fluxes error estimator}
\label{ref:goalDG}

We may now evaluate the term \eqref{eq:Eu} for the Neumann and Dirichlet problems by exploiting the information carried by \eqref{eq:aDG} and \eqref{eq:Fweak}. 
We recall that the symmetric weighted interior penalty discontinuous Galerkin method under analysis is adjoint consistent (cf. \cite{MR2882148}).
Owing to the continuity of $r_{\Omega,i}$ and $k_\Omega \nabla r_{\Omega,i} \cdot \mathbf{n}_e$ on all the edges $e$'s and adding the expression of the equilibrated fluxes \eqref{eq:eqFLuxState}, we obtain:
\begin{equation*}
\begin{aligned}
\widetilde{E}^h_{u,_i} \coloneqq & F_{\Omega,i}(r_{\Omega,i}) - a_{\Omega,i}^h(u_{\Omega,i}^h,r_{\Omega,i}) \\
= & \ \zeta_i \int_{\partial\mathcal{D}}{g r_{\Omega,i} \ ds} + (1-\zeta_i) \int_{\partial\mathcal{D}}{U_D (\gamma r_{\Omega,i} - k_\Omega \nabla r_{\Omega,i} \cdot \mathbf{n} ) ds} \\
& - \sum_{T \in \mathcal{T}_h}{\int_T{\Big( k_{\Omega} \nabla u_{\Omega,i}^h \cdot \nabla r_{\Omega,i} + u_{\Omega,i}^h r_{\Omega,i} \Big) d\mathbf{x}}} \\
& + \sum_{e \in \mathcal{E}_h^\mathcal{I}}{\int_{e}{ \llbracket u_{\Omega,i}^h \rrbracket k_{\Omega} \nabla r_{\Omega,i} \cdot \mathbf{n}_e \ ds}} \\
& + (1-\zeta_i) \sum_{e \in \mathcal{E}_h^\mathcal{B}}{ \int_{e}{ \llbracket u_{\Omega,i}^h \rrbracket k_{\Omega} \nabla r_{\Omega,i} \cdot \mathbf{n}_e \ ds}} \\
& + \sum_{T \in \mathcal{T}_h}{\int_T{\Big( \nabla \cdot \boldsymbol\sigma_{\Omega,i}^h + \pi_Z^\kappa u_{\Omega,i}^h \Big) r_{\Omega,i} \ d\mathbf{x}}} .
\end{aligned}
\label{eq:GOAL1dg}
\end{equation*}
We integrate by parts the last integral and we plug in the expressions \eqref{eq:internalFluxState}-\eqref{eq:boundaryFluxState} of the equilibrated fluxes for the state problems. It follows from the homogeneous Dirichlet condition fulfilled by the adjoint solution $r_{\Omega,D}$ on $\partial\mathcal{D}$ that
\begin{equation*}
\begin{aligned}
\widetilde{E}^h_{u,_i} = & \zeta_i \int_{\partial\mathcal{D}}{\Big( g - \pi_{W \cdot \mathbf{n}}^\kappa(g) \Big) r_{\Omega,i} \ ds} + (1-\zeta_i) \int_{\partial\mathcal{D}}{ (u_{\Omega,i}^h - U_D) k_{\Omega} \nabla r_{\Omega,i} \cdot \mathbf{n} \ ds} \\
& + \sum_{e \in \mathcal{E}_h^\mathcal{I}}{\int_{e}{ \Big( \llbracket u_{\Omega,i}^h \rrbracket k_{\Omega} \nabla r_{\Omega,i} \cdot \mathbf{n}_e  
+ \llbracket \boldsymbol\sigma_{\Omega,i}^h \cdot \mathbf{n}_e \rrbracket r_{\Omega,i} \Big) \ ds}} \\
& + \sum_{T \in \mathcal{T}_h}{\int_T{\Big( \pi_Z^\kappa u_{\Omega,i}^h - u_{\Omega,i}^h \Big) r_{\Omega,i} \ d\mathbf{x}}} \\
& - \sum_{T \in \mathcal{T}_h}{\int_T{\Big( \boldsymbol\sigma_{\Omega,i}^h + k_\Omega \nabla u_{\Omega,i}^h \Big) \cdot \nabla r_{\Omega,i} \ d\mathbf{x}}} .
\end{aligned}
\label{eq:GOAL2dg}
\end{equation*}
We remark that owing to the continuity of the normal traces of the fluxes, $\llbracket \boldsymbol\sigma_{\Omega,i}^h \cdot \mathbf{n}_e \rrbracket = 0$ for all the internal edges. 
By adding and subtracting the terms $r_{\Omega,i}^h$'s featuring the discontinuous Galerkin approximations of the adjoint solutions and taking into account their equilibrated fluxes $\boldsymbol\xi_{\Omega,i}^h$'s, the expressions of the errors $\widetilde{E}^h_{u,_i}$'s read as:
\begin{equation}
\begin{aligned} 
\widetilde{E}^h_{u,_i} = & \ \zeta_i \int_{\partial\mathcal{D}}{\Big( g - \pi_{W \cdot \mathbf{n}}^\kappa(g) \Big) r_{\Omega,i}^h \ ds} 
+ \zeta_i \int_{\partial\mathcal{D}}{\Big( g - \pi_{W \cdot \mathbf{n}}^\kappa(g) \Big)(r_{\Omega,i} - r_{\Omega,i}^h) ds} \\
& - (1-\zeta_i) \int_{\partial\mathcal{D}}{ (u_{\Omega,i}^h - U_D) \boldsymbol\xi_{\Omega,i}^h \cdot \mathbf{n} \ ds} \\
& + (1-\zeta_i) \int_{\partial\mathcal{D}}{ (u_{\Omega,i}^h - U_D) \Big(k_\Omega \nabla r_{\Omega,i} + \boldsymbol\xi_{\Omega,i}^h \Big) \cdot \mathbf{n} \ ds} \\
& - \sum_{e \in \mathcal{E}_h^\mathcal{I}}{\int_{e}{ \llbracket u_{\Omega,i}^h \rrbracket \boldsymbol\xi_{\Omega,i}^h \cdot \mathbf{n}_e \ ds}}
+ \sum_{e \in \mathcal{E}_h^\mathcal{I}}{\int_{e}{ \llbracket u_{\Omega,i}^h \rrbracket \Big(k_\Omega \nabla r_{\Omega,i} + \boldsymbol\xi_{\Omega,i}^h \Big) \cdot \mathbf{n}_e \ ds}} \\
& + \sum_{T \in \mathcal{T}_h}{\int_T{\Big( \pi_Z^\kappa u_{\Omega,i}^h - u_{\Omega,i}^h \Big) r_{\Omega,i}^h \ d\mathbf{x}}} \\
& + \sum_{T \in \mathcal{T}_h}{\int_T{\Big( \pi_Z^\kappa u_{\Omega,i}^h - u_{\Omega,i}^h \Big)(r_{\Omega,i} - r_{\Omega,i}^h) d\mathbf{x}}} \\
& + \sum_{T \in \mathcal{T}_h}{\int_T{\Big( \boldsymbol\sigma_{\Omega,i}^h + k_\Omega \nabla u_{\Omega,i}^h \Big) \cdot k_\Omega^{-1} \boldsymbol\xi_{\Omega,i}^h \ d\mathbf{x}}} \\
& - \sum_{T \in \mathcal{T}_h}{\int_T{\Big( \boldsymbol\sigma_{\Omega,i}^h + k_\Omega \nabla u_{\Omega,i}^h \Big) \cdot \Big(\nabla r_{\Omega,i} + k_\Omega^{-1} \boldsymbol\xi_{\Omega,i}^h\Big) d\mathbf{x}}} .
\end{aligned}
\label{eq:GOAL3dg}
\end{equation}
As already remarked in the estimator derived for the conforming finite element discretization, both the unknown exact solutions of the adjoint problems and their numerical counterparts appear in \eqref{eq:GOAL3dg}. 
Let $I_m^\ell: V_\Omega^{h,m} \rightarrow V_\Omega^{h,\ell}$ be the projection operator from the space of high-order discontinuous Galerkin approximations to the low-order one. 
The fully-computable version of the estimator of the quantity $\widetilde{E}^h_{u,_i}$ is obtained by substituting $r_{\Omega,i} - r_{\Omega,i}^h$ with $r_{\Omega,i}^h - I_m^\ell r_{\Omega,i}^h$ and $\nabla r_{\Omega,i}$ with $\nabla r_{\Omega,i}^h$. \\
Eventually, the upper bound $\overline{E}$ of the error in the shape gradient is obtained by plugging the expressions of $\widetilde{E}^h_{u,_N}$ and $\widetilde{E}^h_{u,_D}$ arising from \eqref{eq:GOAL3dg} into \eqref{eq:ERRvolumetricKV} and by considering its absolute value.
\begin{rmrk}
In \cite{MR3327021}, the authors prove that the contribution of the terms in \eqref{eq:GOAL3dg} featuring the exact solution of the adjoint problems is negligible and the goal-oriented error estimator 
constructed using the previously described equilibrated fluxes approach is asymptotically exact. 
This property guarantees that the bound $\overline{E}$ of the error in the shape gradient tends to zero by reducing the mesh size. Hence, the mesh adaptation procedure performed by the certified descent algorithm eventually leads to the fulfillment of condition \eqref{eq:certified}.
\end{rmrk}

\section{Numerical results}
\label{ref:numerics}

In this section we present some numerical results of the application of the certified descent algorithm based on the equilibrated fluxes approach for the estimation of the error in the shape gradient.
We consider the problem of electrical impedance tomography as a proof of concept to establish some properties of this variant of the certified descent algorithm on a non-trivial scalar test case.
Shape optimization methods are known to provide poor reconstructions in inverse ill-posed problems as the EIT. 
Within this framework, the certified descent algorithm does not aim to remedy the issue of local minima but may act as a counterexample confirming the limitations of gradient-based strategies when dealing with ill-posed problems.
The current work presents an improvement of the original certified descent algorithm introduced in \cite{giacomini:hal-01201914}, in particular since it uses solely local quantities to compute the error in the shape gradient required by the certification procedure.
The numerical results in this section focus on the quantitative bound $\overline{E}$ obtained using the equilibrated fluxes approach for both conforming finite element and discontinuous Galerkin discretizations.

The simulations are obtained using FreeFem++ \cite{MR3043640} and are based on a mesh moving approach for the deformation of the domain.
It is well-known in the literature (cf. e.g. \cite{smo-AP}) that numerical shape optimization may result in poor optimal shapes showing high-frequency oscillations of the boundaries whose length scale is comparable with the mesh size \cite{giacomini:hal-01201914}.
In order to remedy these issues of regularity of the mesh, several strategies have been proposed. 
A possible workaround relies on introducing a regularizing term in the cost functional through a perimeter penalization \cite{MR2270119}. Nonetheless, this strategy strongly depends on the weight of the penalty parameter which may be difficult to tune. 
In order to construct a more general and automatic optimization algorithm, we resort to a two-mesh strategy \cite{smo-AP} which explicitly smoothes the boundary of the shape at each iteration: after solving the state problems using a fine mesh in order to properly capture all the important features of the solutions, a coarser mesh is extracted and a descent direction is computed; thus, the nodes of the coarser mesh are displaced and a novel fine mesh is obtained from the coarser one via a mesh adaptation procedure.
This latter step may be performed either through a uniform mesh refinement or by means of an adaptation routine that exploits the information of the last computed solution of the state problem.
We remark that by deforming the coarse mesh, oscillations of the novel boundary are prevented and consequently a regularization is directly introduced into the problem.
Being this strategy dependent solely on the projection of the information of the finite element solution of the state problems from the fine mesh to the coarse one, the overall strategy results in numerically more stable computations and more regular optimal shapes without introducing any additional parameter \cite{smo-AP}. \\
Within the previously described framework, changes in the topology of the shape are not allowed and the correct number of inclusions has to be set at the beginning of the algorithm and remains the same throughout its evolution.
Techniques based on both topological and shape derivatives to account for topological changes inside the domain have been investigated e.g. in \cite{Hintermuller2008}.

\subsection{Numerical assessment of the goal-oriented estimator}
\label{ref:validation}

In order to evaluate the goal-oriented error estimators derived in sections \ref{ref:goalFE} and \ref{ref:goalDG}, we consider a configuration for which the analytical solution of the state problems is known. 
We introduce the polar coordinate system $(\rho,\vartheta)$ and we set $\mathcal{D} \coloneqq \{ \mathbf{x}=(x,y) \ | \ x^2+y^2 \leq \rho_E^2 \}$ and $\Omega \coloneqq \{ \mathbf{x}=(x,y) \ | \ x^2+y^2 \leq \rho_I^2 \}$ with $\rho_I=4$ and $\rho_E=5$. 
The value of the conductivity parameter is $k_I=10$ inside $\Omega$ and $k_E=1$ in $\mathcal{D} \setminus \Omega$.
We consider the Neumann boundary condition $g=\cos(M \vartheta) \ , \ M=5$ and the Dirichlet datum $U_D$ is the trace of the following function which is the analytical solution of problem \eqref{eq:statePB}:
$$
u_{\Omega,N} = \begin{cases}
                C_0 J_M \left(-i \rho k_I^{-\frac{1}{2}}\right)\cos(M \vartheta) 
\ & , \ \rho \in [0,\rho_I]\\
                \left[ C_1 J_M \left(-i \rho k_E^{-\frac{1}{2}}\right) + C_2 Y_M 
\left(-i \rho k_E^{-\frac{1}{2}}\right) \right] \cos(M \vartheta) \ & , \ \rho 
\in (\rho_I,\rho_E]
               \end{cases}
$$
where $J_M(\cdot)$ and $Y_M(\cdot)$ respectively represent the first- and second-kind Bessel functions of order $M$. 
The constants $C_0,\ldots,C_2$ are detailed in table \ref{tab:constants}.
\begin{table}
\centering
\begin{tabular}[hbt]{| c || l | l |}
\hline
Constant & $\mathbb{R}\text{e}[C_i]$ & $\mathbb{I}\text{m}[C_i]$ \\
\hline & &
\\ [-1em] \hline
$C_0$ & $-6.3 \cdot 10^{-9}$ & $+40.39491005$ \\
\hline
$C_1$ & $+1.30145994$ & $+0.325482825$ \\
\hline
$C_2$ & $+1.5 \cdot 10^{-11}$ & $-1.301459935$ \\
\hline
\end{tabular}
\caption{Constants for the analytical solution.}
\label{tab:constants}
\end{table}

We recall that for both the conforming finite element and the discontinuous Galerkin discretizations we have $\ell=1$ and $m=2$, that is the state problems are solved using functions of degree $1$, whereas the adjoint solutions are approximated using functions of degree $2$.
The corresponding equilibrated fluxes are sought respectively in the space of $RT_0$ and $RT_1$ finite element functions. \\
Figure \ref{fig:QoIrateFE} presents the convergence history of the discretization error in the shape gradient and the goal-oriented estimator $\overline{E}$ for the case of conforming finite element. 
The corresponding quantities for the case of discontinuous Galerkin are depicted in figure \ref{fig:QoIrateDG}. 
The analytical error is computed by substituting in \eqref{eq:errorH} the expression \eqref{eq:volumetricKV} evaluated using respectively the analytical solution introduced at the beginning of this subsection and the numerical solutions arising from the conforming finite element and the discontinuous Galerkin discretizations. \\
Eventually, in figure \ref{fig:effIdx} we present the effectivity indices for the discussed discretizations. 
The effectivity index $\eta$ is defined as the ratio between the estimator and the exact error, that is $\eta \coloneqq \overline{E}/E^h$.
If the effectivity index is bigger (respectively smaller) than $1$, one is overestimating (respectively underestimating) the error. 
The evolution of the effectivity indices in figure \ref{fig:effIdx} confirms that the constructed estimators are guaranteed - that is they provide an upper bound of the error since $\eta > 1$ - and are asymptotically exact since the $\eta \searrow 1$ as the mesh size tends to $0$.
\begin{figure}[htb]
    \subfloat[Conforming finite element.]
    {
    \includegraphics[width=0.32 \columnwidth]{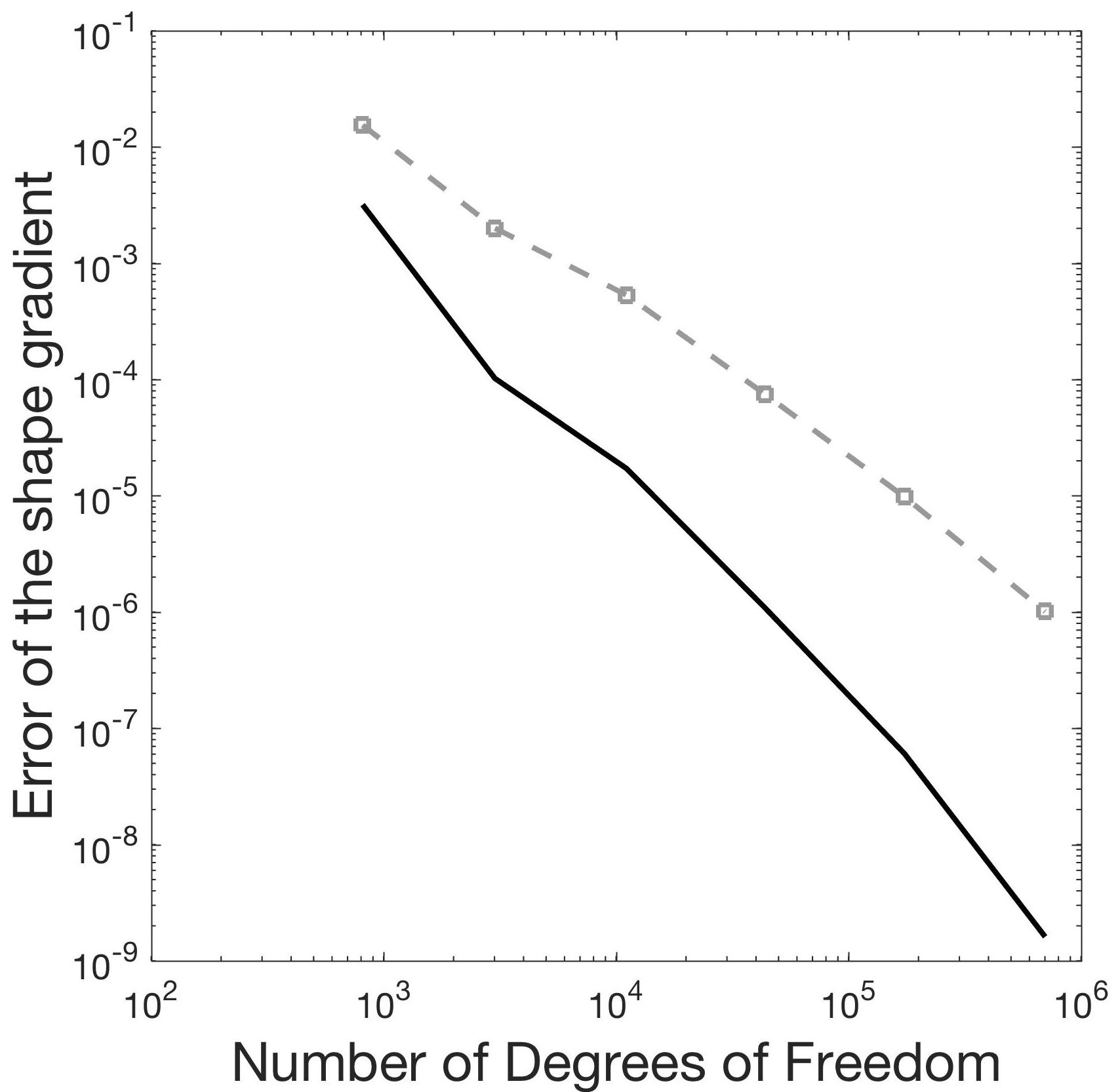}
    \vspace{10pt}
    \label{fig:QoIrateFE}
    }
    \subfloat[Discontinuous Galerkin.]
    {
    \includegraphics[width=0.32 \columnwidth]{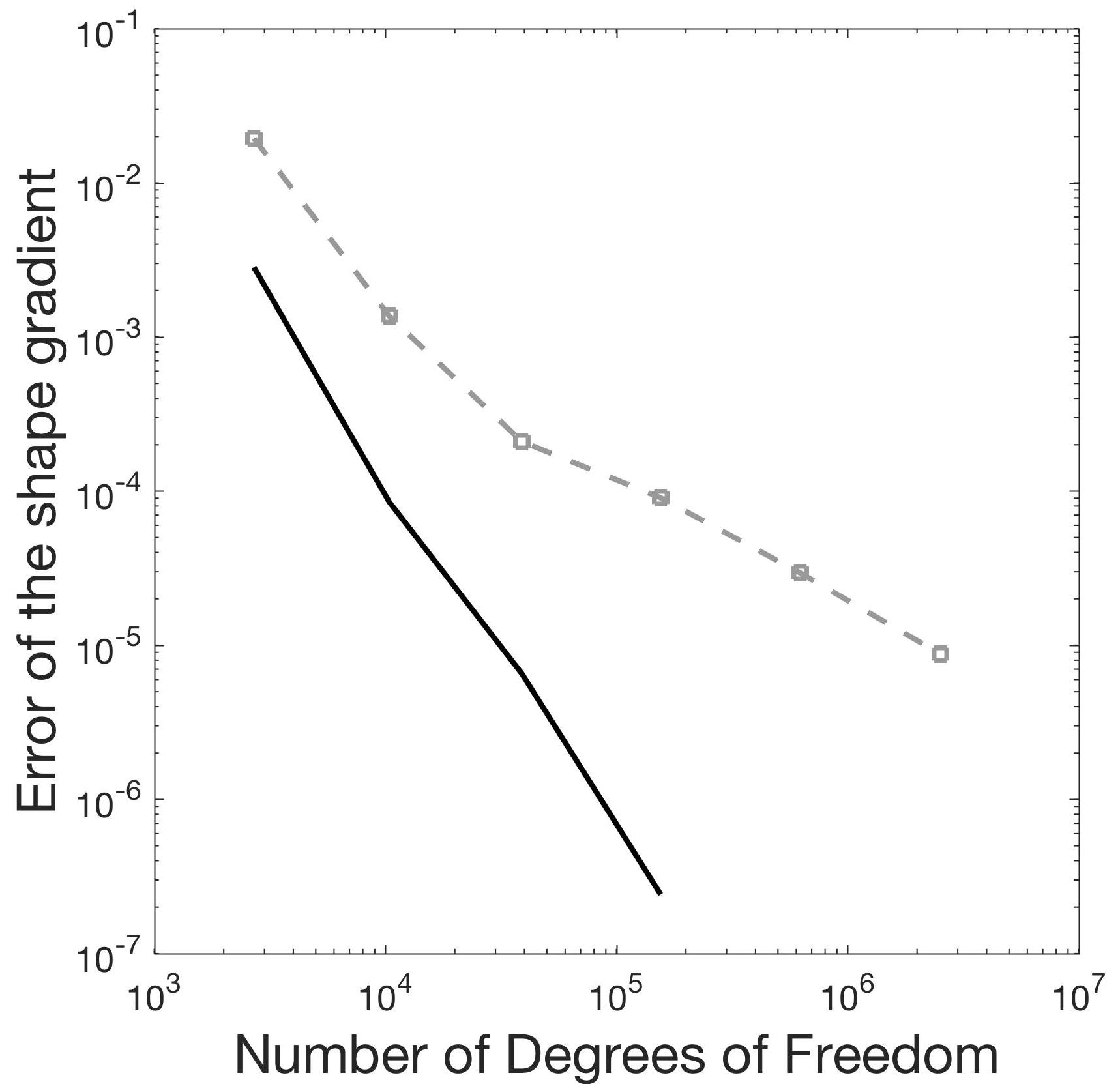}
    \label{fig:QoIrateDG}
    }
    \subfloat[Effectivity index.]
    {
    \includegraphics[width=0.32 \columnwidth]{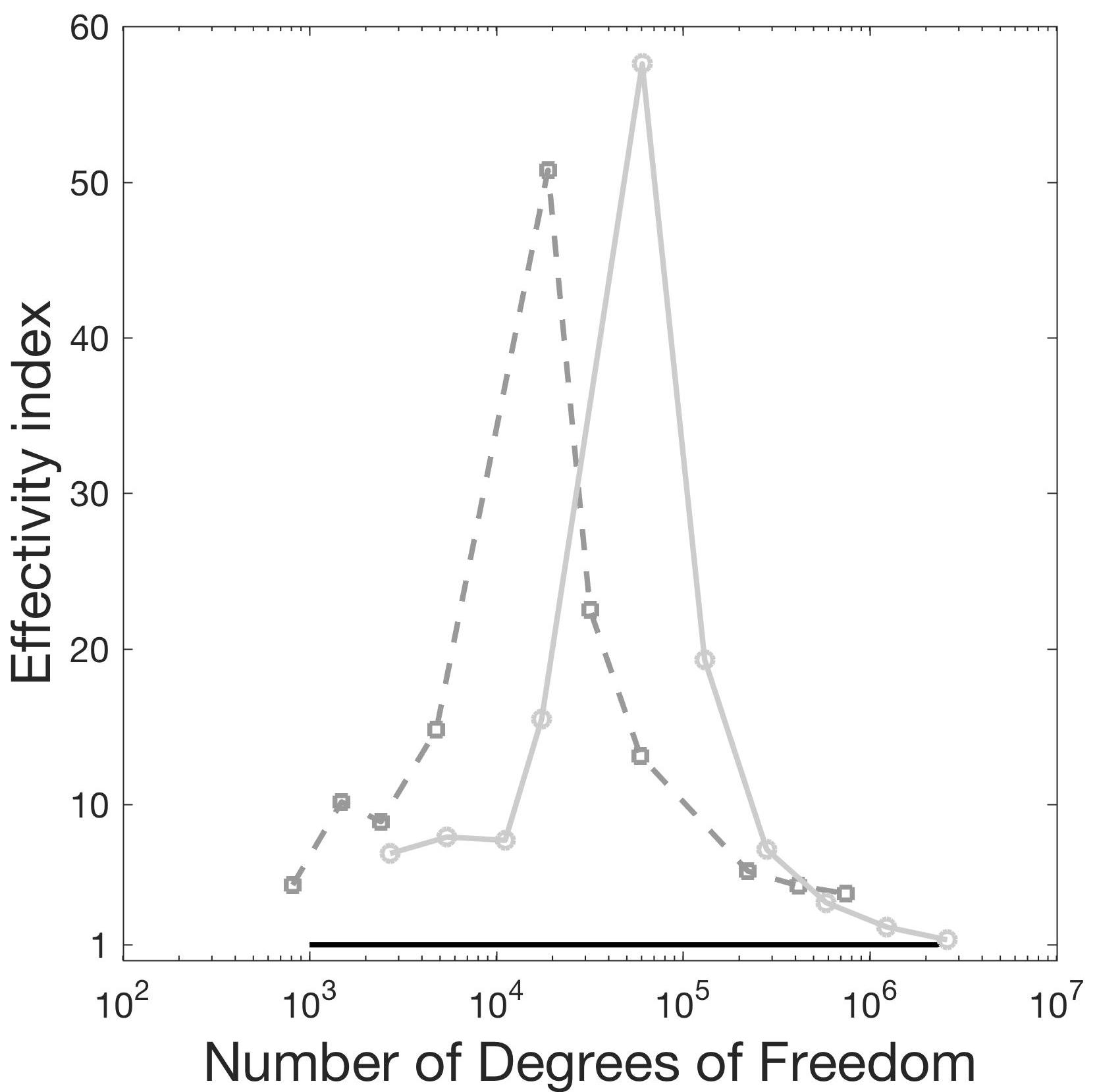}
    \label{fig:effIdx}
    }
    \caption{Convergence rates and effectivity indices of the estimators 
    of the error in the shape gradient with respect to the number of 
    degrees of freedom.  
    Analytical error in the shape gradient (solid black);
    goal-oriented estimator of the error based on the equilibrated fluxes (dashed gray squares) 
    for the discretizations based on 
    (a) conforming finite element and (b) discontinuous Galerkin.     
    (c) Effectivity indices for the conforming finite element (dark gray squares) and 
    discontinuous Galerkin (light gray circles).
    }
    \label{fig:convergences}
\end{figure}

\subsection{Reconstruction of a single inclusion}
\label{ref:1incl}

We consider the problem of reconstructing the inclusion $\Omega$ defined in the previous section by means of a couple $(g,U_D)$ of measurements on the external boundary $\partial\mathcal{D}$. 
In the following simulations, we consider a stopping criterion that combines the condition in step 8 of algorithm \ref{scpt:shape-opt-adaptive} and a bound on the number of admissible mesh elements - i.e. the size of the state and the adjoint problems.
This choice is due to the ill-posed nature of the electrical impedance tomography problem that we chose as test case for the certified descent algorithm. 
As we will highlight throughout this section the ill-posedness of the problem represents an issue that prevents gradient-based strategies from efficiently solving the EIT problem since a huge precision is demanded after few iterations of the optimization procedure. 

First, we consider the configuration described in figure \ref{fig:circleInterface}. The initial guess for the inclusion is represented by the circle of radius $\rho_{\text{ini}}=2$.
The certified descent algorithm is able to correctly identify the interface along which the conductivity parameter $k_\Omega$ is discontinuous (Fig. \ref{fig:circleInterface}). 
Moreover, figure \ref{fig:circleObj} shows that the objective functional $J(\Omega)$ is monotonically decreasing, meaning a genuine descent direction is computed at each 
iteration of the algorithm.
In tables \ref{tab:circleConvFE}-\ref{tab:circleConvDG} we present the specifics of the meshes used to certify the descent direction at several iterations of the CDA whereas in figure \ref{fig:meshCircle} different meshes generated by the algorithm for conforming finite element and discontinuous Galerkin discretizations are proposed. 
In particular, we observe that coarse meshes are reliable during the initial iterations of the algorithm to identify a genuine descent direction and the number of degrees of freedom increases when approaching a 
minimum of the functional $J(\Omega)$. 
This is also well-explained by figure \ref{fig:circleDOF} in which the evolution of the number of degrees of freedom is depicted.
Concerning the quality of the computational meshes, they are mainly uniform during the initial iterations and are refined in the regions where the numerical error in the shape gradient is more important. 
In particular, finer - and possibly anisotropic - elements tend to concentrate near the external boundary where the measurements are located and near the internal interface where the conductivity parameter is discontinuous (Fig. \ref{fig:meshCircle20FE}-\ref{fig:meshCircle20DG}).
Eventually, we remark that figure \ref{fig:circleDOF} also highlights the ill-posed nature of the problem since a huge amount of degrees of freedom is rapidly required by the CDA to certify the descent direction, testifying the difficulties of gradient-based methods to handle inverse problems as the electrical impedance tomography.
By comparing the approximations arising from conforming finite element and discontinuous Galerkin formulations, we remark that the latter provides sharper bounds of the error in the shape gradient thus allowing the algorithm to automatically stop for a given tolerance $\texttt{tol}=10^{-6}$ (cf. table \ref{tab:circleConvDG}). 
On the contrary, the certification in the case of conforming finite element is still able to identify a genuine descent direction at each iteration but rapidly requires a huge number of mesh elements making the 
computational cost explode.
\begin{figure}[htbp]
    \subfloat[Reconstructed interface.]
    {
    \includegraphics[width=0.28 \columnwidth]{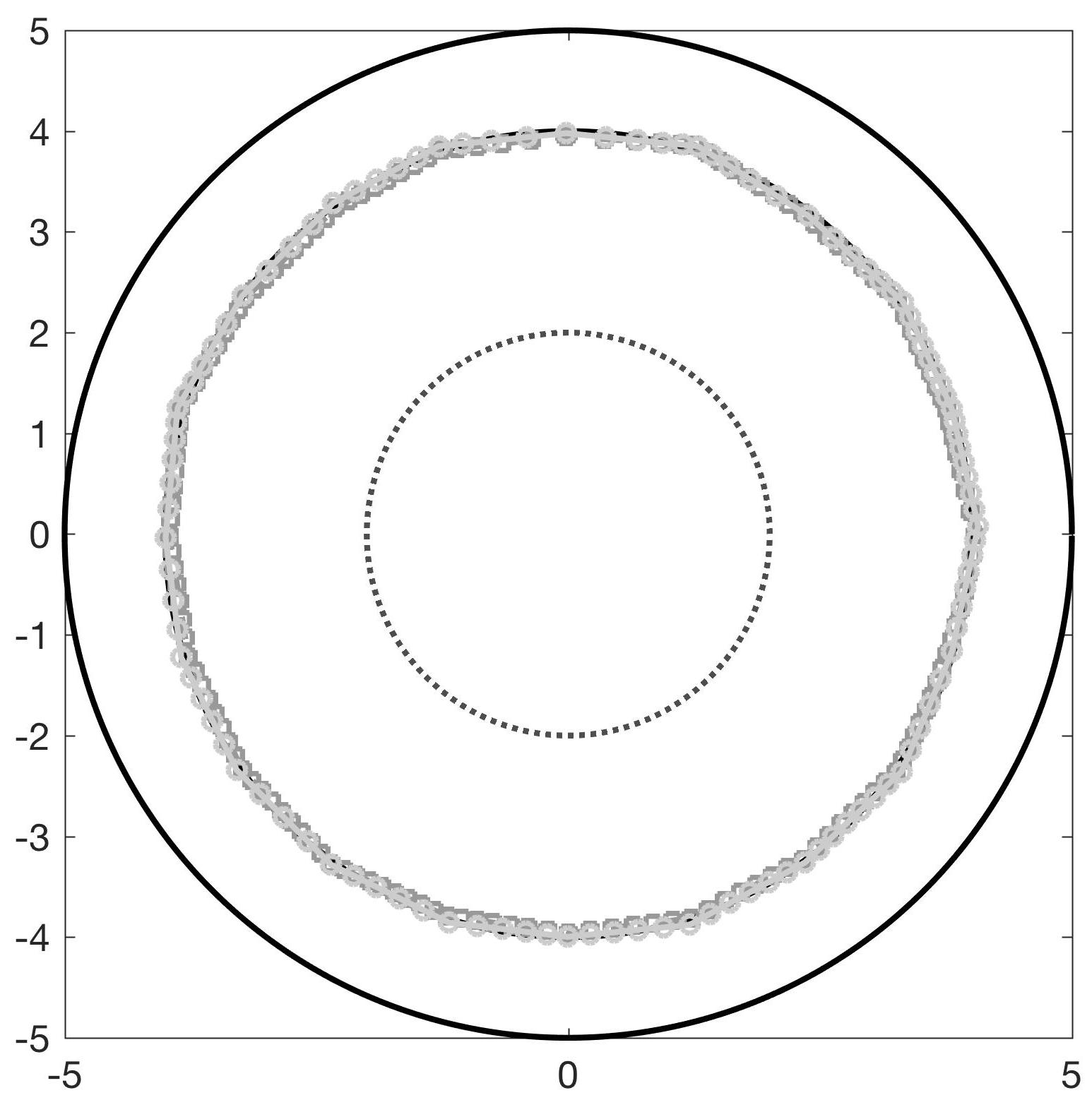}
    \vspace{10pt}
    \label{fig:circleInterface}
    }
    \hfil
    \subfloat[Objective functional.]
    {
    \includegraphics[width=0.33 \columnwidth]{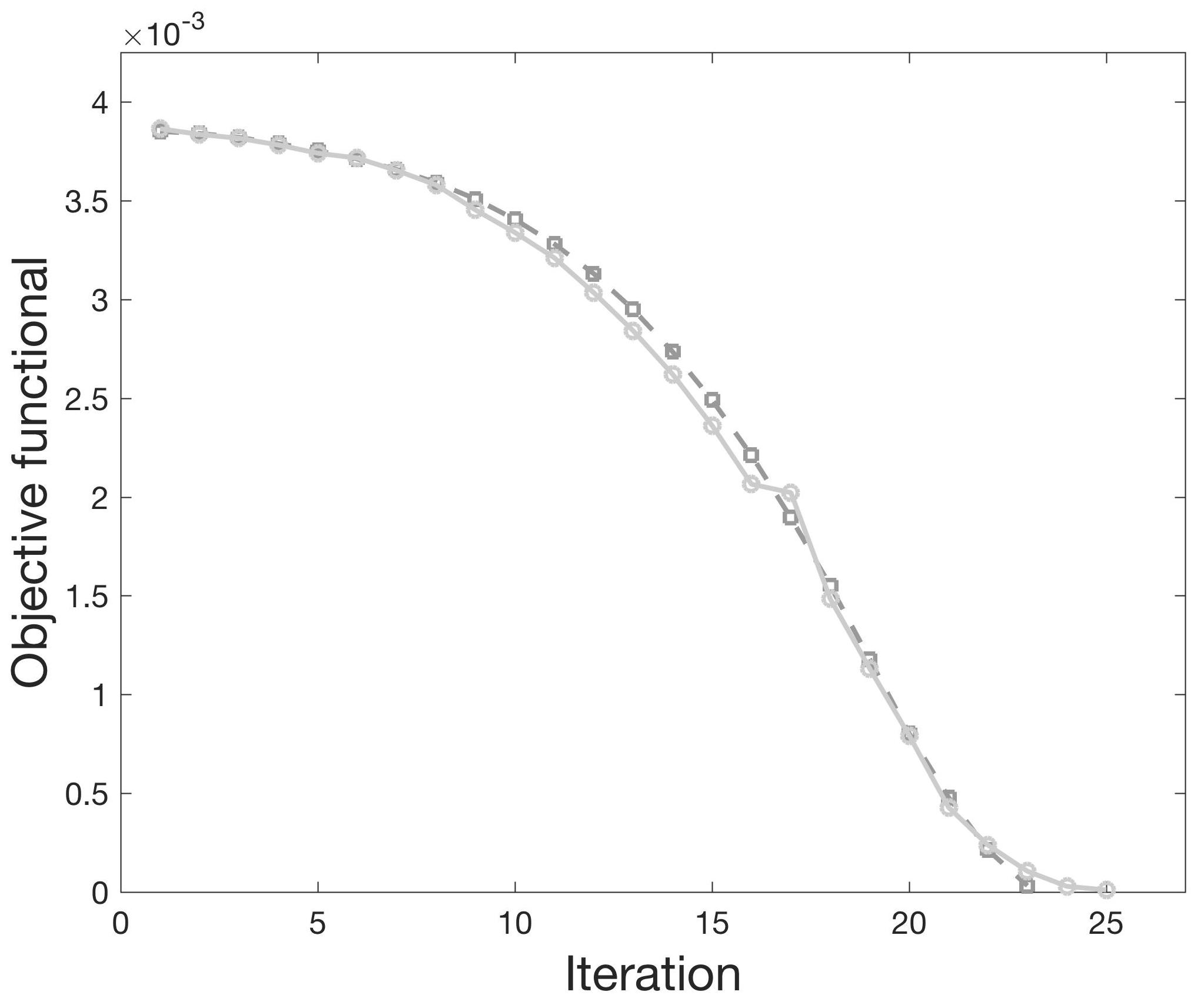}
    \label{fig:circleObj}
    }
    \hfil
    \subfloat[Degrees of freedom.]
    {
    \includegraphics[width=0.33 \columnwidth]{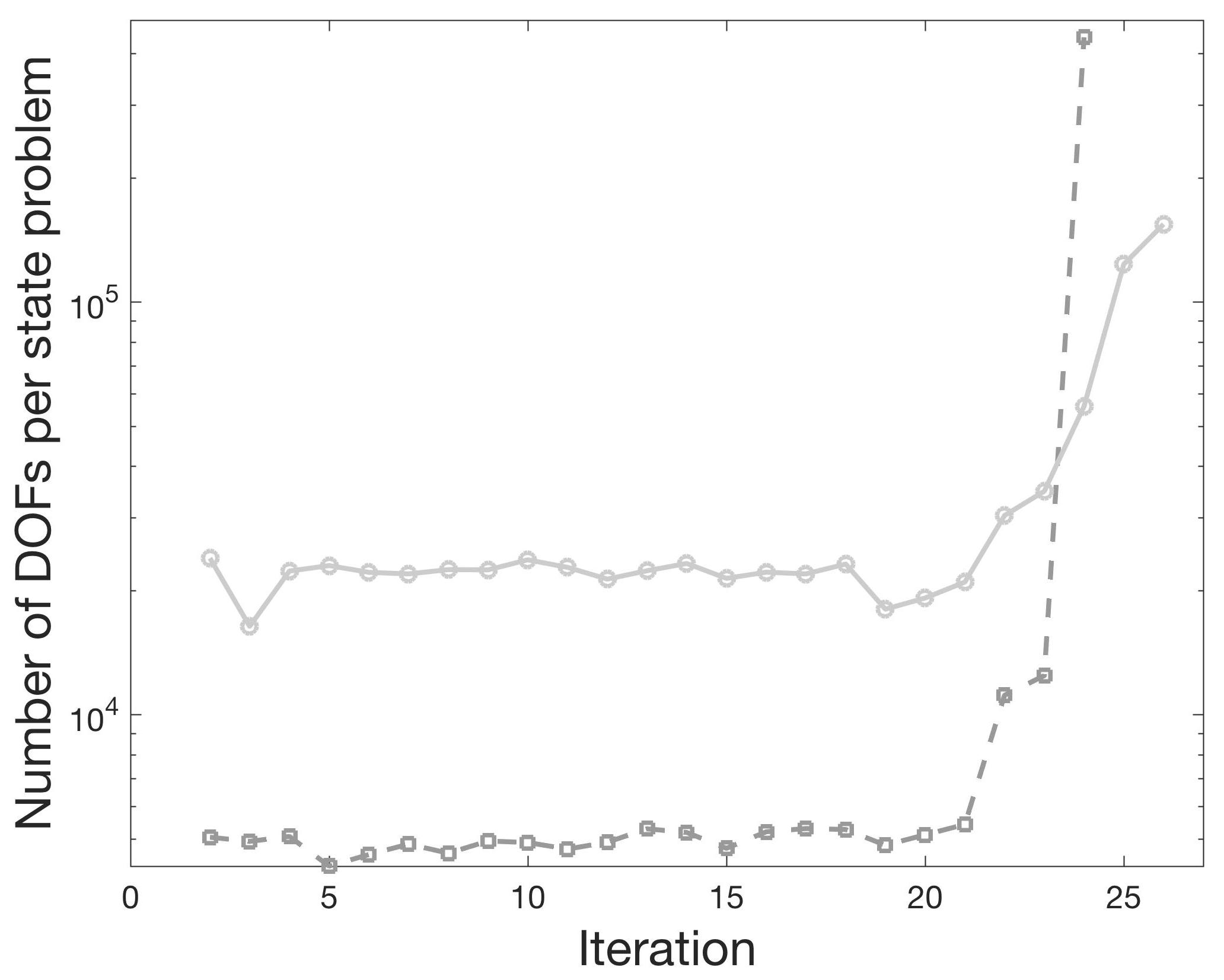}
    \label{fig:circleDOF}
    }
    \caption{Certified descent algorithm for the identification of one inclusion. 
    (a) Initial configuration (dotted black), target inclusion (solid black) and 
    reconstructed interface. (b) Evolution of the objective functional. 
    (c) Number of degrees of freedom. 
    Inversion performed using conforming finite element (dark gray squares) and 
    discontinuous Galerkin (light gray circles).}
    \label{fig:circle}
\end{figure}
\begin{figure}[hbtp]
	\centering
    \subfloat[Iteration \#10.]
    {
    \includegraphics[width=0.35 \columnwidth]{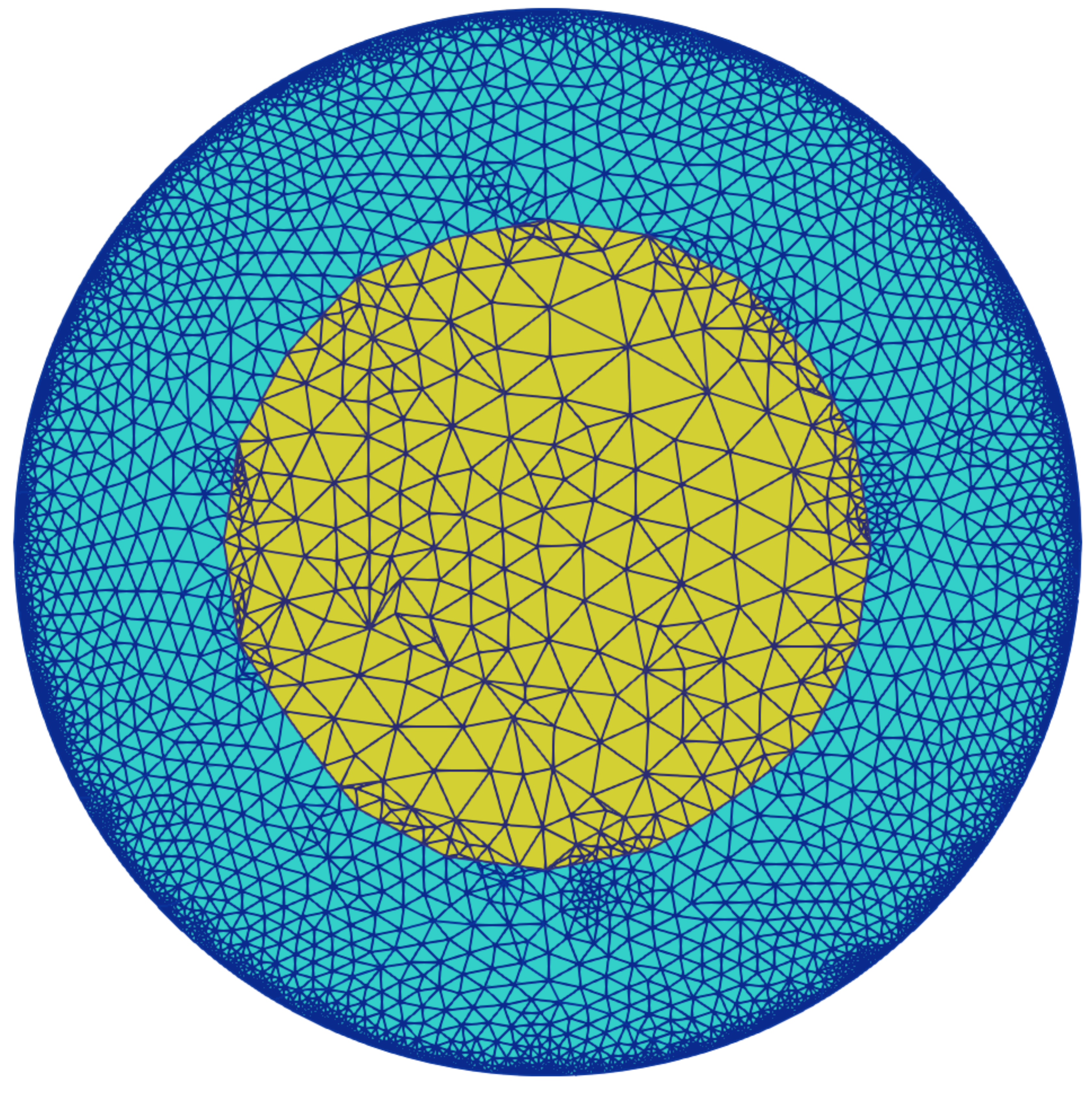}
    \vspace{10pt}
    \label{fig:meshCircle10FE}
    }
    \hfil
    \subfloat[Iteration \#20.]
    {
    \includegraphics[width=0.35 \columnwidth]{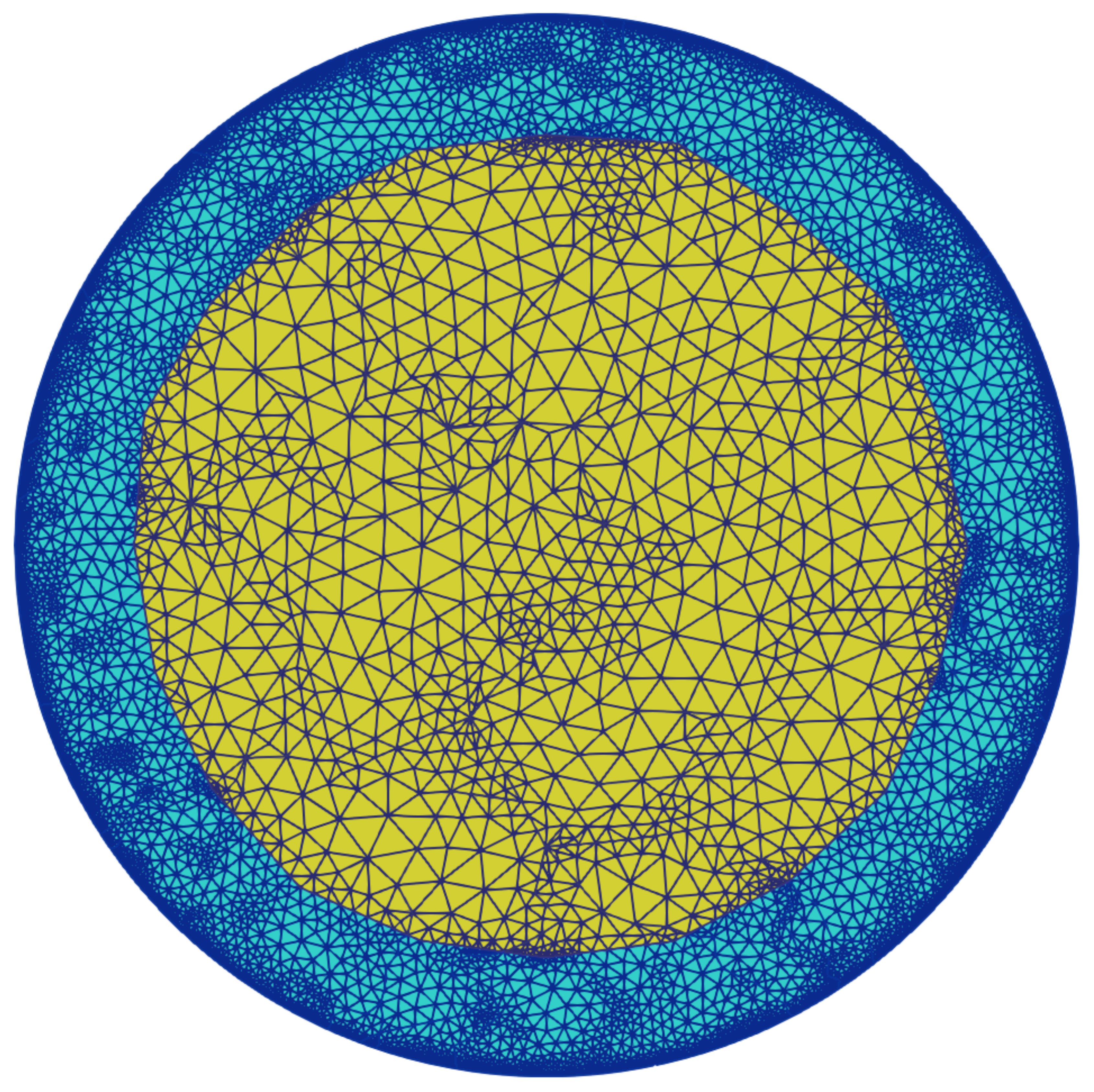}
    \label{fig:meshCircle20FE}
    }
    
    \subfloat[Iteration \#10.]
    {
    \includegraphics[width=0.35 \columnwidth]{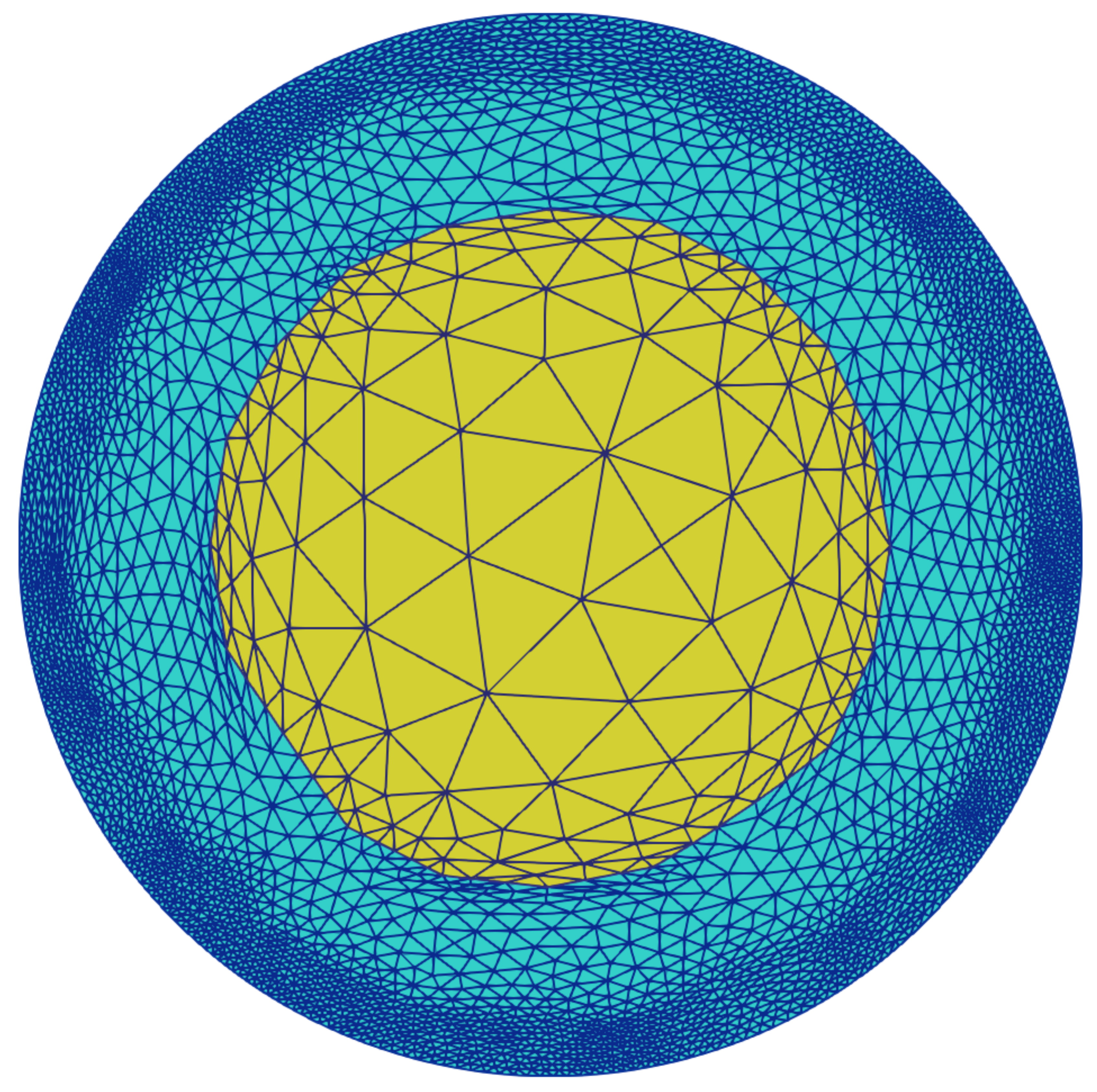}
    \vspace{10pt}
    \label{fig:meshCircle10DG}
    }
    \hfil
    \subfloat[Iteration \#20.]
    {
    \includegraphics[width=0.35 \columnwidth]{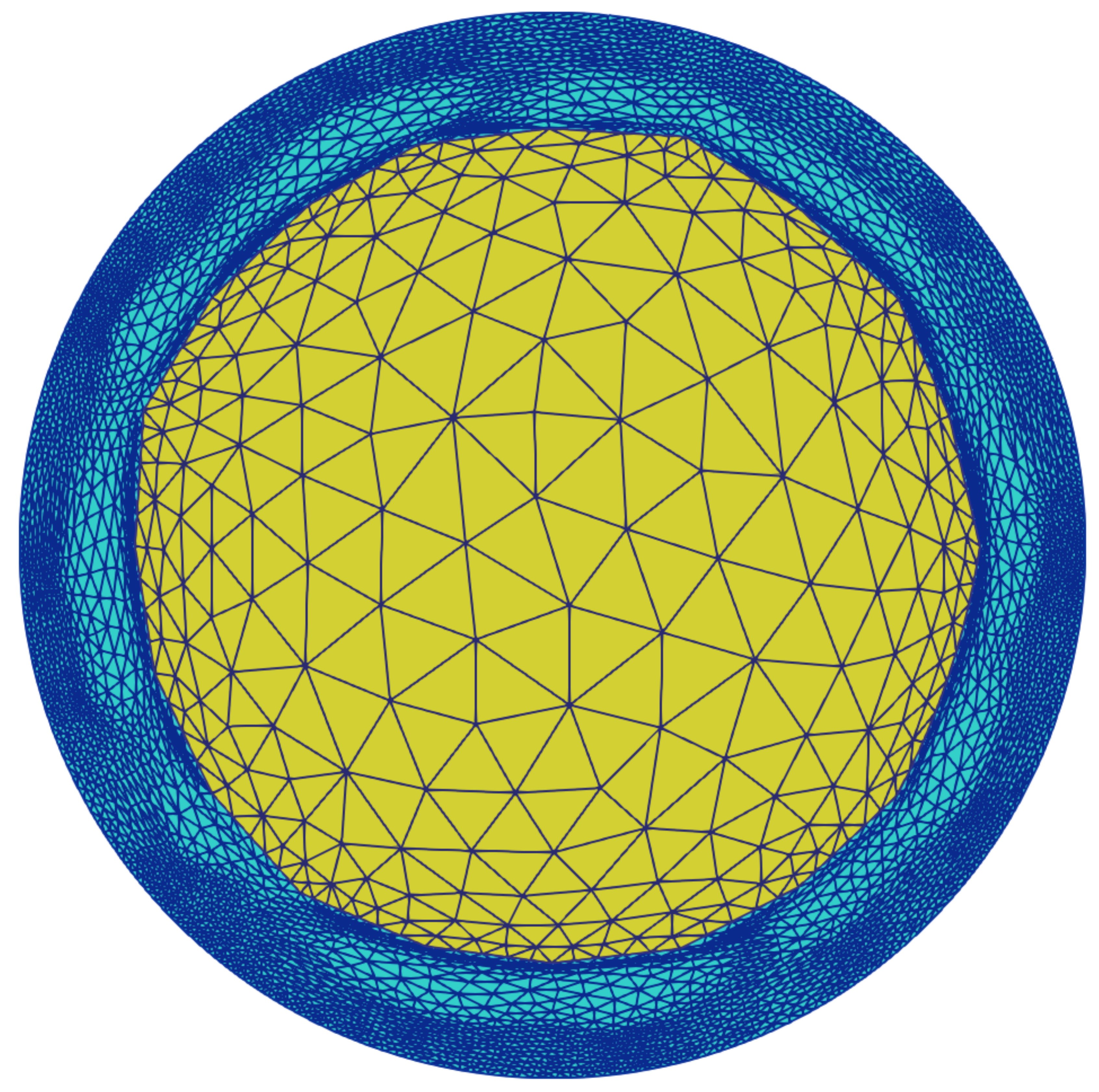}
    \label{fig:meshCircle20DG}
    }
    \caption{Meshes generated by the certified descent algorithm for the test case in figure \ref{fig:circleInterface}.
    Top: conforming finite element. Bottom: discontinuous Galerkin.}
    \label{fig:meshCircle}
\end{figure}
\begin{table}[htb]
    \centering
    \subfloat[Conforming finite element.]
    {
    \centering
    \begin{tabular}[hbt]{| c | c || l | l |}
    \hline
    Iteration & $\# \mathcal{T}_h$ & $\langle d_h J(\Omega),\boldsymbol\theta^h \rangle$ & $\overline{E}$ \\
    \hline & & & 
    \\ [-1em] \hline
    1 & 8863 & $-1.45 \cdot 10^{-6}$ & $1.12 \cdot 10^{-6}$ \\
    \hline
    5 & 8582 & $-4.36 \cdot 10^{-6}$ & $3.31 \cdot 10 ^{-6}$ \\
    \hline
    10 & 8650 & $-1.37 \cdot 10^{-5}$  & $9.17 \cdot 10 ^{-6}$ \\
    \hline
    15 & 9335 & $-2.83 \cdot 10^{-5}$ & $1.80 \cdot 10^{-5}$ \\
    \hline
    20 & 19683 & $-1.53 \cdot 10^{-5}$ & $1.07 \cdot 10^{-5}$ \\
    \hline
    22 & 864808 & $-1.18 \cdot 10^{-6}$ & $1.16 \cdot 10^{-6}$ \\
    \hline
    \end{tabular}
    \label{tab:circleConvFE}
    }
    \hfil
    \subfloat[Discontinuous Galerkin.]
    {
    \centering
    \begin{tabular}[hbt]{| c | c || l | l |}
    \hline
    Iteration & $\# \mathcal{T}_h$ & $\langle d_h J(\Omega),\boldsymbol\theta^h \rangle$ & $\overline{E}$ \\
    \hline & & & 
    \\ [-1em] \hline
    1 & 5454 & $-2.02 \cdot 10^{-6}$ & $1.86 \cdot 10^{-6}$ \\
    \hline
    5 & 7307 & $-6.63 \cdot 10^{-6}$ & $6.24 \cdot 10^{-6}$ \\
    \hline
    10 & 7099 & $-1.73 \cdot 10^{-5}$ & $9.20 \cdot 10^{-6}$ \\
    \hline
    15 & 7307 & $-3.06 \cdot 10^{-5}$ & $7.11 \cdot 10^{-6}$ \\
    \hline
    20 & 10123 & $-1.38 \cdot 10^{-5}$ & $8.98 \cdot 10^{-6}$ \\
    \hline
    24 & 51406 & $-4.60 \cdot 10^{-7}$ & $4.55 \cdot 10^{-7}$ \\
    \hline
    \end{tabular}
    \label{tab:circleConvDG}
    }
\caption{Test case in figure \ref{fig:circleInterface} using (a) conforming finite element 
and (b) discontinuous Galerkin. 
Approximated shape gradient and goal-oriented estimator for different meshes.}
\label{tab:circleConv}
\end{table}

The aforementioned remarks are confirmed and highlighted by the test case in figure \ref{fig:ellipseInterface}. 
On the one hand, it is straightforward to observe that the certified descent algorithm is able to identify a genuine descent direction at each iteration (Fig. \ref{fig:ellipseObj}).
Nevertheless, as extensively discussed in \cite{giacomini:hal-01201914}, the well-known ill-posedness of the electrical impedance tomography problem prevents the certified descent algorithm - and in general, gradient-based strategies - to accurately identify the inclusion in the whole domain.
In figure \ref{fig:ellipseInterface} we observe that the portion of the interface closest to $\partial\mathcal{D}$ is well identified but the precision of the reconstruction decreases moving away from the external boundary and towards the inner part of the domain.
These observations match the results in \cite{161105272}, where the authors remark that a good approximation of the boundary is obtained for the upwind part of the shape whereas a loss of accuracy is observed in its downwind part. 
We remark that the perimeter regularization exploited in the aforementioned work is substituted in the present research by the two-mesh strategy as discussed at the beginning of section \ref{ref:numerics}.
A possible workaround to the low resolution of the reconstruction in the center of the computational domain is represented by the emerging field of hybrid imaging in which classical tomography techniques are coupled with acoustic or elastic waves \cite{doi:10.1137/120863654}.
Nevertheless, these limitations are related to the nature of the electrical impedance tomography problem, especially to the decreasing importance of the boundary conditions when moving away from $\partial\mathcal{D}$ and we cannot expect gradient-based strategies to successfully overcome this issue.
These remarks are confirmed again by the rapidly exploding number of degrees of freedom required by the algorithm to certify the descent direction (Fig. \ref{fig:ellipseDOF}).
The non-feasibility of gradient-based approaches for severely ill-posed problems is testified by the huge precision required by the algorithm which only leads to a negligible improvement of the solution.
\begin{figure}[htbp]
    \subfloat[Reconstructed interface.]
    {
    \includegraphics[width=0.28 \columnwidth]{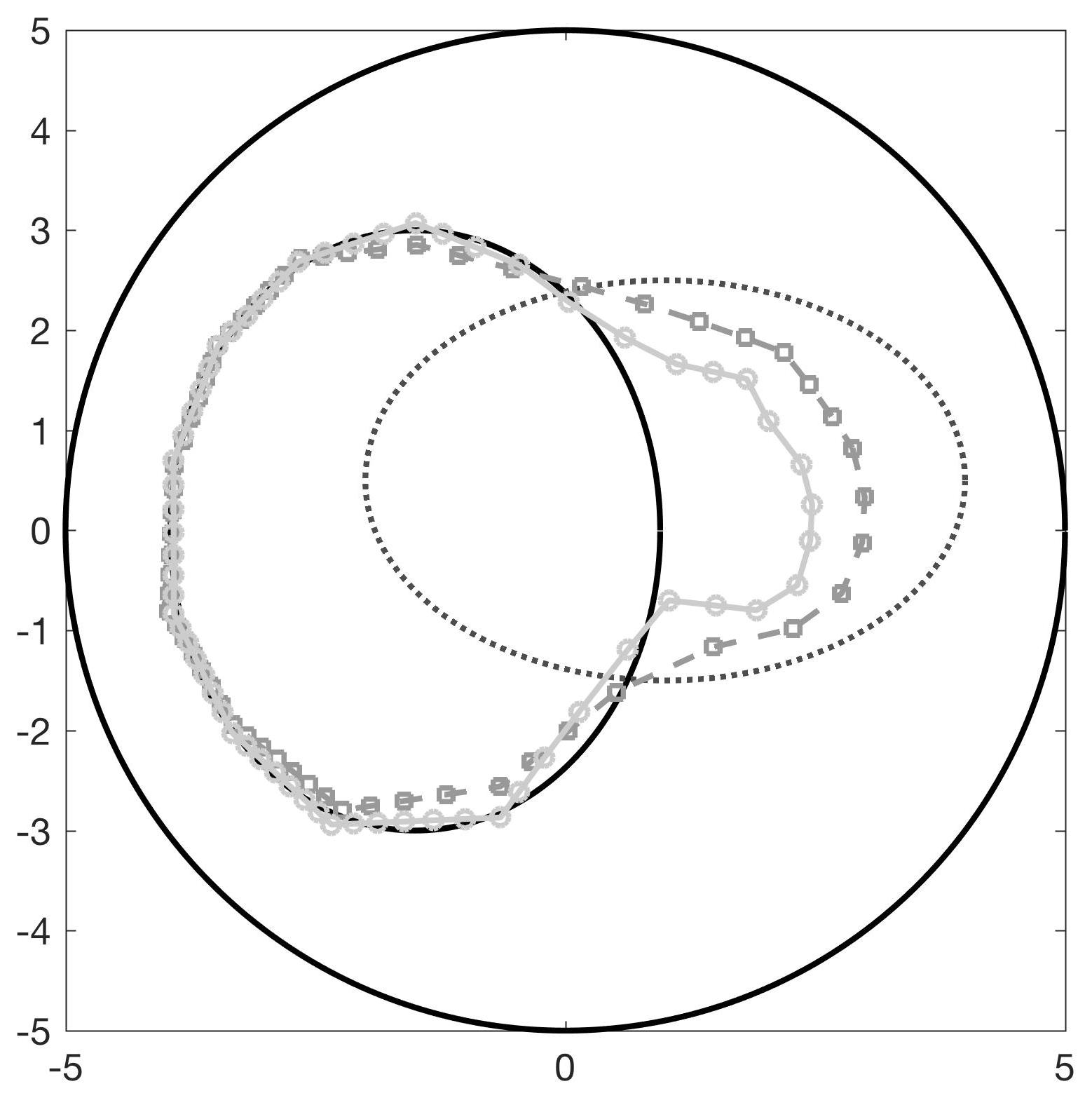}
    \vspace{10pt}
    \label{fig:ellipseInterface}
    }
    \hfil
    \subfloat[Objective functional.]
    {
    \includegraphics[width=0.33 \columnwidth]{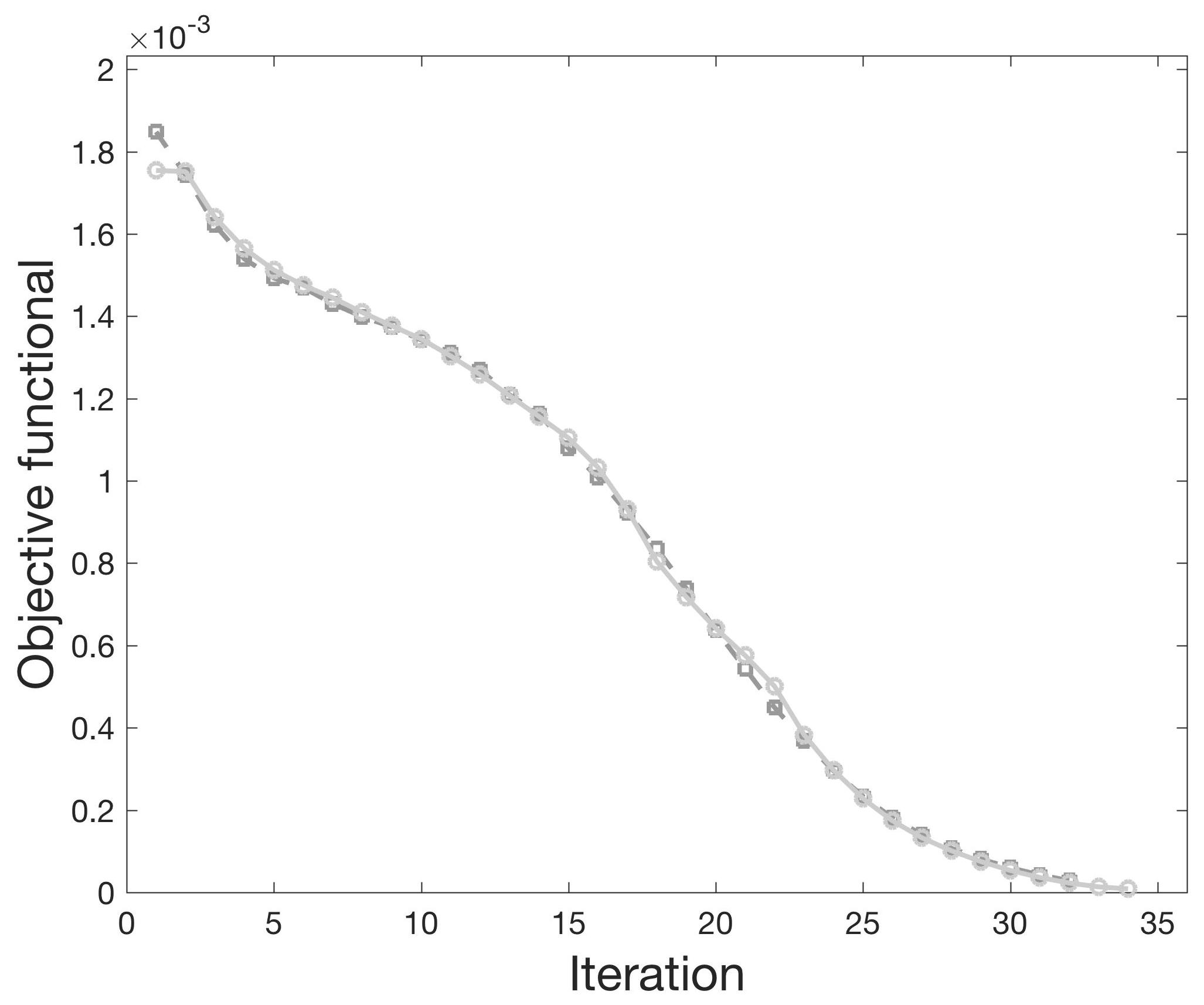}
    \label{fig:ellipseObj}
    }
    \hfil
    \subfloat[Degrees of freedom.]
    {
    \includegraphics[width=0.33 \columnwidth]{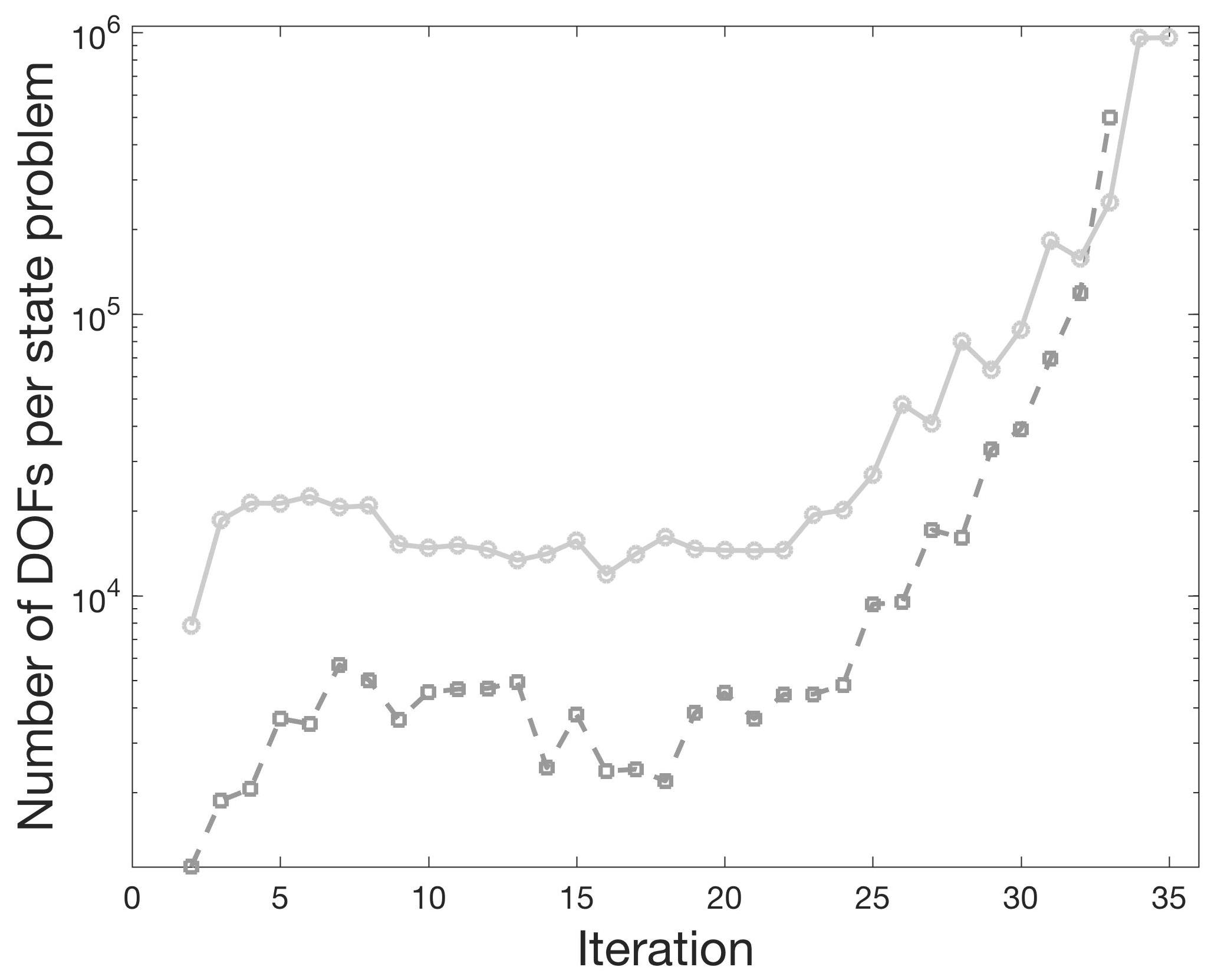}
    \label{fig:ellipseDOF}
    }
    \caption{Certified descent algorithm for the identification of one inclusion. 
    (a) Initial configuration (dotted black), target inclusion (solid black) and 
    reconstructed interface. (b) Evolution of the objective functional. 
    (c) Number of degrees of freedom. 
    Inversion performed using conforming finite element (dark gray squares) and 
    discontinuous Galerkin (light gray circles).}
    \label{fig:ellipse}
\end{figure}    
\begin{figure}[hbtp]
	\centering
    \subfloat[Iteration \#10.]
    {
    \includegraphics[width=0.35 \columnwidth]{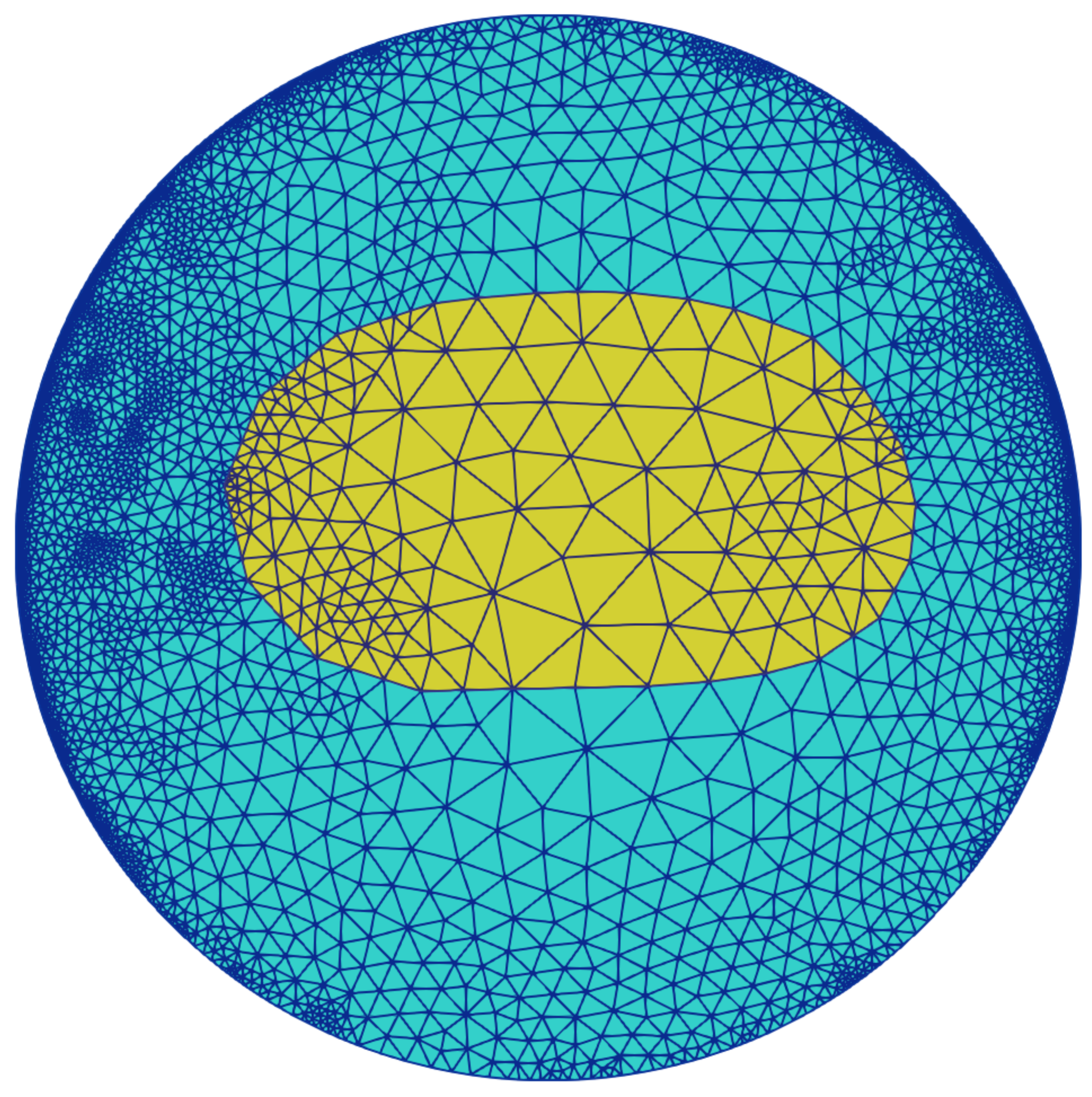}
    \vspace{10pt}
    \label{fig:meshEllipse10FE}
    }
    \hfil
    \subfloat[Iteration \#30.]
    {
    \includegraphics[width=0.35 \columnwidth]{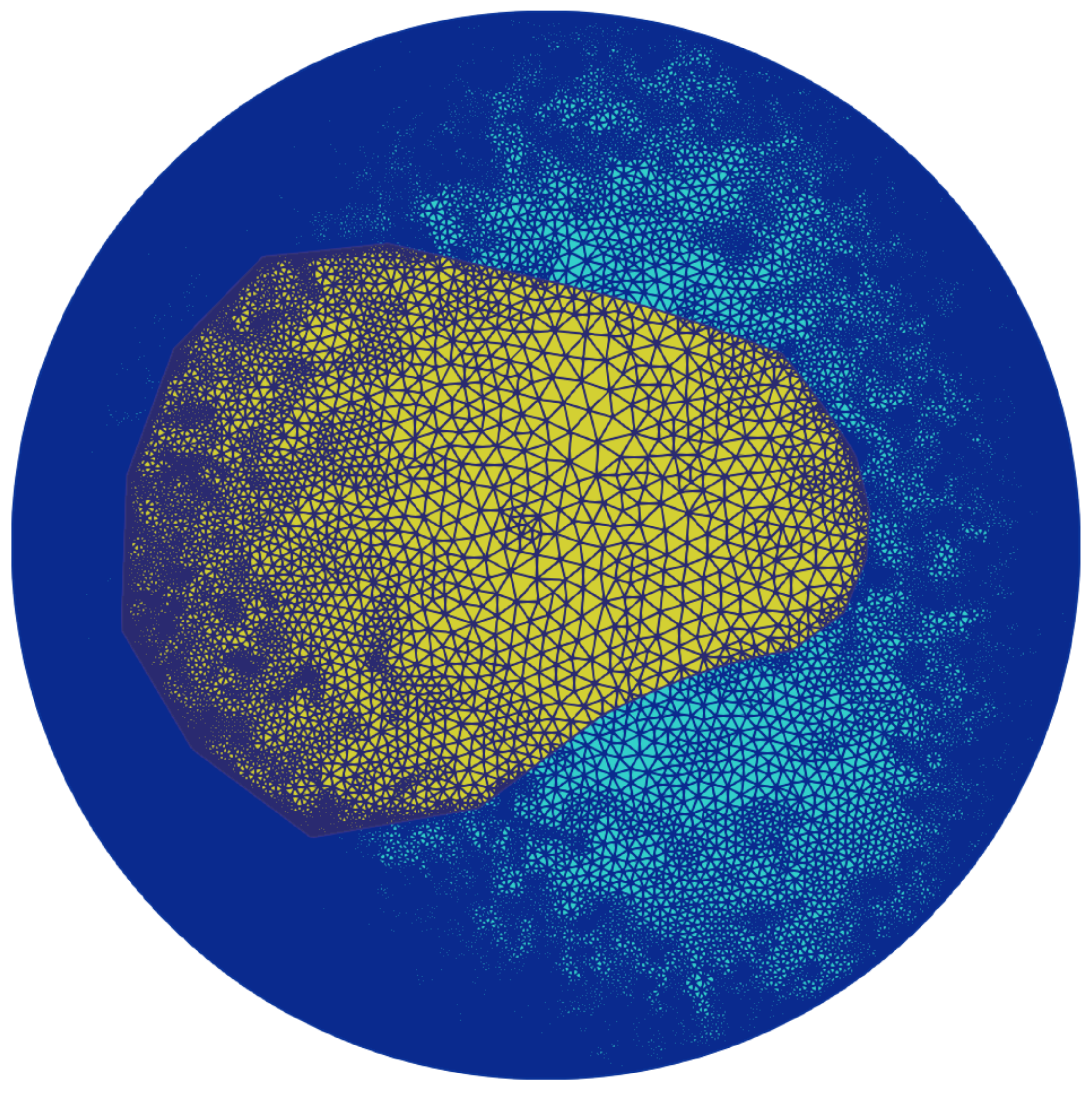}
    \label{fig:meshEllipse30FE}
    }
    
    \subfloat[Iteration \#10.]
    {
    \includegraphics[width=0.35 \columnwidth]{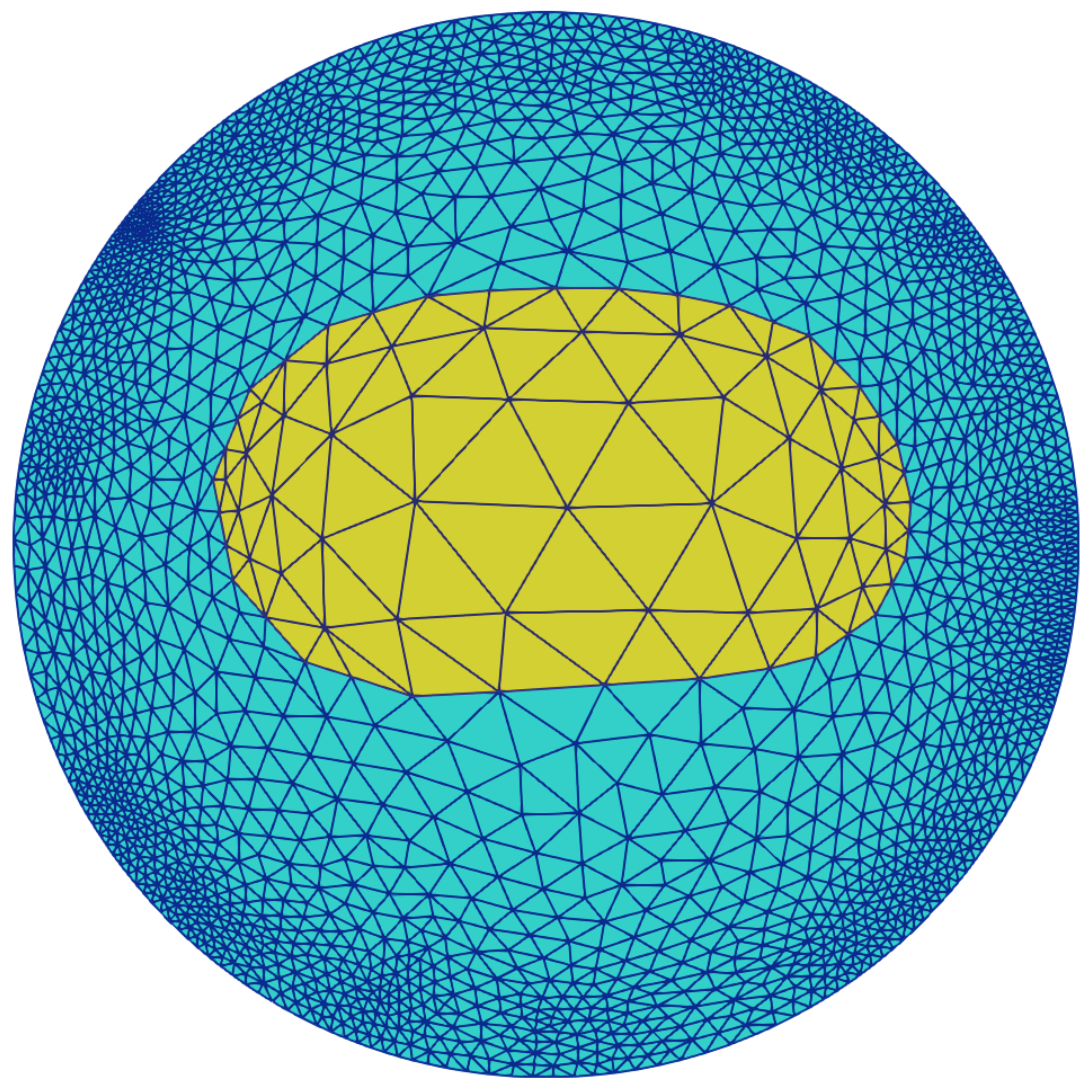}
    \vspace{10pt}
    \label{fig:meshEllipse10DG}
    }
    \hfil
    \subfloat[Iteration \#30.]
    {
    \includegraphics[width=0.35 \columnwidth]{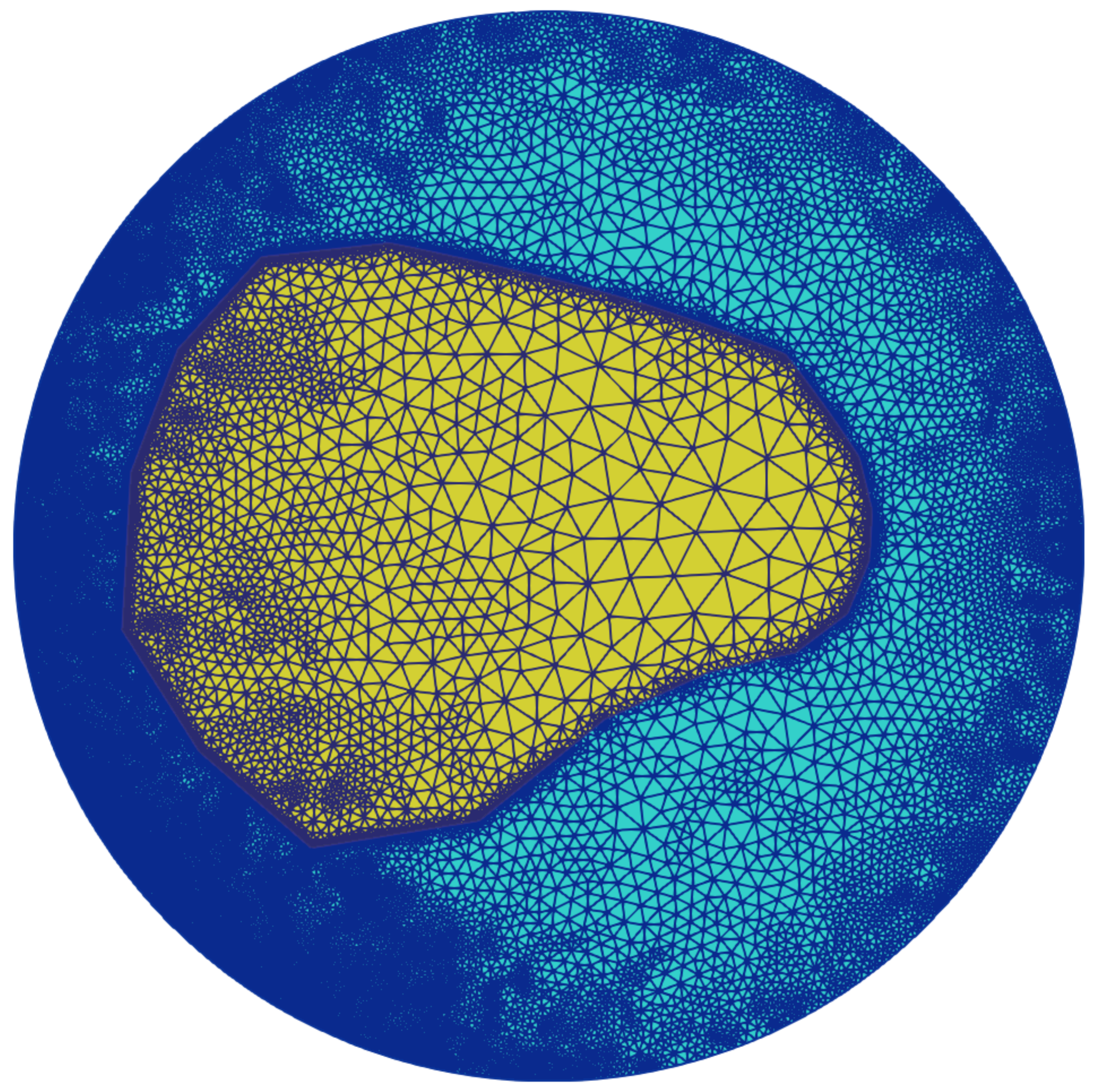}
    \label{fig:meshEllipse30DG}
    }
    \caption{Meshes generated by the certified descent algorithm for the test case in figure \ref{fig:ellipseInterface}.
    Top: conforming finite element. Bottom: discontinuous Galerkin.}
    \label{fig:meshEllipse}
\end{figure}
\begin{table}[htb]
    \centering
    \subfloat[Conforming finite element.]
    {
    \centering
    \begin{tabular}[hbt]{| c | c || l | l |}
    \hline
    Iteration & $\# \mathcal{T}_h$ & $\langle d_h J(\Omega),\boldsymbol\theta^h \rangle$ & $\overline{E}$ \\
    \hline & & & 
    \\ [-1em] \hline
    1 & 3366 & $-2.29 \cdot 10^{-4}$ & $1.79 \cdot 10^{-4}$ \\
    \hline
    10 & 8312 & $-7.63 \cdot 10^{-5}$ & $5.49 \cdot 10 ^{-5}$ \\
    \hline
    20 & 7893 & $-1.29 \cdot 10^{-4}$  & $1.03 \cdot 10 ^{-4}$ \\
    \hline
    30 & 227847 & $-6.16 \cdot 10^{-6}$ & $6.15 \cdot 10^{-6}$ \\
    \hline
    31 & 980555 & $-3.62 \cdot 10^{-6}$ & $3.60 \cdot 10^{-6}$ \\
    \hline
    \end{tabular}
    \label{tab:ellipseConvFE}
    }
    \hfil
    \subfloat[Discontinuous Galerkin.]
    {
    \centering
    \begin{tabular}[hbt]{| c | c || l | l |}
    \hline
    Iteration & $\# \mathcal{T}_h$ & $\langle d_h J(\Omega),\boldsymbol\theta^h \rangle$ & $\overline{E}$ \\
    \hline & & & 
    \\ [-1em] \hline
    1 & 6189 & $-2.21 \cdot 10^{-4}$ & $1.34 \cdot 10^{-4}$ \\
    \hline
    10 & 4868 & $-7.80 \cdot 10^{-5}$ & $4.64 \cdot 10^{-5}$ \\
    \hline
    20 & 4842 & $-1.62 \cdot 10^{-4}$ & $1.30 \cdot 10^{-4}$ \\
    \hline
    30 & 52595 & $-4.71 \cdot 10^{-6}$ & $4.60 \cdot 10^{-6}$ \\
    \hline
    33 & 320137 & $-2.54 \cdot 10^{-7}$ & $2.25 \cdot 10^{-7}$ \\
    \hline
    \end{tabular}
    \label{tab:ellipseConvDG}
    }
\caption{Test case in figure \ref{fig:ellipseInterface} using (a) conforming finite element 
and (b) discontinuous Galerkin. 
Approximated shape gradient and goal-oriented estimator for different meshes.}
\label{tab:ellipseConv}
\end{table}
\\
Though both the version of the CDA based on conforming finite element and the one relying on discontinuous Galerkin are able to certify the descent direction at the beginning of the algorithm, the situation changes after few tens of iterations. 
In particular, the SWIP-dG formulation allows the computation of inexpensive and precise bounds of the error in the shape gradient, whereas using conforming finite element the computational cost rapidly becomes enormous making the certification procedure unfeasible (Table \ref{tab:ellipseConv}).
The previous observation is also qualitatively confirmed by the meshes in figure \ref{fig:meshEllipse}. 
As a matter of fact, despite being similar near the interface on the left-hand side of the domain, the two meshes greatly differ in the right-hand portion and inside the inclusion: within these regions, the mesh generated by the CDA using the SWIP-dG approximation features a lower number of elements, mainly concentrated near the interface where the discontinuity of $k_\Omega$ is located.

\subsection{The case of two inclusions featuring multiple boundary measurements}
\label{ref:2incl}

In this section, we present a more involved test case in which the domain $\mathcal{D}$ features two non-connected inclusions. 
As previously stated, we assume that the number of inclusions is set \emph{a priori} and we restrict to the case of a two-valued conductivity parameter, that is we distinguish a value $k_E$ for the background and a value $k_I$ valid inside both the inclusions.

It is well-known in the literature that multiple boundary measurements are required to retrieve a correct approximation of the inclusion in electrical impedance tomography.
In this section, we consider $D=10$ measurements such that $\forall j=0,\ldots,D-1$
$$
g_j(x,y) = (x+ a_j y)^{b_j} a_j^{c_j} \quad , \quad a_j=1+0.1j \quad ,  \quad 
b_j=\frac{j+1}{2} \quad , \quad c_j=j - 2 \left\lfloor \frac{j}{2} \right\rfloor .
$$
\begin{figure}[htbp]
    \subfloat[Reconstructed interface.]
    {
    \includegraphics[width=0.28 \columnwidth]{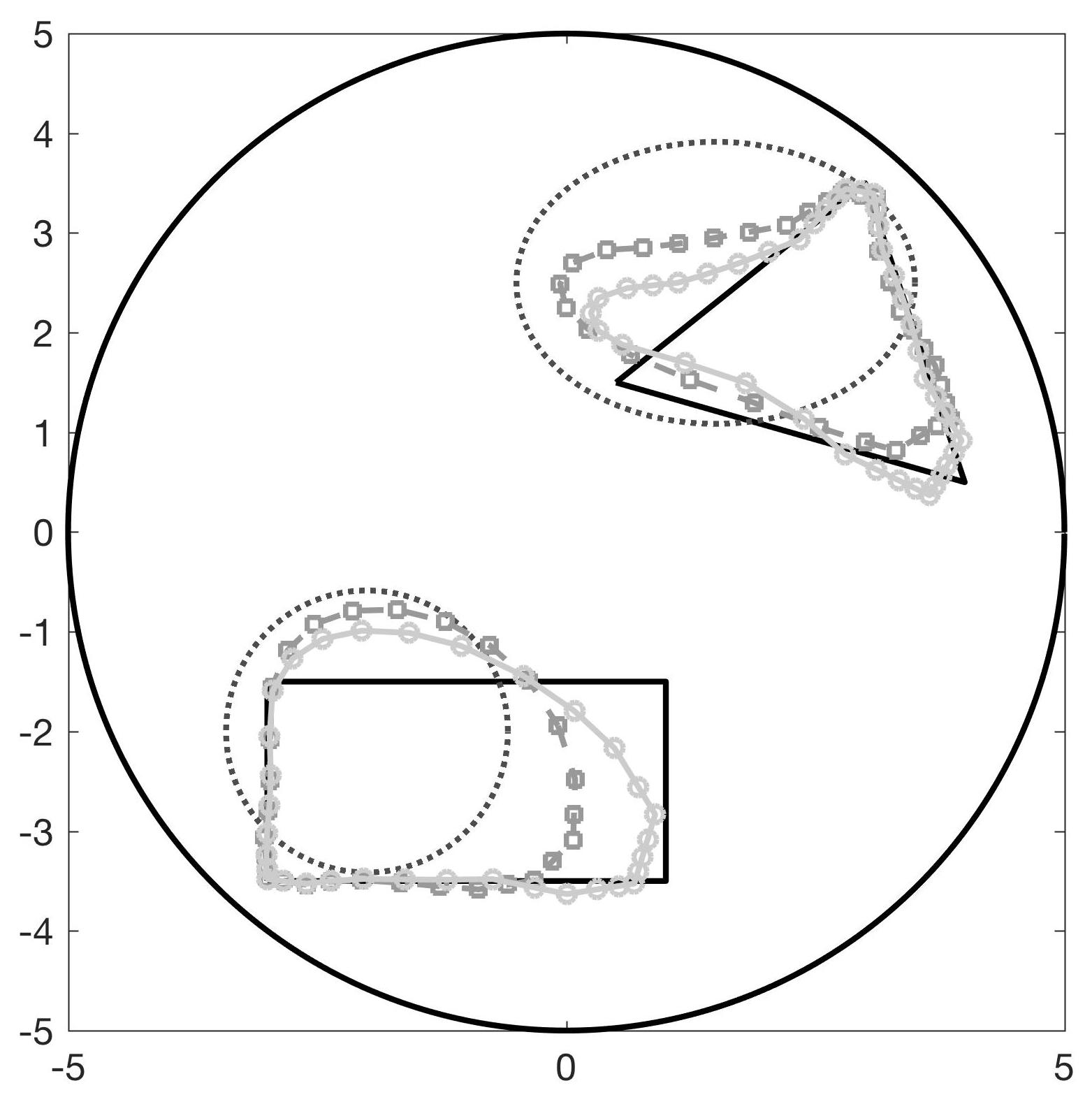}
    \vspace{10pt}
    \label{fig:discInterface}
    }
    \hfil
    \subfloat[Objective functional.]
    {
    \includegraphics[width=0.33 \columnwidth]{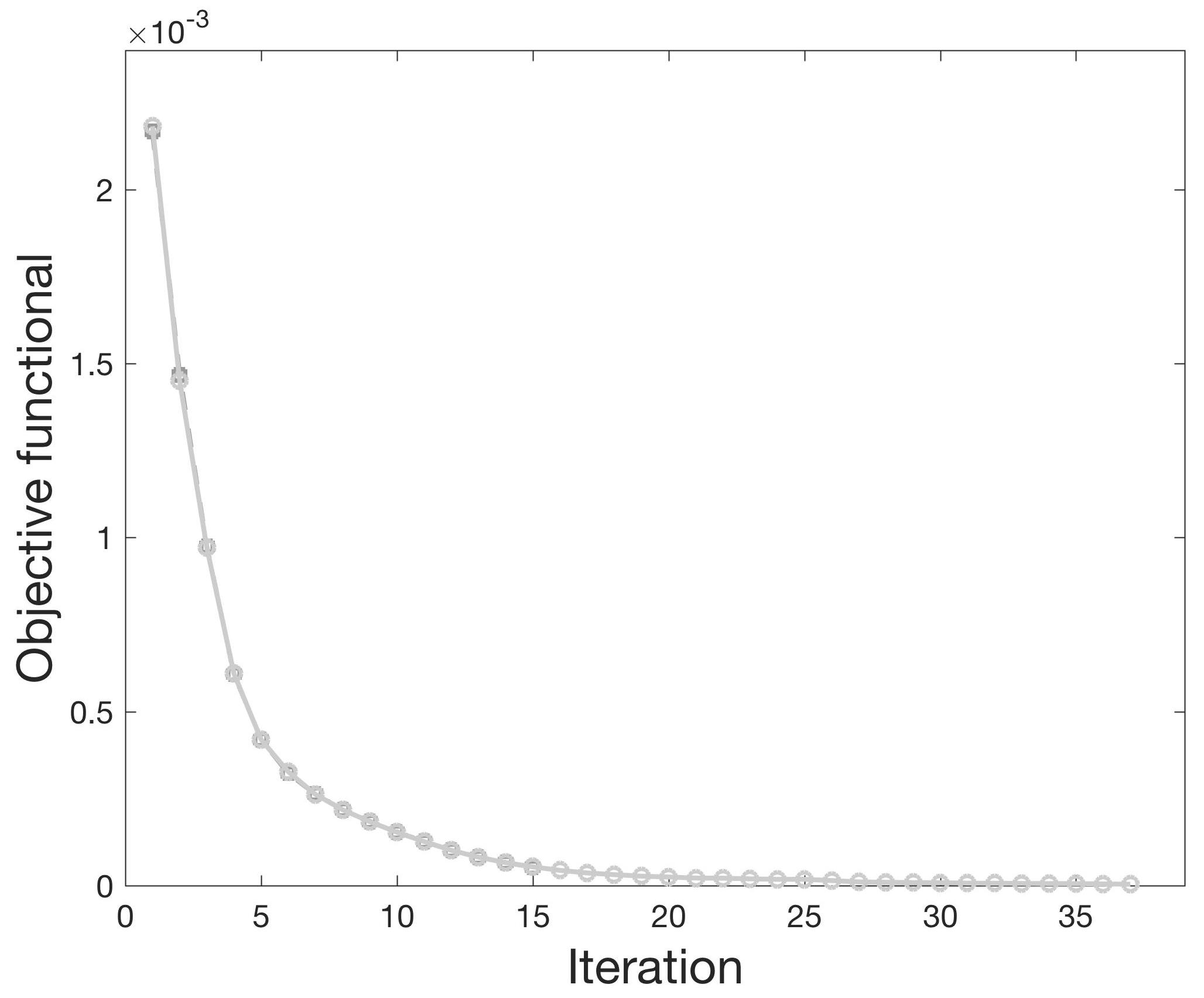}
    \label{fig:discObj}
    }
    \hfil
    \subfloat[Degrees of freedom.]
    {
    \includegraphics[width=0.33 \columnwidth]{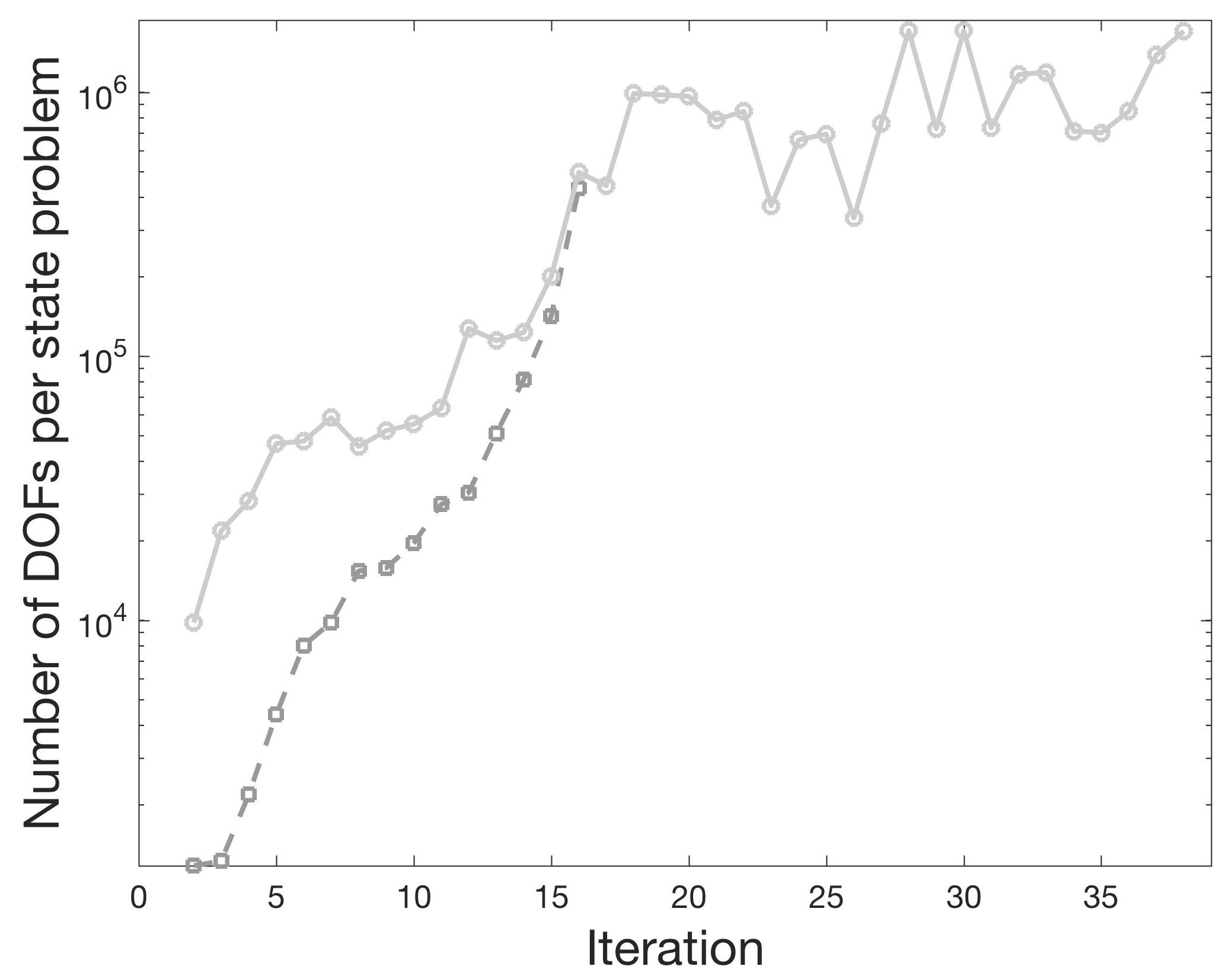}
    \label{fig:discDOF}
    }
    \caption{Certified descent algorithm for the identification of two inclusions. 
    (a) Initial configuration (dotted black), target inclusion (solid black) and 
    reconstructed interface. (b) Evolution of the objective functional. 
    (c) Number of degrees of freedom. 
    Inversion performed using conforming finite element (dark gray squares) and 
    discontinuous Galerkin (light gray circles).}
    \label{fig:disc}
\end{figure}
\begin{figure}[htbp]
	\centering
    \subfloat[Iteration \#10.]
    {
    \includegraphics[width=0.35 \columnwidth]{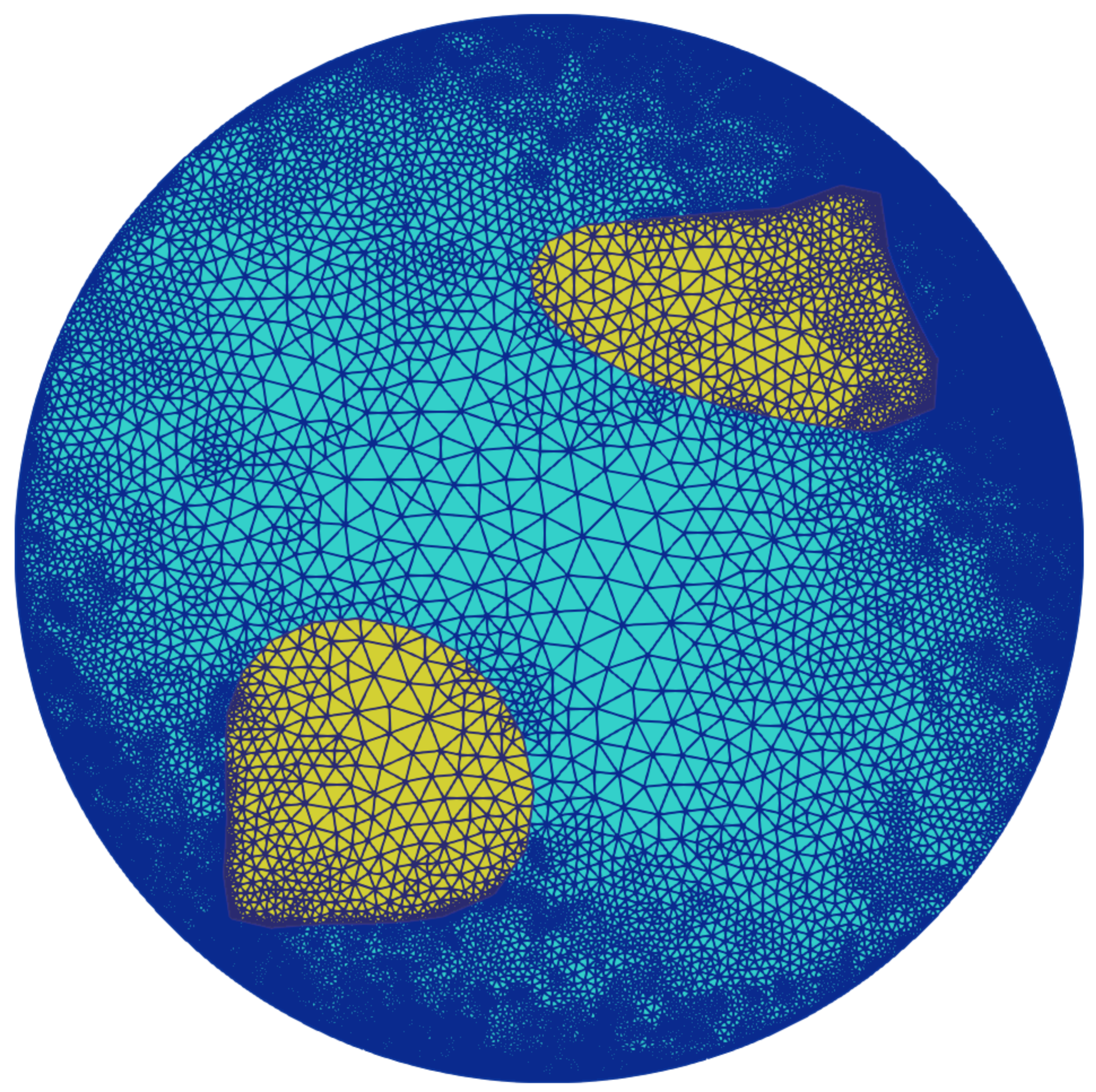}
    \vspace{10pt}
    \label{fig:meshDisc10FE}
    }
    \hfil
    \subfloat[Iteration \#14.]
    {
    \includegraphics[width=0.35 \columnwidth]{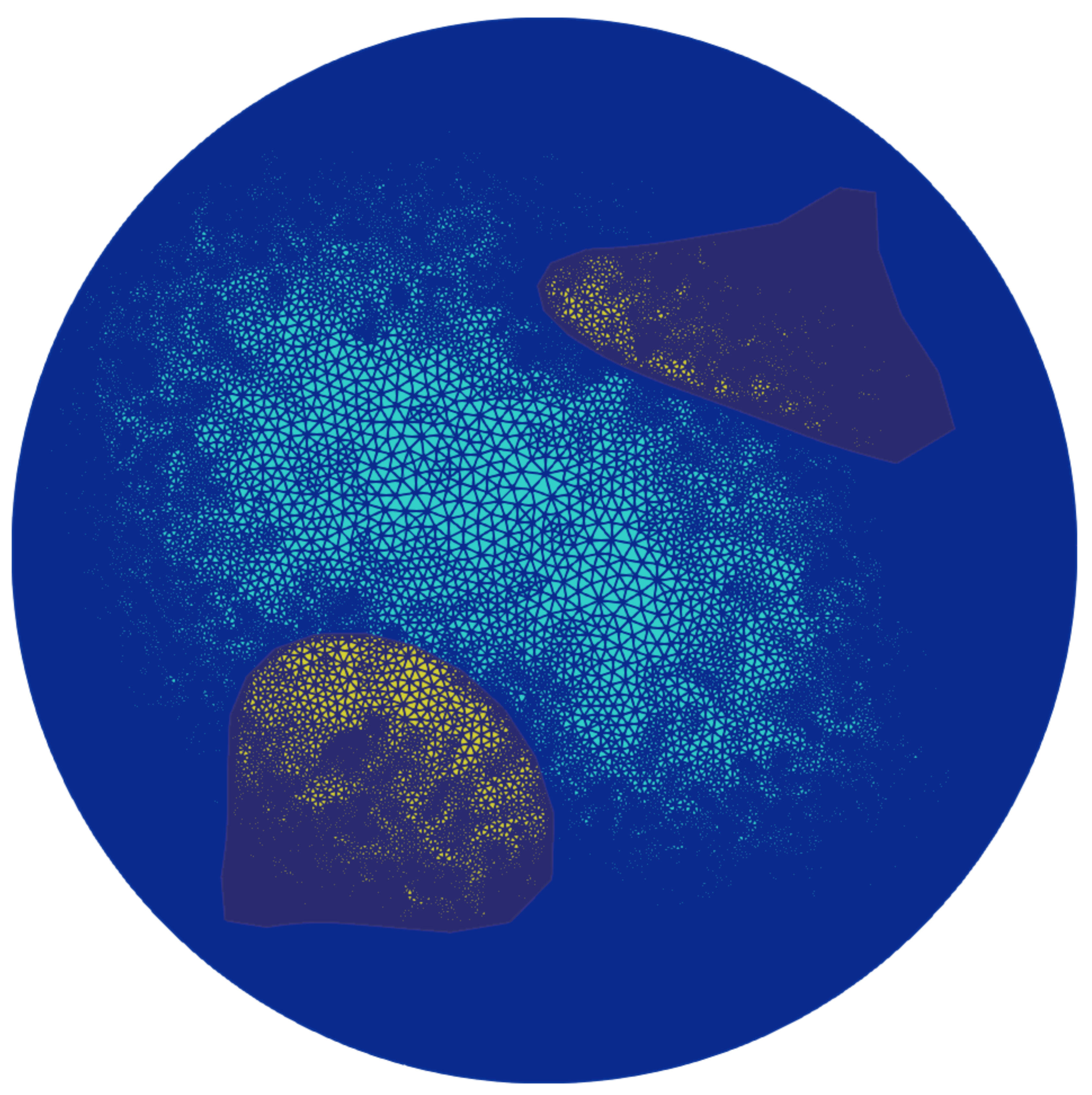}
    \label{fig:meshDisc14FE}
    }
    
    \subfloat[Iteration \#10.]
    {
    \includegraphics[width=0.35 \columnwidth]{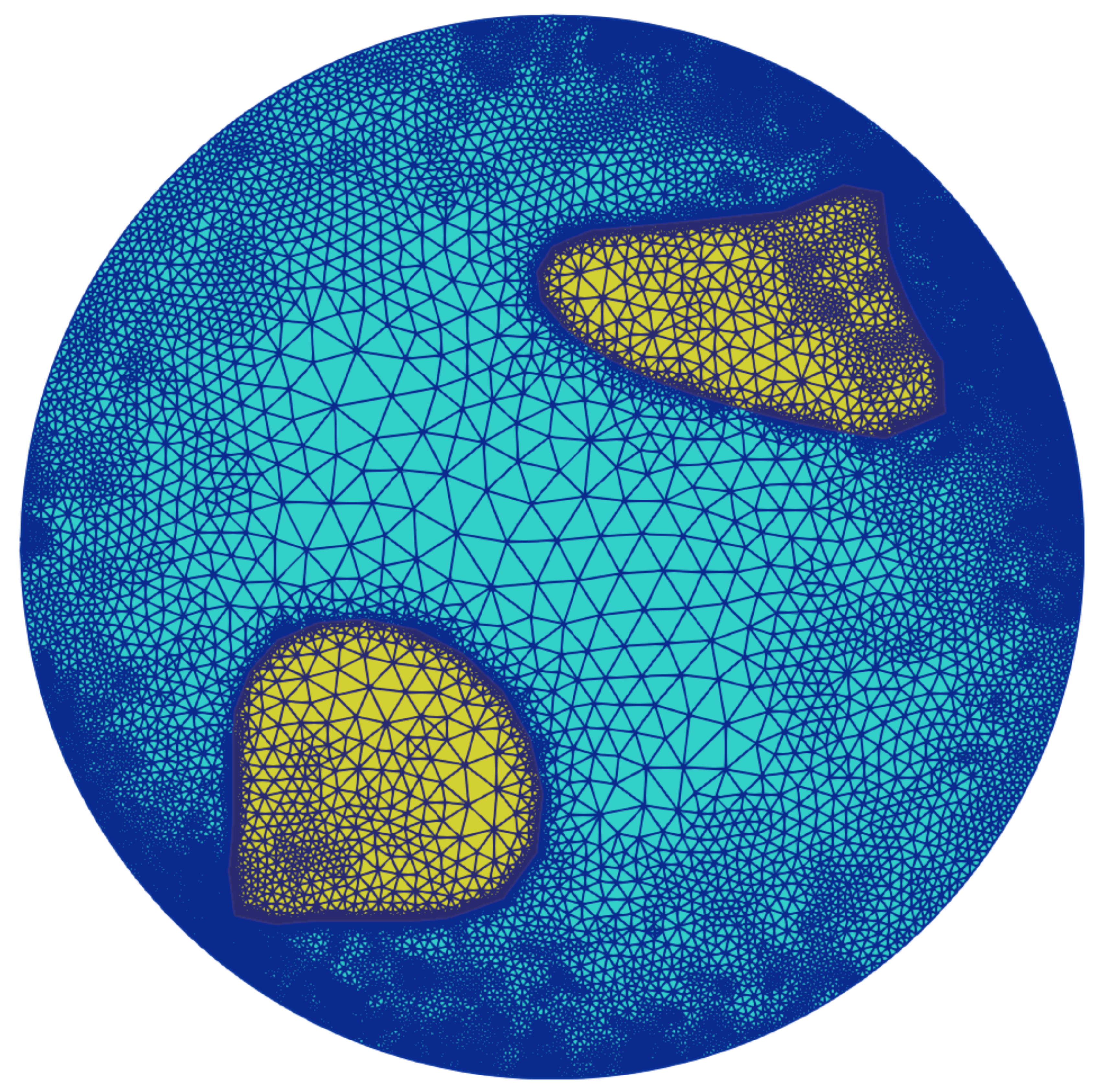}
    \vspace{10pt}
    \label{fig:meshDisc10DG}
    }
    \hfil
    \subfloat[Iteration \#30.]
    {
    \includegraphics[width=0.35 \columnwidth]{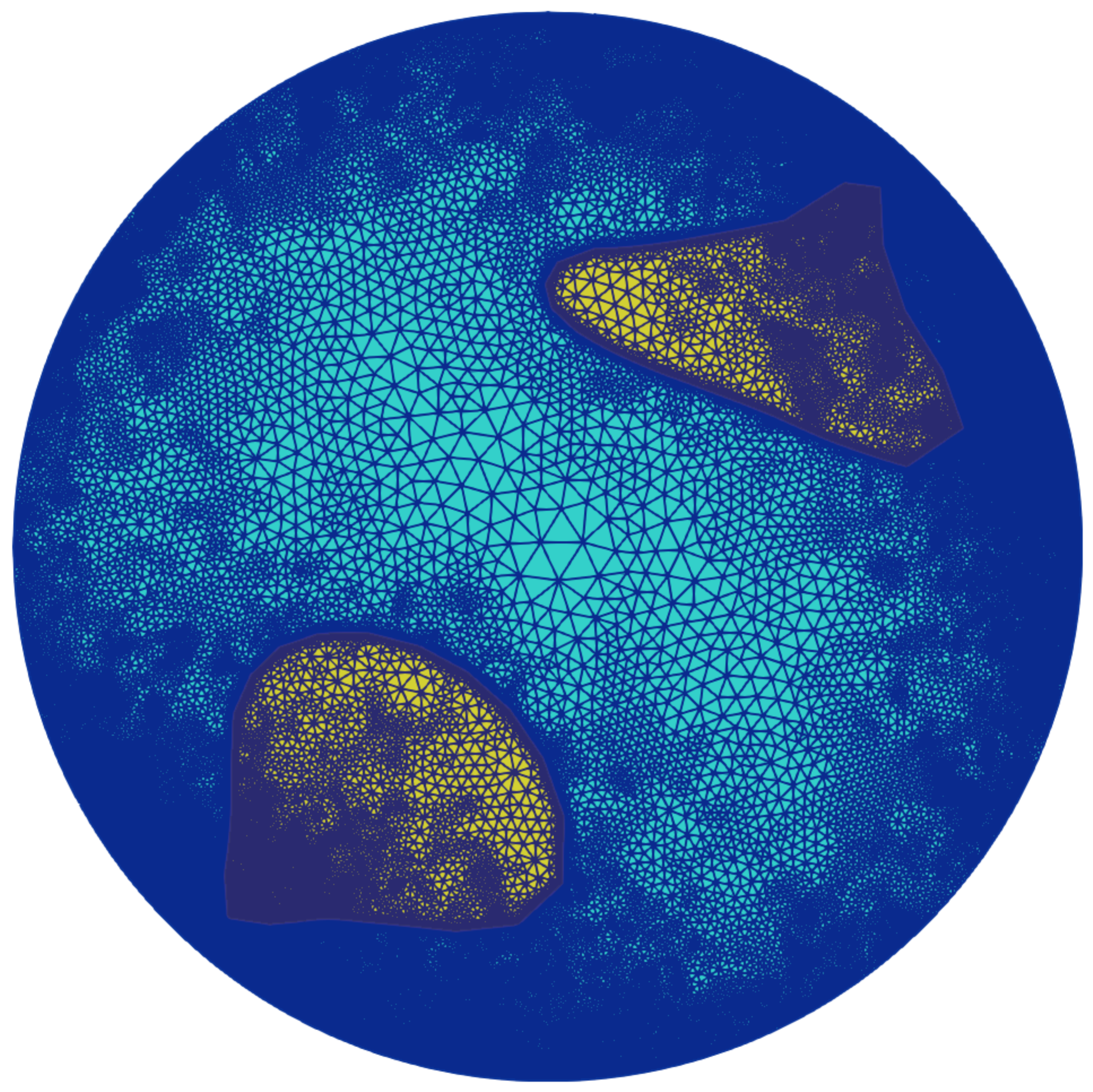}
    \label{fig:meshDisc30DG}
    }
    \caption{Meshes generated by the certified descent algorithm for the test case in figure \ref{fig:discInterface}.
    Top: conforming finite element. Bottom: discontinuous Galerkin.}
    \label{fig:meshDisc}
\end{figure}
\begin{table}[htb]
    \centering
    \subfloat[Conforming finite element.]
    {
    \centering
    \begin{tabular}[hbt]{| c | c || l | l |}
    \hline
    Iteration & $\# \mathcal{T}_h$ & $\langle d_h J(\Omega),\boldsymbol\theta^h \rangle$ & $\overline{E}$ \\
    \hline & & & 
    \\ [-1em] \hline
    1 & 2221 & $-1.62 \cdot 10^{-3}$ & $1.59 \cdot 10^{-3}$ \\
    \hline
    10 & 56487 & $-1.09 \cdot 10^{-5}$ & $1.04 \cdot 10^{-5}$ \\
    \hline
    14 & 852782 & $-3.36 \cdot 10^{-6}$ & $3.21 \cdot 10^{-6}$ \\
    \hline
    \end{tabular}
    \label{tab:discConvFE}
    }
    \hfil
    \subfloat[Discontinuous Galerkin.]
    {
    \centering
    \begin{tabular}[hbt]{| c | c || l | l |}
    \hline
    Iteration & $\# \mathcal{T}_h$ & $\langle d_h J(\Omega),\boldsymbol\theta^h \rangle$ & $\overline{E}$ \\
    \hline & & & 
    \\ [-1em] \hline
    1 & 7282 & $-1.59 \cdot 10^{-3}$ & $1.20 \cdot 10^{-3}$ \\
    \hline
    10 & 42564 & $-1.10 \cdot 10^{-5}$ & $7.27 \cdot 10^{-6}$ \\
    \hline
    20 & 282718 & $-3.82 \cdot 10^{-6}$ & $2.61 \cdot 10^{-6}$ \\
    \hline
    30 & 389571 & $-1.40 \cdot 10^{-6}$ & $8.74 \cdot 10^{-7}$ \\
    \hline
    36 & 568548 & $-1.92 \cdot 10^{-7}$ & $1.87 \cdot 10^{-7}$ \\
    \hline    
    \end{tabular}
    \label{tab:discConvDG}
    }
\caption{Test case in figure \ref{fig:discInterface} using (a) conforming finite element 
and (b) discontinuous Galerkin. 
Approximated shape gradient and goal-oriented estimator for different meshes.}
\label{tab:discConv}
\end{table}
\\
As previously remarked, the certified descent algorithm is able to identify the portions of the interfaces that lie near the external boundary $\partial\mathcal{D}$ whereas the inner parts suffer from a poor reconstruction (Fig. \ref{fig:discInterface}).
Moreover, also in this case after few tens of iterations, the certification procedure requires a huge number of degrees of freedom to identify a genuine descent direction for the objective functional $J(\Omega)$ (Fig. \ref{fig:discDOF}).
Both the inability of the method to reconstruct the interface far from the external boundary and the rapidly increasing number of degrees of freedom required to certify the descent direction clearly testifies the limitations of classical gradient-based approaches when dealing with electrical impedance tomography. \\
Nevertheless, this new variant of the certified descent algorithm proves to be able to certify the descent direction in order to construct a minimizing sequence of shapes for which the objective functional is monotonically decreasing (Fig. \ref{fig:discObj}). 
Moreover, the quantitative information carried by the error bound $\overline{E}$ allows to derive a reliable stopping criterion that automatizes the overall optimization procedure.
\begin{rmrk}
The tables presented in this section show that the strategy based on a conforming finite element discretization rapidly requires a huge number of mesh elements to perform the certification of the descent direction. 
In figure \ref{fig:meshDisc14FE}, we report the mesh after $14$ iterations of the CDA and we remark that the algorithm is stopped due to its excessive computational cost. 
On the contrary, the version of the CDA relying on the discontinuous Galerkin approximation is able to certify the descent direction using a coarser mesh and refining it near the external boundary and in the region where the inclusion is located (Fig. \ref{fig:meshDisc10DG}-\ref{fig:meshDisc30DG}).
However, it is important to recall that the discontinuous Galerkin formulations feature a higher number of degrees of freedom per mesh element, making the overall dimensions of the optimization problems comparable.
Nevertheless, from a practical point of view the computation of the error bound $\overline{E}$ in the framework of conforming finite element relies on the solution of a number of local subproblems on patches of elements equal to the number of vertices of the triangulation $\mathcal{T}_h$. 
On the contrary, the discontinuous Galerkin discretization is locally conservative and yields to a straightforward technique to construct the equilibrated fluxes based on an inexpensive local post-process of the solutions of the state and adjoint problems. 
Thus, both approaches result valid and present an improvement of the original certified descent algorithm introduced in \cite{giacomini:hal-01201914} which required the solution of additional global problems to perform the certification procedure.
Nevertheless, the computational cost of the version based on the discontinuous Galerkin formulation appears more competitive, especially in view of future developments focusing on vectorial and three-dimensional problems.
\end{rmrk}

\section{Conclusion}
\label{ref:conclusion}

As already pointed out in \cite{giacomini:hal-01201914}, the certified descent algorithm (CDA) for shape optimization uses the quantitative information of the goal-oriented estimator to construct on the one hand a sequence of shapes leading to a monotonically decreasing evolution of the objective functional and on the other hand a novel stopping criterion for the overall optimization procedure. 
The main drawback of the aforementioned strategy was the high computational cost due to the solution of additional global variational problems to estimate the error in the shape gradient via the complementary energy principle. \\
In this work, we proposed an improved version of the CDA which uses solely local quantities to certify that the computed direction is a genuine descent direction for the functional under analysis. 
In particular, we derived a goal-oriented estimator of the error in the shape gradient via the construction of equilibrated fluxes. 
This approach has been developed for both conforming finite element and discontinuous Galerkin discretizations and has been tested on the scalar inverse problem of electrical impedance tomography.
On the one hand, using a conforming finite element discretization, the number of degrees of freedom required by the approximation of the state and adjoint problems is small but the construction of the equilibrated fluxes for the estimator of the error in the shape gradient requires the solution of local subproblems defined on patches of elements whose number is equal to the number of vertices of the triangulation.
On the other hand, though the discontinuous Galerkin formulation of the problems features a higher number of degrees of freedom per mesh element, the computation of the error estimator based on the 
equilibrated fluxes approach is straightforward via a post-process which involves solely local quantities.
Both strategies proved to be valid but the bounds provided by the discontinuous Galerkin approach appeared more precise and computationally less expensive.

Ongoing investigations focus on the application of the certified descent algorithm to the vectorial problem of shape optimization in linear elasticity.

\appendix

\section{Weak imposition of the essential boundary conditions}
\label{ref:essential}

We present a formal derivation of the variational formulation of an elliptic problem featuring weakly-imposed Dirichlet boundary conditions.
The idea of this approach dates back to the classical paper by Nitsche \cite{Nitsche1971} and has been extensively studied in recent years by several authors (cf. e.g. \cite{NME:NME4815} and references therein).
We recall that the solution of a boundary value problem may be interpreted as an optimization problem. Let us introduce the Lagrangian functional associated with the state problem \eqref{eq:statePB} featuring Dirichlet boundary conditions:
\begin{equation}
\Lambda(w,\lambda) = \frac{1}{2} \int_\mathcal{D}{\Big(k_\Omega | \nabla w |^2 + |w|^2 \Big) d\mathbf{x}} - \int_{\partial\mathcal{D}}{\lambda(w-U_D) ds} .
\label{eq:lagrangianWeak} 
\end{equation}
The solution of the aforementioned boundary value problem is equivalent to the following min-max problem: 
$$
\min_{w \in H^1(\mathcal{D})} \max_{\lambda \in H^{-\frac{1}{2}}(\mathcal{D})} \Lambda(w,\lambda) .
$$
The first-order optimality conditions for \eqref{eq:lagrangianWeak} read as
\begin{align*}
\left\{
\begin{aligned}
& \int_\mathcal{D}{\Big( k_\Omega \nabla w \cdot \nabla \delta w + w \delta w \Big) d\mathbf{x}} - \int_{\partial\mathcal{D}}{\lambda \delta w \ ds} = 0 ,\\
& \int_{\partial\mathcal{D}}{(w-U_D) \delta\lambda \ ds} = 0 .
\end{aligned}
\right.
\label{eq:OptimalityEssential}
\end{align*}
From the second condition, we retrieve the Dirichlet boundary condition on $\partial\mathcal{D}$. Integrating by parts the first condition and owing to the strong form of the problem, we obtain
\begin{equation*}
\int_{\partial\mathcal{D}}{(k_\Omega \nabla w \cdot \mathbf{n} - \lambda) \delta w \ ds} = 0 .
\label{eq:lagrangeMult}
\end{equation*}
By plugging $\lambda = k_\Omega \nabla w \cdot \mathbf{n} \ \text{on} \ \partial\mathcal{D}$ into \eqref{eq:lagrangianWeak} we may now derive the following dual variational problem by seeking $w \in H^1(\mathcal{D})$ such that $\forall \delta w \in H^1(\mathcal{D})$
\begin{equation}
\begin{aligned}
\int_\mathcal{D}{\Big( k_\Omega \nabla w \cdot \nabla \delta w + w \delta w \Big) d\mathbf{x}} 
& - \int_{\partial\mathcal{D}}{\Big( k_\Omega \nabla w \cdot \mathbf{n} \delta w + w k_\Omega \nabla \delta w \cdot \mathbf{n} \Big) ds} \\
& \hspace{48pt} = - \int_{\partial\mathcal{D}}{U_D k_\Omega \nabla \delta w \cdot \mathbf{n} \ ds} .
\end{aligned}
\label{eq:weakInstable}
\end{equation}
We remark that the bilinear form on the left-hand side of \eqref{eq:weakInstable} is not coercive thus we cannot establish the well-posedness of this problem.
To bypass this issue, we consider the following augmented Lagrangian functional and we construct the corresponding dual variational formulation for the problem under analysis:
\begin{equation*}
\Upsilon(w,\lambda,\gamma) = \Lambda(w,\lambda) + \frac{1}{2} \int_{\partial\mathcal{D}}{\gamma(w-U_D)^2 ds} .
\label{eq:augLagrangianWeak} 
\end{equation*}
Following the same procedure used to derive \eqref{eq:weakInstable}, we seek $w \in H^1(\mathcal{D})$ such that $\forall \delta w \in H^1(\mathcal{D})$
\begin{equation}
\begin{aligned}
\int_\mathcal{D}{\Big( k_\Omega \nabla w \cdot \nabla \delta w + w \delta w \Big) d\mathbf{x}} 
- & \int_{\partial\mathcal{D}}{\Big( k_\Omega \nabla w \cdot \mathbf{n} \delta w + w k_\Omega \nabla \delta w \cdot \mathbf{n} \Big) ds} 
+ \int_{\partial\mathcal{D}}{\gamma w \delta w \ ds} \\
& \hspace{73pt} = \int_{\partial\mathcal{D}}{U_D \Big(\gamma \delta w - k_\Omega \nabla \delta w \cdot \mathbf{n} \Big) ds} .
\end{aligned}
\label{eq:weakSable}
\end{equation}
It is straightforward to observe that the bilinear form on the left-hand side of \eqref{eq:weakSable} is coercive owing a \emph{sufficiently large} value of $\gamma$ is chosen.

\subsubsection*{Acknowledgements}
The author expresses his sincere gratitude to Alexandre Ern for the useful advices and to Olivier Pantz for many fruitful discussions and for carefully reading the manuscript.
The author wishes to thank the anonymous reviewers for their comments that helped to greatly improve the manuscript.
Part of this work has been developed during a stay of the author at the Laboratoire J.A. Dieudonn\'e at Universit\'e de Nice-Sophia Antipolis whose support is kindly acknowledged.

\bibliographystyle{abbrv}
\bibliography{./newBibliography}   

\end{document}